\documentclass{amsart}

\usepackage{amssymb,mathrsfs,amscd,graphicx}
\allowdisplaybreaks
\usepackage{array}
\usepackage{multirow}

\usepackage{mathtools} 

\usepackage[nocompress]{cite}  
\usepackage{enumerate}

\usepackage{tikz}
\usetikzlibrary{matrix,arrows,cd}

 \usepackage{hyperref}
\usepackage{letltxmacro}
\LetLtxMacro{\oldsqrt}{\sqrt}
\renewcommand{\sqrt}[2][]{\,\oldsqrt[#1]{#2}\,}

\makeatletter
\def\@tocline#1#2#3#4#5#6#7{\relax
  \ifnum #1>\c@tocdepth 
  \else
    \par \addpenalty\@secpenalty\addvspace{#2}%
    \begingroup \hyphenpenalty\@M
    \@ifempty{#4}{%
      \@tempdima\csname r@tocindent\number#1\endcsname\relax
    }{%
      \@tempdima#4\relax
    }%
    \parindent\z@ \leftskip#3\relax \advance\leftskip\@tempdima\relax
    \rightskip\@pnumwidth plus4em \parfillskip-\@pnumwidth
    #5\leavevmode\hskip-\@tempdima
      \ifcase #1
       \or\or \hskip 1em \or \hskip 2em \else \hskip 3em \fi%
      #6\nobreak\relax
    \dotfill\hbox to\@pnumwidth{\@tocpagenum{#7}}\par
    \nobreak
    \endgroup
  \fi}
\makeatother

\usepackage{bbm}
\def\bmu{\boldsymbol \mu}

\usepackage{stmaryrd}
\newcommand{\dbr}[1]{\left\llbracket #1 \right\rrbracket}
\newcommand{\abs}[1]{\lvert #1 \rvert}

\newcommand{\ol}{\overline}
\newcommand{\zmod}[1]{\mathbb{Z}/ #1 \mathbb{Z}}

\newcommand{\dangle}[1]{\left\langle #1 \right\rangle}
\newcommand{\qalg}[3]{\left(\frac{#1, #2}{#3}\right)}

\newcommand{\Lsymb}[2]{\genfrac{(}{)}{}{}{#1}{#2}}    

\newcommand{\<}{\langle}
\renewcommand{\>}{\rangle}
\newcommand{\bs}{{\pmb{\star}}} 

\DeclareMathSymbol{\twoheadrightarrow} {\mathrel}{AMSa}{"10}

\DeclareMathOperator{\Cl}{Cl}
\DeclareMathOperator{\Tp}{Tp}
\DeclareMathOperator{\Gor}{Gor}

\DeclareMathOperator{\ord}{ord}

\DeclareMathOperator{\fchar}{char}

\DeclareMathOperator{\Pic}{Pic}
\DeclareMathOperator{\Nr}{Nr}

\DeclareMathOperator{\Aut}{Aut}
\DeclareMathOperator{\Inn}{Inn}
\DeclareMathOperator{\End}{End}
\DeclareMathOperator{\Hom}{Hom}
\DeclareMathOperator{\Emb}{Emb}

\DeclareMathOperator{\Mass}{Mass}

\DeclareMathOperator{\Gal}{Gal}

\DeclareMathOperator{\Tr}{Tr}
\DeclareMathOperator{\Nm}{N}  

\DeclareMathOperator{\SO}{SO}

\DeclareMathOperator{\PGL}{PGL}

\newcommand{\Z}{\mathbb Z}
\newcommand{\Q}{\mathbb Q}
\newcommand{\R}{\mathbb R}

\renewcommand{\H}{\mathbb H}  
\renewcommand{\O}{\mathbb O} 
\newcommand{\F}{\mathbb F}

\makeatletter

\def\greekbolds#1{%
 \@for\next:=#1\do{%
    \def\X##1;{%
     \expandafter\def\csname V##1\endcsname{\boldsymbol{\csname##1\endcsname}}
     }
   \expandafter\X\next;
  }
}

\greekbolds{alpha,beta,iota,gamma,lambda,nu,eta,Gamma,varsigma}

\def\make@bb#1{\expandafter\def
  \csname bb#1\endcsname{{\mathbb{#1}}}\ignorespaces}

\def\make@bbm#1{\expandafter\def
  \csname bb#1\endcsname{{\mathbbm{#1}}}\ignorespaces}

\def\make@bf#1{\expandafter\def\csname bf#1\endcsname{{\bf
      #1}}\ignorespaces} 

\def\make@gr#1{\expandafter\def
  \csname gr#1\endcsname{{\mathfrak{#1}}}\ignorespaces}

\def\make@scr#1{\expandafter\def
  \csname scr#1\endcsname{{\mathscr{#1}}}\ignorespaces}

\def\make@cal#1{\expandafter\def\csname cal#1\endcsname{{\mathcal
      #1}}\ignorespaces} 

\def\do@Letters#1{#1A #1B #1C #1D #1E #1F #1G #1H #1I #1J #1K #1L #1M
                 #1N #1O #1P #1Q #1R #1S #1T #1U #1V #1W #1X #1Y #1Z}
\def\do@letters#1{#1a #1b #1c #1d #1e #1f #1g #1h #1i #1j #1k #1l #1m
                 #1n #1o #1p #1q #1r #1s #1t #1u #1v #1w #1x #1y #1z}
\do@Letters\make@bb   \do@letters\make@bbm
\do@Letters\make@cal  
\do@Letters\make@scr 
\do@Letters\make@bf \do@letters\make@bf   
\do@Letters\make@gr   \do@letters\make@gr
\makeatother

\newcounter{thmcounter}
\numberwithin{thmcounter}{section}  
\newtheorem{thm}[thmcounter]{Theorem}
\newtheorem{lem}[thmcounter]{Lemma}
\newtheorem{def-lem}[thmcounter]{Definition-Lemma}
\newtheorem{cor}[thmcounter]{Corollary}
\newtheorem{prop}[thmcounter]{Proposition}
\theoremstyle{definition}
\newtheorem{defn}[thmcounter]{Definition}

\newtheorem{rem}[thmcounter]{Remark}

\newtheorem{sect}[thmcounter]{}

\numberwithin{equation}{section}
\numberwithin{figure}{section}
\numberwithin{table}{section}
\newtheoremstyle{notitle}  
  {}
  {}
  {\itshape}
  {}
  {}
  {\ }
  {.5em}
  {}
\theoremstyle{notitle}
\newtheorem*{pnt}{}

\title[Unit Groups]{Unit Groups of Maximal Orders in Totally Definite
  Quaternion Algebras over real quadratic fields}

\author{Qun Li}

\address{(Li) School of Mathematics and Statistics, Wuhan University, Luojiashan,
Wuhan, Hubei, 430072, P.R. China.}

\email{qun\_l@whu.edu.cn}

\author{Jiangwei Xue}
\address{(Xue) Collaborative Innovation Centre of Mathematics, 
School of Mathematics and Statistics, Wuhan University, Luojiashan,
Wuhan, Hubei, 430072, P.R. China.}

\address{(Xue) Hubei Key Laboratory of Computational Science (Wuhan
  University), Wuhan, Hubei,  430072, P.R. China.}
\email{xue\_j@whu.edu.cn}

\author{Chia-Fu Yu}
\address{(Yu) Institute of Mathematics,
  Academia Sinica, Astronomy-Mathematics
  Building, No. 1, Sec. 4, Roosevelt Road, Taipei 10617, TAIWAN.}

\address{(Yu) National Center for Theoretical Sciences, 
  Astronomy-Mathematics
  Building, No. 1, Sec. 4, Roosevelt Road, Taipei 10617, TAIWAN.}
\email{chiafu@math.sinica.edu.tw}

\begin{document}
\date{\today}

\begin{abstract}
  We study a form of refined class number formula (resp.~type number
  formula) for maximal orders in totally definite quaternion algebras
  over real quadratic fields, by taking into consideration the
  automorphism groups of right ideal classes (resp.~unit groups of
  maximal orders).  For each finite noncyclic group $G$, we give an
  explicit formula for the number of conjugacy classes of maximal
  orders whose unit groups modulo center are isomorphic to $G$, and write
  down a representative for each conjugacy class. This
  leads to a complete recipe (even explicit formulas in special cases) for
  the refined class number formula for all finite groups.  As an
  application,  we prove the existence of superspecial abelian
 surfaces whose endomorphism algebras coincide with $\Q(\sqrt{p})$ in
all positive characteristic $p\not\equiv 1\pmod{24}$. 






\end{abstract}

\maketitle
\tableofcontents   


\section{Introduction}

Let $F$ be a totally real number field with ring of integers $O_F$,
and $H$ be a \textit{totally definite} quaternion $F$-algebra, that is,
$H\otimes_{F, \sigma}\bbR$ is isomorphic to the Hamilton quaternion
$\bbR$-algebra $\bbH$ for every embedding $\sigma:F\hookrightarrow \bbR$.
Fix a maximal $O_F$-order $\O$ in $H$.  The \emph{class number} $h(\O)$, by
definition, is the cardinality of the finite set $\Cl(\O)$ of right
ideal classes of $\bbO$. It depends only on $H$ and is also denoted by
$h(H)$ and called the class number of $H$. The \emph{type number}
$t(H)$ is the cardinality of the finite set $\Tp(H)$ of all
$H^\times$-conjugacy classes of maximal $O_F$-orders of
$H$. 



For any $O_F$-order $\calO$ in $H$, the quotient group
$\calO^\bs:=\calO^\times/O_F^\times$ of the unit group $\calO^\times$
by $O_F^\times$ is finite by \cite[Theorem~V.1.2]{vigneras} and called the $\emph{reduced unit group}$
of $\calO$. For a (fractional) right ideal $I\subset H$ of $\calO$, the \emph{reduced
  automorphism group} of $I$ is defined to be $\calO_l(I)^\bs$, where
$\calO_l(I):=\{x\in H\mid xI\subseteq I\}$ is the left order of
$I$. The reduced unit group $\calO^\bs$ can easily be regarded as a
finite subgroup of $\SO_3(\R)$ (see
Section~\ref{sect:general-classification-grp}). By the well-known
classification of finite
subgroups of $\SO_3(\R)$, $\calO^\bs$ is isomorphic to either a cyclic
group $C_n$ of order $n$, a dihedral group $D_n$ of order $2n$, or one
of the groups $A_4, S_4$ and $A_5$ (see \cite[Theorem~I.3.6]{vigneras}). For each group $G$ in this list, we
define
\begin{align}
\Cl(\O,G)&:=\{[I]\in \Cl(\O)\mid
         \calO_l(I)^\bs \simeq G\},&  h(H, G):&=\abs{\Cl(\O, G)},\label{eq:105}\\
\Tp(H,G)&:=\{\dbr{\O'}\in \Tp(H)\mid \O'^\bs\simeq G \},&  t(H, G):&=\abs{\Tp(H,G)},\label{eq:106}
\end{align}
where $[I]$ denotes the right ideal class of $I$, and $\dbr{\O'}$ denotes
the $H^\times$-conjugacy class of the maximal order $\O'$.  The
quantity $h(H, G)$ (resp.~$t(H, G)$) can be regarded as a refined
class number (resp.~refined type number) of $H$. The classical proof
of the independence of $h(H)$ on the choice of the maximal order $\O$
can be adapted to show that the same holds for $h(H, G)$ as well (see
\cite[Lemma~I.4.9(2)]{vigneras}). Indeed, any two maximal orders
$\O, \O'$ are \emph{linked}, that is to say, there exists a full
$O_F$-lattice $J\subset H$
such that $\calO_l(J)=\O$ and $\calO_r(J)=\O'$, where $\calO_r(J):=\{x\in H\mid Jx\subseteq J\}$.  The map
$[I]\mapsto [IJ]$ defines a bijection $\Cl(\O)\to \Cl(\O')$. Since
$\calO_l(I)=\calO_l(IJ)$ for each $I$, it follows from the definition
that $h(H, G)$ does not depend on the choice of $\O$.  If $H$ is
clear from the context, then we drop it from the notation and write
$h(G)$ and $t(G)$ instead.

The main tool for studying class numbers and type numbers is Eichler's
trace formula (\cite{eichler:crelle55, Pizer-1976},
cf.~\cite{vigneras}).  This has been used to study various arithmetic
problems concerning totally definite quaternion algebras, including the
analogous Gauss problem and the cancellation property by several people
\cite{vigneras-crelle-1976, Smertnig-crelle-2015,
  Hallouin-Maire-crelle-2006, Brzezinski:class_no_one,
  Kirschmer-Voight, Kirschmer-Lorch}.
Brzezinski
\cite{Brzezinski:class_no_one} obtains a complete list of all
orders (including non-Gorenstein orders) in definite quaternion 
$\Q$-algebras with class number one.   
Kirschmer and Voight \cite{Kirschmer-Voight} determine all Eichler
$O_F$-orders with class number $\le 2$. Kirschmer and Lorch \cite{Kirschmer-Lorch}
determine all
Eichler $O_F$-orders with type number $\le 2$. 

Vign\'eras \cite[Theorem 3.1]{vigneras:ens} gives an explicit formula
for $h(\bbO)$ (including Eichler orders $\bbO$) when $F$ is a real
quadratic field. Explicit formulas for type numbers, however,
are comparably unknown.
In \cite{Pizer1973} Pizer proves a general formula for type numbers, 
and uses this to deduce an explicit type number formula for 
Eichler orders in an arbitrary definite quaternion $\Q$-algebra \cite{Pizer1973}.
So far there is no known \emph{explicit} formula for
$t(H)$ in the literature when $F$ is an \emph{arbitrary} quadratic field. As far as the authors are aware, 
the only known
cases are due to Kitaoka \cite[1.12 and 3.11]{kitaoka:nmj1973} and 
Ponomarev \cite[Theorem part (c), p.~102]{ponomarev:aa1981} 
when $F=\Q(\sqrt{p})$ with a prime number $p$, and 
$H=H_{\infty_1,\infty_2}$ is the totally definite quaternion $F$-algebra unramified at every finite place of $F$. 
Under the same hypothesis on $H$, 
we prove the following result for more general totally real fields in
\cite[Section~3]{xue-yu:type_no}.  The idea of the proof is sketched in
Remark~\ref{rem:free-act-odd-hF} for the convenience of the readers. 


\begin{prop}\label{1.1}
  Let $H$ be a totally definite quaternion algebra over a totally real
  number field $F$ of even degree over $\Q$. Assume  that $H$ is unramified at all finite places
  of $F$ and that $h(F)$ is odd. Then for any finite group $G$
  one has $h(H, G)=h(F) t(H, G)$. In particular, the equality $h(H)=h(F)
  t(H)$ holds. 
\end{prop}

Thanks to Vign\'eras's explicit formula \cite[Theorem 3.1]{vigneras:ens}, 
we obtain an explicit formula for $t(H_{\infty_1,\infty_2})$ when $F$ has odd class number. 
For a complete list of quadratic fields with
odd class numbers, see \cite[Corollary~18.4]{Conner-Hurrelbrink}. In
particular, $h(\Q(\sqrt{p}))$ is odd for every prime $p$. 

The task of this paper is to determine explicitly the refined class
number $h(H, G)$ for an arbitrary totally definite $H$ over any real quadratic field
$F$. A key step turns out to be determining explicitly $t(H, G)$ for
all finite \emph{noncyclic} $G$. Let $d\in \bbN$ be the square-free positive
integer such that $F=\Q(\sqrt{d})$, and $\varepsilon\in O_F^\times$ be the
fundamental unit of $F$. 
The cases $d\in \{2,3,5\}$ will be treated separately. Suppose that  $d\ge 6$. Up to isomorphism,  the reduced unit group $\calO^\bs$ of any $O_F$-order $\calO\subset H$
falls into the following list as shown in
Section~\ref{sec:minimal-g-orders}:
\begin{equation} \label{eq:calG} 
\calG=\{C_1,C_2,C_3,C_4, C_6,D_2^\dagger, D_2^\ddagger,  D_3^\dagger, D_3^\ddagger, D_4, D_6,A_4, S_4\}. 
\end{equation} 
Recall that $C_n$ denotes the cyclic group of order $n$, and $D_n$
denotes the dihedral group of order $2n$, so $D_2$ is just the Klein
4-group.  The reduced unit groups isomorphic to $D_n$ for $n=2,3$ are
further distinguished into two different kinds (labeled by
$D_n^\dagger$ and $D_n^\ddagger$ respectively) according to the
reduced norms of certain units (see
Definition~\ref{def-of-kinds}). The cases $D_n^\ddagger$ with $n\in \{2,3\}$ occur only when $\Nm_{F/\Q}(\varepsilon)=1$ since the reduced norm of every nonzero element in a totally definite quaternion algebra is totally positive. 
For a noncyclic group $G$ in $\calG$,
we introduce in Definitions~\ref{defn:minimal-g-orders} and~\ref{def:min-g-ord-up-isom} the notion of
\emph{minimal $G$-orders}.  Essentially, a minimal $G$-order is an
$O_F$-order $\calO$ in $H$ such that $\calO^\bs\simeq G$ and minimal
with respect to inclusion. By Proposition~\ref{prop:uniqueness-finit-subgp}, a minimal
$G$-order, if exists, is unique up to $O_F$-isomorphism. Since any
maximal $O_F$-order $\O$ with $\O^\bs\supseteq G$ contains a minimal
$G$-order, the calculation of $t(H, G)$ reduces to counting and classifying the maximal orders containing a fixed minimal
$G$-order $\calO$. Let $\beth(\calO)$ be the number of conjugacy
classes of  maximal orders containing $\calO$. 



We summarize our main results as the following two theorems. See Section~\ref{sec:maximal-orders} and Section~\ref{sec:maximal-orders-part2}.

\begin{thm}\label{1.2} 
   Suppose that $d\ge 6$. We have
\begin{align}
\label{eq:D_6} 
  t(D_6)&=
  \begin{cases}
    1 &\quad \text{if } H\simeq \qalg{-1}{-3}{F} \text{ and } 
    3\varepsilon\in F^{\times 2};\\
    0 &\quad \text{otherwise;}
  \end{cases}\\
  \label{eq:S_4}
  t(S_4)=t(D_4)&=
  \begin{cases}
    1 &\quad\text{if } H\simeq\qalg{-1}{-1}{F} \text{ and } 
       2 \varepsilon\in F^{\times 2};\\
    0 &\quad\text{otherwise;}
  \end{cases}\\
  \label{eq:A_4}
  t(A_4)&=
  \begin{cases}
    1 &\quad\text{if } H\simeq \qalg{-1}{-1}{F} \text{ and }  
    2\varepsilon\not\in F^{\times 2} ;\\
    0 &\quad\text{otherwise;}
  \end{cases} \\
  \label{eq:D_2I} 
    t(D_2^\dagger)&=
  \begin{cases}
    1 &\quad\text{if } 
    H\simeq\qalg{-1}{-1}{F}, \left(\frac{F}{2}\right)=0 \text{ and }
    2\varepsilon\not\in F^{\times 2}  
    ;\\ 0 &\quad\text{otherwise;}
\end{cases}\\
   \label{eq:D_3I}
    t(D_3^\dagger)&=
  \begin{cases}
    1 &\quad\text{if } 
    H\simeq\qalg{-1}{-3}{F} \text{ and }
    3\varepsilon\not\in F^{\times 2}  
    ;\\ 0 &\quad\text{otherwise.}
\end{cases} 
\end{align}
Here the Artin symbol $\left(\frac{F}{2}\right)=0$ if and only if $2$
is ramified in $F$. 
\end{thm}

\begin{thm}\label{1.3}
  Suppose that $d\ge 6$. For $n\in \{2,3\}$, let 
$\scrO_n^\ddagger$ be a minimal $D_n^\ddagger$-order, which is 
unique up to $O_F$-isomorphism if exists.

If $\Nm_{F/\Q}(\varepsilon)=1$ and $H\simeq \qalg{-1}{-\varepsilon}{F}$, then 
\begin{equation}\label{eq:D2II}
t(D_2^\ddagger)+t(D_4)+t(S_4)+t(D_6)=\beth(\scrO_2^\ddagger).
\end{equation}

If $\Nm_{F/\Q}(\varepsilon)=1$ and $H\simeq \qalg{-\varepsilon}{-3}{F}$, then 
\begin{equation}\label{eq:D3II}
t(D_3^\ddagger)+t(S_4)+t(D_6)=\beth(\scrO_3^\ddagger).
\end{equation}
The numbers $\beth(\scrO_n^\ddagger)$ for $n=2,3$ are determined
explicitly in Propositions~\ref{prop:scroII-unified-proof} and
\ref{prop:scrO3-typeII}.  In all the remaining cases, minimal $D_n^\ddagger$-orders do not exist and $t(D_n^\ddagger)=0$.
\end{thm}
Moreover, for each finite noncyclic group $G\in \calG$, if a maximal $G$-order (maximal $O_F$-order
with reduced unit group $G$) exists, then we write down explicitly a
representative $\O'$ in each $H^\times$-conjugacy class and calculate its normalizer $\calN(\O')$.     This leads to explicit formulas of $h(G)$ for noncyclic groups 
$G\in\calG$ in Section~\ref{sec:maximal-orders}
and Section~\ref{sec:maximal-orders-part2}.


The strategy for computing $h(C_n)$ with $n>1$ is described in
Section~\ref{sec:general-strategy-cyclic}. The main idea is to apply
the trace formula (\ref{eq:41}) (see
\cite[Theorem~III.5.2]{vigneras}).  According to (\ref{eq:42}), the
computations depend on the classification of maximal orders with
\emph{noncyclic} reduced unit groups. Lastly, $h(C_1)$ is computed
using the \emph{mass formula} (\ref{eq:114}) and the knowledge of
$h(G)$ for every nontrivial group $G$.  In conclusion, we obtain the
following result:

\begin{pnt}\label{1.4}
 Let $F=\Q(\sqrt{d})$ be a real quadratic field and $H$ be an
 arbitrary totally definite quaternion $F$-algebra. 
We have  a complete recipe for writing down $h(G)$ for each finite
group $G$. 
\end{pnt}

In fact, the only obstacle between us and a complete formula for
$h(G)$ is the overwhelming number of cases that the problem naturally
divides into, rendering any unified formula too cumbersome
and unwieldy.  However, for any class of quadratic real fields that
one has a good grasp on the fundamental units, the deduction of explicit
formulas for $h(G)$ based on our recipe becomes entirely routine. One such example is 
when $H=H_{\infty_1,\infty_2}$ and
$d=p$ is a prime; see Theorems~\ref{thm:p=1mod4typenum} and \ref{thm:p=3mod4typenum}. 
This extends a result of Hashimoto
\cite{Hashimoto-twisted-tr-formula} by a different method. 


\begin{rem}
  We emphasize that same as Vig\'eras's explicit formula \cite[Theorem
  3.1]{vigneras:ens}, our refined formula depends not only on the
  square-free integer $d$ defining $F$ but also on a good
  understanding of the fundamental unit $\varepsilon\in
  O_F^\times$.
  Our recipe for $h(G)$ refines the explicit formula for
  $h(H)$ given by Vign\'eras, which in turn is an application of the more general
  Eichler's class number formula \eqref{eq:168} \cite[Corollaire~V.2.5]{vigneras}. However, the equality
  $h(H)=\sum_{G\in \calG}h(G)$ does \emph{not} provide an alternative
  approach to her formula. Indeed,  to compute $h(C_n)$ with
  $n\in \{1, 2, 3, 4, 6\}$, we have to work out all the information
  required by Eichler's formula, so summing all $h(G)$
  together merely reproduces Vign\'eras's result along the same
  route. See Section~\ref{sect:mass-formula} for more details. 
\end{rem}  



For the following two theorems,  let $p\in \bbN$ be a prime, $F=\Q(\sqrt{p})$, and
$H=H_{\infty_1, \infty_2}$ be the totally definite quaternion $F$-algebra
unramified at all the finite places of $F$. We write $\zeta_F(s)$ for
the Dedekind zeta function of $F$, whose special value at $s=-1$ can be
calculated using Siegel's formula \cite[Table~2,
p.~70]{Zagier-1976-zeta}: 
\begin{equation}
  \label{eq:132}
\zeta_F(-1)=\frac{1}{60}\sum_{\substack{b^2+4ac=\grd_F\\ a,c>0}} a,
\end{equation}
where $\grd_F$ denotes the discriminant of
$F$, and $a,b,c\in \Z$. 
For simplicity, denote the class number $h(\Q(\sqrt{m}))$ by $h(m)$ for any square-free
$m\in \Z$. We first recall a result
of   Hashimoto  \cite{Hashimoto-twisted-tr-formula}. 
\begin{thm}[Hashimoto]
\label{thm:p=1mod4typenum} 
  If $p\equiv
  1\pmod 4$ and $p>5$, then 
 \begin{align*}
   t(C_1) &= \frac{\zeta_F(-1)}{2}-\frac{h(-p)}{8}-\frac{h(-3p)}{12}-\frac{1}{4}\Lsymb{p}{3}-\frac{1}{4}\Lsymb{2}{p}+\frac{1}{2}, \\
   t(C_2) &= \frac{h(-p)}{4}+\frac{1}{2}\Lsymb{p}{3}+\frac{1}{4}\left(\frac{2}{p}\right)-\frac{3}{4}, \\
   t(C_3) &= \frac{h(-3p)}{4}+\frac{1}{4}\Lsymb{p}{3}+\frac{1}{2}\left(\frac{2}{p}\right)-\frac{3}{4}, \\
   t(D_3) &= \frac{1}{2}\left( 1-\Lsymb{p}{3} \right), \\
   t(A_4) &= \frac{1}{2}\left( 1-\Lsymb{2}{p} \right).
 \end{align*}
\end{thm}

We get the results for the remaining primes $p$ as a direct application of
our recipe.
\begin{thm}\label{thm:p=3mod4typenum}
 If $p\equiv 3\pmod{4}$ and $p>5$, then
 \begin{align*}
   t(C_1) &= \frac{\zeta_F(-1)}{2}+\left(-7+3\left(\frac{2}{p}\right)\right)\frac{h(-p)}{8}-\frac{h(-2p)}{4}-\frac{h(-3p)}{12}+\frac{3}{2},  \\
   t(C_2) &= \left(2-\left(\frac{2}{p}\right)\right)\frac{h(-p)}{2}+\frac{h(-2p)}{2}-\frac{5}{2}, \\
   t(C_3) &= \frac{h(-3p)}{4}-1, \\
   t(C_4) &= \left(3-\left(\frac{2}{p}\right)\right)\frac{h(-p)}{2}-1, \\
   t(D_3) &= 1, \\
   t(D_4) &= 1, \\
   t(S_4) &= 1.
 \end{align*}
For $p=2,3$ and $5$, we have
\begin{center}
\begin{tabular}{|c|c|c|c|}
\hline
  $p$ &    $2$    &      $3$       &    $5$   \\ \hline
  $t(H)$ & $1$ & $2$ & $1$\\
  \hline
  $t(G)$ & $t(S_4)=1$   & $t(S_4)=t(D_{12})=1$ &   $t(A_5)=1$ \\
\hline
\end{tabular}
\end{center}
\end{thm}

Lastly, we apply the above results to the study of superspecial
abelian surfaces \cite[Definition~1.7, Ch.1]{li-oort}. Indeed, one of
our motivations is to count the number of certain superspecial abelian
surfaces with a fixed reduced automorphism group $G$. 
This extends results of our earlier works \cite{xue-yang-yu:sp_as2,xue-yang-yu:sp_as,xue-yang-yu:num_inv,xue-yang-yu:ECNF}
where we compute explicitly the number of these abelian surfaces over finite fields.  
We also construct superspecial abelian surfaces $X$ over some field $K$ of
characteristic $p$ with endomorphism algebra
$\End^0(X)=\Q(\sqrt{p})$, provided that
$p \not\equiv 1 \pmod{24}$. The construction makes use of results of
Florian Pop \cite{Pop-1996} on embedding problems for large fields.

This paper is organized as follows. In Section~\ref{sec:prelimiary}, we
recall some preliminary results on orders in quaternion algebras. The general strategy for
computing $h(G)$ for totally definite quaternion algebras over
arbitrary totally real fields is explained in
Section~\ref{sec:general-stra}. We restrict ourselves to the case of
quadratic real fields $F=\Q(\sqrt{d})$ starting from
Section~\ref{sec:minimal-g-orders}, where we introduce the concept of
minimal $G$-orders for finite noncyclic groups $G$. 
Section~\ref{sec:maximal-orders} and
Section~\ref{sec:maximal-orders-part2} contain the case-by-case study
of the maximal orders containing the minimal $G$-orders for each
$G$, and the results are summarized in Theorems~\ref{1.2} and
\ref{1.3}, respectively. In Section~\ref{sec:quadratic-orders}, we
classify the quadratic $O_F$-orders with nontrivial reduced unit groups in CM-extensions of $F$. 
The formulas in Theorem~\ref{thm:p=3mod4typenum} are calculated in
Section~\ref{sec:d=p}. We conclude with two applications to
superspecial abelian surfaces in Section~\ref{sec:Sp}.

\section{Preliminaries on orders in quaternion algebras}
\label{sec:prelimiary}
Throughout this section, $F$ is either a global field or a
non-archimedean local field,  and $H$ is a quaternion $F$-algebra. The algebra $H$
admits a canonical involution $x\mapsto \bar{x}$ such that
$\Tr(x)=x+\bar{x}$ and $\Nr(x)=x\bar{x}$ are respectively the reduced trace and
reduced norm of $x\in H$.  We
always assume that $\fchar(F)\neq 2$. If $H=\qalg{a}{b}{F}$ for
$a,b\in F^\times$, then $\{1, i, j, k\}$ denotes the standard
$F$-basis of $H$ subjected to the following multiplication rules
\begin{equation}
  \label{eq:12}
k=ij,\quad   i^2=a,\quad j^2=b, \quad \text{and}\quad ij=-ji. 
\end{equation}
When $F$ is local, $\qalg{a}{b}{F}$ splits over $F$ if and only if the Hilbert
symbol $(a,b)=1$. Often, $H$ is also presented as $K+Kx$, where
$K$ is a separable $F$-algebra of dimension $2$, and $x\in H$ is
an element such that
  \begin{equation}
    \label{eq:33}
x^2=c\in F^\times, \quad \text{and}\quad \forall\, y\in
K,\ xy=\bar{y}x. 
  \end{equation}
  Here $y\mapsto \bar{y}$ is the unique nontrivial $F$-automorphism of
  $K$. Following \cite[Section~I.1]{vigneras}, we write $H=\{K, c\}$
  for the above presentation.



If $F$ is local, then we fix a uniformizer $\pi\in F^\times$, and
write $\nu: F^\times\twoheadrightarrow \Z$ for the
discrete valuation of $F$.  Denote by $O_F, \grp, \grk$ respectively
the valuation ring, the maximal ideal and the residue field of $\nu$.
If $F$ is global, we fix a finite set $S$ of places of $F$ containing
all the archimedean ones, and write $O_F$ for the ring of $S$-integers of
$F$. In fact,  when $F$ is a number field,  $S$ will always be the
set of archimedean valuations (unless specified otherwise), so $O_F$ is just the usual ring of
integers of $F$.
For a nonzero prime ideal $\grp\subset O_F$, its corresponding
discrete valuation is denoted by $\nu_\grp$, and the $\grp$-adic
completion of $F$ is denoted by $F_\grp$.

Let $\Lambda$ be an $O_F$-\textit{lattice} in $H$, i.e. a finitely
generated $O_F$-submodule that spans $H$ over $F$. Its dual lattice is
defined to be
$\Lambda^\vee:=\{x\in H\mid \Tr(x\Lambda)\subseteq O_F\}$.  An
\textit{order} in $H$ always refers to an $O_F$-lattice that is at the
same time a subring of $H$ containing $O_F$. For an order
$\calO\subset H$, any maximal order $\O$ containing $\calO$ is a
lattice intermediate to $\calO\subseteq \calO^\vee$. There are only
finitely many such lattices. 

The
\textit{discriminant} of an order $\calO\subset H$ is denoted by
$\grd(\calO)$. If $\calO$ is a free $O_F$-module with basis
$\{x_1, \ldots, x_4\}$ (e.g.~when $F$ is local), then $\grd(\calO)$ is
the square root of the $O_F$-ideal
$\det(\Tr(x_s\bar x_t)_{1\leq s, t\leq 4})O_F$.  If
$\calO'\subseteq \calO$ is a suborder of $\calO$, then
$\grd(\calO')=\chi(\calO, \calO')\grd(\calO)$, where
$\chi(\calO, \calO')$ is the ideal index of $\calO'\subseteq\calO$ as
in \cite[Section~III.1]{Serre_local}. For any finite extension $K/F$,
we have $\grd(\calO\otimes_{O_F}O_K)=\grd(\calO)O_K$.  An order
$\calO$ is maximal if and only if $\grd(\calO)$ coincides with the
discriminant $\grd(H)$ of $H$, defined as the product of all finite primes of $F$
that are ramified in $H$ \cite[Definition~2.7.4]{hyperbolic-3-mfld}.  When $F$ is global, $\calO_\grp$ denotes the
$\grp$-adic completion of $\calO$ at a finite prime $\grp$.

\begin{sect}\label{sect:parity-of-element}
  Let $\calN(\calO)=\{x\in H^\times\mid x\calO x^{-1}=\calO\}$ be the
  normalizer of $\calO$.  First suppose that $F$ is local. Following
  \cite[Section~2]{Brzezinski-crelle-1990}, we say that an element
  $x\in H^\times$ is \textit{even} (resp.~\textit{odd}) if
  $\nu(\Nr(x))$ is even (resp.~odd). The notion of parity applies to
  elements of $H^\times/F^\times$ as well. Clearly, any unit
  $u\in \calO^\times$ is even.  Let $\O$ be a maximal order in a split
  quaternion algebra $H\simeq M_2(F)$. Then $\O$ is
  $H^\times$-conjugate to $M_2(O_F)$, and
  $\calN(\O)=F^\times\O^\times$ by \cite[Section~II.2, p.~40]{vigneras}
  (see also \cite[Proposition~2.1]{Brzezinski-crelle-1990}). In
  particular, if $x$ is odd, then $x\not\in \calN(\O)$. 
  
  If $F$ is
  global and $\grp$ is a finite prime of $F$, an element
  $x\in H^\times$ is said to be even (resp.~odd) at $\grp$ if
  $\nu_\grp(\Nr(x))$ is even (resp.~odd).
\end{sect}

\begin{lem}\label{lem:max-order-normalizer}
  Let $\O$ be a maximal order in $H$, and $u\in \calN(\O)$ be an element
  of the normalizer of $\O$. If $\Nr(u)\in O_F$, then $u\in \O$; if further $\Nr(u)\in O_F^\times$, then $u\in \O^\times$. 
\end{lem}
\begin{proof}
  Suppose that $\Nr(u)\in O_F$.  We show that
  $u\in \O_\grp$ for every finite prime $\grp$ of $F$. If $H$ is
  ramified at $\grp$, then $\O_\grp$ coincides with the unique maximal
  order $\{z\in H_\grp\mid \Nr(z)\in O_{F_\grp}\}$ in $H_\grp$; if
  $H$ splits at $\grp$, then $u\in \O_\grp$ because 
  $\calN(\O_\grp)=F_\grp^\times\O_\grp^\times$ as explained above. The second part of the lemma follows immediately. 
\end{proof}



\begin{sect}
  Given an order $\calO$ in $H$, we write $\grS(\calO)$ for the set of
  maximal orders containing $\calO$, and $\aleph(\calO)$ for the
  cardinality of $\grS(\calO)$. The quotient group
  $\calN(\calO)/\calO^\times$ acts naturally on $\grS(\calO)$ by
  conjugation. If $F$ is global and $\grp$ is a nonzero prime 
  ideal of $O_F$, we set 
  $\aleph_\grp(\calO):=\aleph(\calO_\grp)$. 
  Note that $\calO_\grp$ is maximal (forcing $\aleph_\grp(\calO)=1$) for almost all
$\grp$, and  
  \begin{equation}
    \label{eq:34}
\aleph(\calO)=\prod_\grp  \aleph_\grp(\calO), 
  \end{equation}
where the product runs over all finite primes of $F$.  If further $F$ is a number
  field and $p\in \bbN$ is an integral prime, then we set
  $\aleph_p(\calO)=\prod_{\grp\mid (pO_F)} \aleph_\grp(\calO)$. 

  Suppose that $F$ is local. If $H$ is division, then it has a unique
  maximal order, so $\aleph(\calO)=1$ for all $\calO$. Now suppose
  that $H=M_2(F)$. The \textit{Bruhat-Tits tree} $\grT$ (of
  $\PGL_2(F)$) is a homogeneous tree of degree $\abs{\grk}+1$ whose
  vertices are the maximal orders of $M_2(F)$ and such that two
  vertices are connected by an edge if and only if the two maximal
  orders have distance one \cite[Section~II.2]{vigneras}.  Let $\grT(\calO)$ be the subtree whose
  set of vertices is $\grS(\calO)$. For example, if $\calO=\O\cap \O'$
  is an Eichler order of level $\pi^n O_F$, i.e. the intersection of two
  maximal orders $\O$ and $\O'$ of distance $n$, then by
  \cite[Corollary~2.5]{Brzezinski-1983}, $\grT(\calO)$ is the unique
  path connecting $\O$ and $\O'$ on $\grT$, and
  $\aleph(\calO)=n+1$. In general, Arenas-Carmona
  \cite{Arenas-Carmona-2013} has shown that $\grT(\calO)$ is a
  \textit{thick line}, i.e.~the maximal subtree whose vertices lie no
  further than a fixed distance (called the \textit{depth}) from a
  line segment, which is called the \textit{stem} of the thick line. The stem may have
  length 0, in which case it degenerates into a single vertex.  A
  thick line with stem length $0$ and depth $1$ is called a
  \emph{star}, and its stem is called the \emph{center} of
  the star.
  Arenas-Carmona and Saavedra \cite{A-Carmona-Saavedra} provide
  concrete formulas for the depth and stem length (and hence
  $\aleph(\calO)$) of $\grT(\calO)$ when $\calO$ is generated by a
  pair of orthogonal pure quaternions. However, the present paper does
  not depend on their formulas.  It is also worthwhile to mention
  Fang-Ting Tu's result \cite{MR2825964} that the intersection of any
  finite number of maximal orders coincides with the intersection of
  three maximal orders.
\end{sect}

\begin{sect}\label{sect:def-eichler-invariant}
  Let $\grJ(\calO)$ be the \textit{Jacobson radical} of
  $\calO$, and $\grk'/\grk$ be the unique quadratic extension of the finite field
  $\grk$. First, suppose that $F$ is local.  Using lifting of idempotents, one shows that
  $\calO\simeq M_2(O_F)$ if and only if
  $\calO/\grJ(\calO)\simeq M_2(\grk)$
  (cf. \cite[Proposition~2.1]{Brzezinski-1983}). When
  $\calO\not \simeq M_2(O_F)$, we have
  \begin{equation}
    \label{eq:45}
\calO/\grJ(\calO)\simeq \grk,\qquad  \grk\times\grk,\quad \text{or}\quad \grk', 
  \end{equation}
   and the \textit{Eichler
    invariant} $e(\calO)$ is defined to be $0$, $1$ or $-1$
  accordingly \cite[Definition~1.8]{Brzezinski-1983}.  The Eichler invariant of $M_2(O_F)$ is defined to be $2$.
  By
  \cite[Chapter 6, Exercise~14]{Drozd-Kirichenko-finite-dim-alg},
  $e(\calO)$ behaves under a finite field extension $K/F$ in the
  following way: if $e(\calO)=-1$ and $\grk'$ is contained in the
  residue field of $K$, 
  then $e(\calO\otimes_{O_F}O_K)=1$, otherwise
  $e(\calO\otimes_{O_F}O_K)=e(\calO)$. We refer to
  \cite[Section~1]{Brzezinski-1983} for the concepts of
  \textit{Gorenstein}, \textit{Bass},  and \textit{hereditary} orders. 
  If $e(\calO)=1$, then $\calO$ is an Eichler order and hence a Bass order. If
  $e(\calO)=-1$, then $\calO$ is Bass as well. The Bass orders are
  explicitly described in \cite[Section~1]{Brzezinski-crelle-1990}.
  The orders with Eichler invariant 0 are more complicated, and the
  Gorenstein ones are discussed in \cite[Section~4]{Brzezinski-1983}. 

If $F$ is global and $\grp$ is a nonzero prime ideal of $O_F$, we denote by
$e_\grp(\calO)$ the Eichler invariant of $\calO_\grp$. An order $\calO$ is Gorenstein (resp.~Bass,
resp.~hereditary) if and only if $\calO_\grp$ is Gorenstein (resp.~Bass,
resp.~hereditary) for every $\grp$. 

The \emph{Brandt invariant} $\grb(\calO)$ is defined to be
$\grd(\calO)\Nr(\calO^\vee)$. By
\cite[Proposition~1.3]{Brzezinski-1983}, $\grb(\calO)$ is an integral
ideal of $O_F$, and $\calO$ is Gorenstein if and
only if $\grb(\calO)=O_F$.
\end{sect}

\begin{lem}\label{lem:bass-order-eichler-inv=0}
Let $H\simeq M_2(F)$ be a split quaternion algebra over a local field $F$, and
$\calO\subseteq H$ be a Bass order with Eichler invariant
$e(\calO)=0$. Then $\aleph(\calO)=2$. 
\end{lem}
\begin{proof}
  By \cite[Corollary~4.3]{Brzezinski-1983} and \cite[Section~1]{Brzezinski-crelle-1990}, every maximal order
  containing $\calO$ necessarily contains its hereditary closure
  $\calH(\calO)$, which has discriminant $\pi O_F$.  Since $H$ splits
  over $F$, $\calH(\calO)$ is an Eichler order of level $\pi O_F$.  There are
  precisely two maximal orders containing it.  
\end{proof}

\begin{lem}\label{lem:non-Gorenstein-order}
  Let $F$ be a local field, and  $H= M_2(F)$. Up to $H^\times$-conjugation, $\calO=O_F+\pi M_2(O_F)$ is
  the unique \emph{non-Gorenstein} order with $\grd(\calO)=\pi^3O_F$.  The subtree $\grT(\calO)$ is a star centered at the
  Gorenstein saturation\footnote{It is also called the \textit{Gorenstein
    closure} in some literature.} $M_2(O_F)$ of $\calO$. In particular,
  $\aleph(\calO)=\abs{\grk}+2$. 
\end{lem}
\begin{proof}
First, suppose that $\calO=O_F+\pi M_2(O_F)$. Its discriminant is
\[\grd(\calO)=\chi(M_2(O_F),
\calO)\grd(M_2(O_F))=\pi^3O_F.\] The dual lattice of
$\calO$ is 
\[\calO^\vee=\frac{1}{\pi}\left\{
  \begin{bmatrix}
    a & b\\ c & d
  \end{bmatrix}\in M_2(O_F)\,\middle\vert\, a\equiv d\pmod{\pi O_F}
\right\}, \]
which has reduced norm $\Nr(\calO^\vee)=\frac{1}{\pi^2}O_F$. Hence the
Brandt invariant $\grb(\calO)$ is $\grd(\calO)\Nr(\calO^\vee)=\pi O_F$. By the criterion in
\cite[Proposition~1.3]{Brzezinski-1983}, $\calO$ is non-Gorenstein.
In the notation of \cite[Section~1.1]{A-Carmona-Saavedra},
$\calO=M_2(O_F)^{[1]}$, so $\grT(\calO)$ is a star centered at
$M_2(O_F)$. More concretely, a maximal order contains
$\calO=O_F+\pi M_2(O_F)$ if and only if it has at most distance 1 from
$M_2(O_F)$.  The star has a unique center and $1+\abs{\grk}$
external vertices, so $\aleph(\calO)=\abs{\grk}+2$.

Conversely, let $\calO$ be an arbitrary non-Gorenstein order in
$H=M_2(F)$ with $\grd(\calO)=\pi^3O_F$, and $\Gor(\calO)$ be its
Gorenstein saturation. By \cite[Propositon~1.4]{Brzezinski-1983}, 
\[\calO=O_F+\grb(\calO)\Gor(\calO), \quad \text{where}
\quad \grb(\calO)=\pi^rO_F\quad \text{with} \quad r>1.\]
Hence  $\grd(\Gor(\calO))=\chi(\Gor(\calO),
\calO)^{-1}\grd(\calO)=\grb(\calO)^{-3}\grd(\calO)=\pi^{3-3r}O_F$. Since 
$\grd(\Gor(\calO))$ is integral, $r=1$ and $\Gor(\calO)$ is a maximal
order in $M_2(F)$.  Thus  $\calO$ is conjugate to $O_F+\pi
M_2(O_F)$. 
\end{proof}

\begin{lem}\label{lem:gen-lem-scro2II}
  Let $H=\qalg{a}{b}{F}$ with $a, b\in O_F\smallsetminus \{0\}$.  The
  order $\calO=O_F[i,j]$ is a Gorenstein order with 
  discriminant $\grd(\calO)=4abO_F$. Suppose further that $F$ is a
  number field, and both $a,b\in O_F^\times$. Then
  $\calO$ is maximal at every nondyadic prime of $F$, and
  $e_\grp(\calO)=0$ for every dyadic prime $\grp$.
  Moreover, if $a=-1$, then $(1+i)\in \calN(\calO)$.
\end{lem}
\begin{proof}
One calculates that $\grd(\calO)=4ab O_F$, and the dual
basis of $\{1, i, j, k\}$ is
\[\left\{\frac{1}{2}, \quad \frac{i}{2a}, \quad \frac{j}{2b}, \quad
  \frac{-k}{2ab}\right\}.\]
In particular, $O_F=\grd(\calO)\Nr(-k/2ab)\subseteq
\grd(\calO)\Nr(\calO^\vee)=\grb(\calO).$
It follows from  \cite[Proposition~1.3]{Brzezinski-1983} that $\calO$
is Gorenstein.  Suppose that $F$
is a number field and $a,b\in O_F^\times$. We have $\grd(\calO)=4O_F$,
and the maximality of $\calO$ at the nondyadic primes of $F$ follows
directly.  Let $\grp$ be a dyadic prime of $F$ with finite
residue field $\grk$.  Then $\calO_\grp$ is not maximal since
$\grd(\calO_\grp)$ is not square-free. By (\ref{eq:45}), an equation of
the form $x^2=c\in \grk$ with $x\in\calO_\grp/\grJ(\calO_\grp)$ has a
unique solution that lies in $\grk$.  It follows that the reductions
of both $i$ and $j$ modulo $\grJ(\calO_\grp)$ are in $\grk$, and hence
$\calO_\grp/\grJ(\calO_\grp)=\grk$, i.e. $e_\grp(\calO)=0$.  When
$a=-1$, we have
\begin{equation}\label{eq:28}
(1+i)i(1+i)^{-1}=i\quad  \text{and}\quad  (1+i)j(1+i)^{-1}=k.  
\end{equation}
Thus $(1+i)\in \calN(\calO)$. 
 \end{proof}

 \begin{lem}\label{lem:scro2II-Bass-criterion}
   Let $H=\qalg{a}{b}{F}$ with
   $a, b\in O_F\smallsetminus \{0\}$, and $\calO=O_F[i,j]$.  Suppose that
   $F$ is a local field, $\grd(\calO)=4abO_F$ is divisible by
   $\pi^3O_F$, and $e(\calO)=0$. Then $\calO$ is a Bass
   order if and only if at least one of the orders $O_F[i]$ or $O_F[j]$ is maximal in
   its total quotient ring. 
 \end{lem}
 \begin{proof}
The sufficiency follows directly from
\cite[Proposition~1.11]{Brzezinski-crelle-1990}.  We  prove the
converse. So assume that neither $O_F[i]$ nor $O_F[j]$ are maximal
orders. Then there exists a unique overorder $B\subset F(i)$ of $O_F[i]$
such that $\chi(B, O_F[i])=\pi O_F$. Similarly, we find an overorder
$B'\subset F(j)$ of $O_F[j]$ with $\chi(B', O_F[j])=\pi O_F$. Now let $O=B+Bj$, and
$O'=B'+iB'$. These are distinct overorders of $\calO$ since
$O\cap F(i)=B$ while $O'\cap F(i)=O_F[i]\neq B$.
The choices of $B$ and $B'$ imply that 
\begin{equation}
\label{eq:58}
\chi(O,\calO)=\chi(O', \calO)=\pi^2O_F.
\end{equation}
Hence $\grd(O)=\grd(O')=\pi^{-2}\grd(\calO)\neq O_F$ by the assumption on
$\grd(\calO)$. It follows from
\cite[Corollary~4.3]{Brzezinski-1983} that
$\calO$ cannot be Bass, otherwise the 
overorder of $\calO$ satisfying (\ref{eq:58}) is unique.  
 \end{proof}

\begin{lem}\label{lem:gen-lem-scro3II}
  Let $H=\qalg{a}{b}{F}$ with $a,b\in O_F\smallsetminus \{0\}$ and
  $b\equiv 1\pmod{4O_F}$.  Then $\calO:=O_F[i, (1+j)/2]$ is a
  Gorenstein order with $\grd(\calO)=abO_F$ and $j\in \calN(\calO)$.
  Suppose further that $F$ is a number field,  $a\in O_F^\times$ and
  $b=-3$. Then $\calO$ is maximal at any nonzero prime of $O_F$
  coprime to $3$. Let $\grp$ be a finite prime of $F$ above $3$ with
  residue field $\grk$. If
  $H_\grp$ is division and $\grp$ is unramified above $3$, then
  $\calO_\grp$ is maximal; otherwise we have
\[e_\grp(\calO)=
\begin{cases}
  1 \qquad & \text{if } (a\bmod{\grp}) \in (\grk^\times)^2,\\
 -1 \qquad & \text{if } (a\bmod{\grp}) \not\in (\grk^\times)^2. 
\end{cases}
\] 
\end{lem}
\begin{proof}
  Note that $(1+j)/2$ is integral over $O_F$ since
  $b\equiv 1\pmod{4O_F}$, and one easily verifies that
  $j\in \calN(\calO)$. The order $\calO$ is a free $O_F$-module with
  basis $\{1, i, (1+j)/2, (i+k)/2\}$.  Direct calculation shows that
  the dual basis is
  \begin{equation}
\label{eq:61}
\left\{\frac{b-j}{2b}, \quad \frac{bi+k}{2ab},\quad \frac{j}{b},      \quad \frac{k}{ab}\right\}, 
  \end{equation}
and $\grd(\calO)=abO_F$. By \cite[Proposition~1.3(b)]{Brzezinski-1983},
$\calO$ is Gorenstein since $\grb(\calO)=O_F$.  

Now suppose that $F$ is a number field, $a\in O_F^\times$ and
$b=-3$. Then $\grd(\calO)=3O_F$, and hence $\calO$ is maximal at any
nonzero prime of $O_F$ coprime to $3$.  If $\grp$ is a prime
unramified above $3$ and $H_\grp$ is division, then
$\grd(\calO_\grp)=\grp O_{F_\grp}=\grd(H_\grp)$, so $\calO_\grp$ is
the unique maximal order in $H_\grp$. In the remaining cases,
$\calO_\grp$ is non-maximal. Since $(1+j)/2$ is a primitive third root
of unity and $\fchar(\grk)=3$, we have
$(1+j)/2\equiv 1 \pmod{\grJ(\calO_\grp)}$ by (\ref{eq:45}).  
Let $\breve{\imath}$ denote the image of $i$ in
$\calO_\grp/\grJ(\calO_\grp)$. We then have 
\[\grk[\,\breve{\imath}\,]=
\begin{cases}
  \grk\oplus \grk &\text{if} \quad (a\bmod{\grp}) \in (\grk^\times)^2,\\
  \grk'           &\text{if} \quad (a\bmod{\grp}) \not\in (\grk^\times)^2.
\end{cases}
\]
The formula for $e_\grp(\calO)$ follows.
\end{proof}

For the convenience of the reader, we state the following well-known lemma. 

\begin{lem}\label{lem:unram-ext}
  Let $Q$ be a non-archimedean local field with ring of integers $R$ and uniformizer
  $\pi_0$, and $H_0$ be a quaternion division algebra over $Q$ with the
  unique maximal $R$-order $\O_0$. Suppose that $F/Q$ is an unramified
  quadratic extension. Then $\calO:=\O_0\otimes_RO_F$ is an Eichler
  order of level $\pi_0 O_F$ in $H:=H_0\otimes_QF\simeq M_2(F)$.
\end{lem}
\begin{proof}
  As remarked in Section~\ref{sect:def-eichler-invariant},
  $\calO$  is an order with Eichler invariant $1$ and discriminant
  $\pi_0O_F$. The lemma follows from \cite[Proposition~2.1]{Brzezinski-1983}. 
\end{proof}

We treat the case where $F/Q$ is ramified next. 
\begin{lem}\label{lem:ram-ext}
  Let $Q, R, H_0, \O_0$ and $\pi_0$ be as in
  Lemma~\ref{lem:unram-ext}. Suppose that $F/Q$ is a ramified
  quadratic extension. Then there is a unique $O_F$-order $\O$ in
  $H:=H_0\otimes_Q F\simeq M_2(F)$ properly containing
  $\calO:=\O_0\otimes_R O_F$. The order $\O$ is necessarily maximal, and 
$\calO\simeq O_L+\pi \O$, where $L/F$ is the unique unramified
quadratic extension of $F$, identified with a subfield of $H$ such
that $O_L\subset \O$.  
\end{lem}
By \cite[Theoreme~3.2]{vigneras}, any two embeddings of $O_L$ into
$\O$ are conjugate by an element of $\O^\times$. Thus the structure of
$O_L+\pi \O$ does not depend on the choice of the embedding. 
\begin{proof}
 By Section~\ref{sect:def-eichler-invariant}, 
  $e(\calO)=e(\O_0)=-1$.  So according to
  \cite[Proposition~3.1]{Brzezinski-1983}, there is a unique \textit{minimal}
  overorder $\O$ containing $\calO$ with
  $\chi(\O, \calO)=\pi^2O_F$ and
  $\grJ(\calO)=\pi\O$. We have 
  \[\grd(\O)=\grd(\O_0)\chi(\O,
  \calO)^{-1}=\pi_0\pi^{-2}O_F=O_F. \]
Therefore,  $\O$ is a maximal order in $H$, and hence the unique order
properly containing $\calO$. 

Let $L_0/Q$ be the unique unramified quadratic extension of $Q$,
identified with a subfield of $H_0$ such that $H_0=L_0+L_0x$ with
$x^2=\pi_0$ and $xy=\bar{y}x$ for all $y\in L_0$.  Then by
\cite[Corollary~II.1.7]{vigneras}, $\O_0=O_{L_0}+O_{L_0}x$.  Since
$F/Q$ is ramified, we have 
$L=L_0\otimes_\Q F$ and $O_L=O_{L_0}\otimes_R O_F$. Hence 
\begin{equation}
  \label{eq:35}
\calO=\O_0\otimes_RO_F=O_L+O_Lx. 
\end{equation}
In particular, $\calO\supseteq O_L+\pi\O$. On the other hand, 
$\chi(\O, O_L+\pi\O)=\pi^2O_F=\chi(\O, \calO)$.  Thus
$\calO=O_L+\pi\O$ (see also \cite[Proposition~1.12]{Brzezinski-crelle-1990}). 
\end{proof}

\section{The General strategy}
\label{sec:general-stra}
In this section, we assume that $F$ is a totally real number field,
and $H$ is a totally definite
quaternion $F$-algebra. 
We keep the notation in the introduction and fill in
the details for the strategy of computing  $h(G)=h(H, G)$ in
(\ref{eq:105}). Let $\calO$ be an arbitrary $O_F$-order in $H$.  
We first study the general structure of the reduced unit
group $\calO^\bs=\calO^\times/O_F^\times$.

\begin{sect}\label{sect:general-classification-grp}
Fix an embedding $\sigma:F\hookrightarrow \R$.  We have the following
successive inclusions of groups
\[\calO^\bs=\calO^\times/O_F^\times\hookrightarrow H^\times/F^\times\xhookrightarrow{\sigma}
\H^\times/\R^\times.  \]
The quotient $\H^\times/\R^\times$ is 
isomorphic to the special orthogonal group $\SO_3(\R)$, whose finite subgroups have been
classified \cite[Section~I.3]{vigneras}.  Thus $\calO^\bs$ is isomorphic to
one of the  following groups: 
\begin{itemize}
\item a cyclic group $C_n$ of order $n\geq 1$;
\item a dihedral group $D_n$ of
order $2n$ with $n\geq 2$;
\item an exceptional group $A_4$, $S_4$, or $A_5$. 
\end{itemize}
\end{sect}



Since $H$ is totally definite, the field $F(x)$ generated by a
\emph{non-central} element $x\in H$ is a CM-extension (i.e.~a totally
imaginary quadratic extension) of $F$.  An $O_F$-order in a
CM-extension of $F$ is called a \emph{CM $O_F$-order}. 
 If $x\in \calO$, then
$B:=F(x)\cap \calO$ is a CM $O_F$-order;  if further $x\in
\calO^\times$, then 
\begin{equation}
  \label{eq:121}
w(B):=[B^\times:O_F^\times]>1 
\end{equation}
since $x\in \calO^\times$ is non-central. 
Let $\scrB$ be the set of all CM $O_F$-orders $B$ (up to isomorphism)
with $w(B)>1$. We recall the following two facts:
\begin{enumerate}[(i)]
\item $\scrB$ is a finite set; 
\item the reduced unit group $B^\bs:=B^\times/O_F^\times$ is cyclic
  for every $B\in \scrB$. 
\end{enumerate}
Indeed, Pizer denotes the set $\scrB$ by ``$C_1$'' in \cite[Remarks,
p.~92]{Pizer1973} shows that it is a subset of a larger finite set
``$C$'' (ibid.).  Classically, the finiteness of $\scrB$ guarantees that the summation in the Eichler's
class number formula (\ref{eq:168}) has only finitely many terms.  When $F$ is a
quadratic real field $\Q(\sqrt{d})$ with square-free $d\geq 6$, the set $\scrB$ is
classified in Section~\ref{sec:quadratic-orders}. We will prove (ii) in Lemma~\ref{lem:cyclic-red-gp-CM-fields} (see also \cite[Section~2.1]{xue-yang-yu:num_inv}).

\begin{sect}
Let $\calO$ be an arbitrary $O_F$-order in $H$.   Given an $O_F$-order $B$ inside a CM-extension $K/F$,
  we write $\Emb(B, \calO)$ for the \textit{finite} set of optimal $O_F$-embeddings of $B$
  into $\calO$. 
In other words, 
\[\Emb(B,\calO):=\{\varphi\in \Hom_F(K, H)\mid \varphi(K)\cap
\calO=\varphi(B)\}.\]
The unit group $\calO^\times$ acts on $\Emb(B, \calO)$ from the
right by $\varphi\mapsto u^{-1}\varphi u$ for all
$\varphi\in \Emb(B, \calO)$ and $u\in \calO^\times$, and the action descends to  $\calO^\bs$. We denote
\begin{equation}
  \label{eq:133}
 m(B,\calO,\calO^\times):= \abs{ \Emb(B,\calO)/\calO^\times}. 
\end{equation}
For
each finite prime $\grp$ of $F$, we set
\begin{equation}
  \label{eq:109}
m_\grp(B):=m(B_\grp,\calO_\grp,\calO_\grp^\times)=\abs{
  \Emb(B_\grp,\calO_\grp)/\calO_\grp^\times}.
\end{equation}


Given a fractional locally principal right ideal $I$ of $\calO$, we
write $[I]$ for its ideal class, and $\calO_l(I)$ for its associated
left order $\{x\in H\mid xI\subseteq I\}$.  Let  $\Cl(\calO)$ be
the finite set of locally principal right ideal classes of $\calO$. 
By \cite[Theorem~5.11, p. 92]{vigneras}, 
\begin{equation}
  \label{eq:41}
  \sum_{[I]\in \Cl(\calO)} m(B,\calO_l(I), \calO_l(I)^\times)=h(B)\prod_{\grp} m_\grp(B),
\end{equation}
where the product on the right runs over all finite primes of $F$, and $m_\grp(B)=1$ for all but finitely many $\grp$. A priori,  \cite[Theorem~5.11]{vigneras} is stated for Eichler
orders, but it applies in much more general settings.  See
\cite[Lemma~3.2]{wei-yu:classno} and
\cite[Lemma~3.2.1]{xue-yang-yu:ECNF}.
When $\calO$ is maximal, we have 
\begin{equation}\label{eq:111}
m_\grp(B):=
\begin{cases}
  1-\left(\frac{B}{\grp}\right )  & \text{if }
  \grp|\grd(H),\\
  1 & \text{otherwise},
\end{cases}   
\end{equation}
where $\left(\frac{B}{\grp}\right )$ is the Eichler symbol
\cite[p.~94 and Section~II.3]{vigneras}. 
\end{sect}

\begin{sect}\label{sect:maximal-cyclic-subgp}
  Suppose that $\calO^\times\neq O_F^\times$. We explain the basic
  idea of counting $m(B, \calO, \calO^\times)$ for $B\in \scrB$.  A
  nontrivial cyclic subgroup of $\calO^\bs$ is said to be
  \emph{maximal} if it is not properly contained in any other cyclic
  subgroup of $\calO^\bs$. For each CM $O_F$-order $B\in \scrB$, an
  optimal $O_F$-embedding $\varphi: B\to \calO$ identifies $B^\bs$
  with a maximal cyclic subgroup of $\calO^\bs$.  Conversely, let $C$
  be a maximal cyclic subgroup of $\calO^\bs$, and $u\in \calO^\times$
  be a representative of an arbitrary nontrivial element of $C$.  The
  CM-field $F(u)\subset H$ depends only on $C$ and not on the choice
  of $u$, and the same holds true for the $O_F$-order
  $B_C:=F(u)\cap \calO$.  There is an optimal $O_F$-embedding
  $\varphi: B\to \calO$ such that $\varphi(B^\bs)=C$ if and only if
  $B\simeq B_C$.

  For each $B\in \scrB$, let $\scrC(B, \calO)$ be the subset of maximal
  cyclic subgroups $C\subseteq \calO^\bs$ such that $B_C\simeq B$. It
  is clear that $\scrC(B, \calO)$ is invariant under the conjugation by
  $\calO^\bs$ from the right.  We have an $\calO^\bs$-equivariant
  surjective $2:1$ map
\[\Emb(B, \calO)\twoheadrightarrow \scrC(B, \calO), \qquad \varphi,
\bar\varphi\mapsto \varphi(B^\bs),\]
where $\bar\varphi$ denotes the complex conjugate of $\varphi$. 

For each $C\in \scrC(B, \calO)$, let $N_{\calO^\bs}(C)$ be the
normalizer of $C$ in $\calO^\bs$. We claim that the induced map 
\begin{equation}
  \label{eq:122}
  \Emb(B, \calO)/\calO^\bs\twoheadrightarrow \scrC(B, \calO)/\calO^\bs
\end{equation}
is a 2:1 cover ramified over the conjugacy classes $[C]$ with
$N_{\calO^\bs}(C)\neq C$. In other words, $\varphi$ and $\bar\varphi$
belong to the same $\calO^\times$-conjugacy class if and only if
$N_{\calO^\bs}(\varphi(B^\bs))\neq \varphi(B^\bs)$.  The necessity is
obvious. For the sufficiency, suppose that
$\tilde{v}\in N_{\calO^\bs}(\varphi(B^\bs))$ and
$\tilde{v}\not\in \varphi(B^\bs)$. Let $v\in \calO^\times$ be a
representative of $\tilde{v}$. Then $v^{-1}\varphi(B) v=\varphi(B)$
since $\varphi(B)$ is uniquely determined by the maximal cyclic subgroup
$\varphi(B^\bs)\subseteq \calO^\bs$. It follows that
$v^{-1}\varphi v\in \{\varphi, \bar\varphi\}$.  But $v^{-1}\varphi
v\neq \varphi$, otherwise $v$ lies in the centralizer of $\varphi(B)$
in $\calO$, which is $\varphi(B)$ itself.  Therefore, $v^{-1}\varphi
v=\bar\varphi$, and our claim is verified. 
\end{sect}

The dihedral group $D_2$ is just the Klein $4$-group. It is
abelian and has $3$ maximal cyclic subgroups of order $2$.  If
$\calO^\bs\simeq D_2$, then (\ref{eq:122}) is a bijection and $m(B,
 \calO, \calO^\times)=\abs{\scrC(B, \calO)/\calO^\bs}=\abs{\scrC(B, \calO)}$. 

\begin{prop}\label{prop:explicit-num-opt-embed}
  Assume that $\calO^\times \neq O_F^\times$ and $\calO^\bs\not\simeq
  D_2$. Then $m(B, \calO,
\calO^\times)\leq 2$ for every CM $O_F$-order  $B\in \scrB$. More explicitly, we have the following cases. 
\newcounter{propm}
\begin{enumerate}
\item If $B\not\simeq B_C$ for any maximal cyclic subgroup $C$ of
  $\calO^\bs$, then $\Emb(B, \calO)=\emptyset$, and hence
$m(B, \calO, \calO^\times)=0$. 
\item If $\calO^\bs$ is cyclic, then there is a unique
  $B'\in \scrB$ such that $\Emb(B', \calO)\neq \emptyset$. Moreover,
  we have 
  \begin{equation}
    \label{eq:124}
 m(B', \calO, \calO^\times)=2.    
  \end{equation}
\setcounter{propm}{\value{enumi}} 
\end{enumerate}
Assume further that $B\simeq B_C$ for
some maximal cyclic subgroup $C$ of $\calO^\bs$.
\begin{enumerate}
\setcounter{enumi}{\value{propm}}
\item Suppose that $\calO^\bs\simeq D_n$ with $n\geq 3$. If $n$ is
  even, then $\calO^\bs$ has two conjugacy classes of maximal cyclic
  subgroups of order $2$, denoted by $[C']$ and
  $[C'']$.
\[m(B, \calO, \calO^\times)=
\begin{cases}
1 \qquad &\text{if }  \abs{C}=n;\\
2 \qquad &\text{if }  \abs{C}=2 \text{ and $n$ is odd};\\
1 \qquad &\text{if }  \abs{C}=2, \text{ $n$ is even and }
B_{C'}\not\simeq B_{C''};\\
2 \qquad &\text{if }  \abs{C}=2, \text{ $n$ is even and }
B_{C'}\simeq B_{C''}.
\end{cases}\]
\item If $\calO^\bs\simeq S_4$ or $A_5$, then $m(B, \calO,
  \calO^\times)=1$; if $\calO^\bs\simeq A_4$, then 
\[m(B, \calO, \calO^\times)=
\begin{cases}
1 \qquad &\text{if }  \abs{C}=2;\\
2 \qquad &\text{if }  \abs{C}=3.
\end{cases}\]


\end{enumerate}
\end{prop}
\begin{proof}
Part (1) follows directly from Section~\ref{sect:maximal-cyclic-subgp}.
The remaining parts all reduce to counting the conjugacy classes of
maximal cyclic subgroups in $\calO^\bs$ and working out the
normalizers. For example, 
if $\calO^\bs\simeq
  D_n$ with $n\geq 3$, then we present it as 
\begin{equation}
  \label{eq:130}
\calO^\bs=\dangle{\tilde\eta, \tilde u\in \calO^\bs \mid
  \ord(\tilde\eta)=n, \ord(\tilde u)=2, \tilde\eta\tilde  u=\tilde u
    \tilde\eta^{-1}}. 
\end{equation}
It has a unique maximal cyclic subgroup of order $n$, namely
$\dangle{\tilde{\eta}}$, which is normal. Let $C$ be a maximal cyclic
subgroup of $\calO^\bs$ distinct from $\dangle{\tilde\eta}$. Then
$\abs{C}=2$ and $C=\dangle{\tilde u\tilde\eta^r}$ for some
$0\leq r < n$. If $n$ is odd, such subgroups form a single conjugacy
class, and $N_{\calO^\bs}(C)=C$; if $n$ is even, then they form two conjugacy
classes according to the parity of $r$ (e.g.~we may set $C'=\dangle{\tilde u}$ and
$C''=\dangle{\tilde u\tilde\eta}$), and $N_{\calO^\bs}(C)\neq C$ since
$\tilde\eta^{n/2}$ lies in the center of $\calO^\bs$.  This
proves part (3).  The proof of the remaining parts are left to the interested reader.
\end{proof}

For the rest of this section, we fix a maximal order $\O\subset H$. 
\begin{sect}
   Denote the $H^\times$-conjugacy
  class of a maximal order
  $\O'\subset H$ by $\dbr{\O'}$.   There is a surjective map of finite
  sets 
\begin{equation}
  \label{eq:36}
\Upsilon:  \Cl(\O)\to \Tp(H), \qquad [I]\mapsto \dbr{\calO_l(I)}. 
\end{equation}
The fibers of $\Upsilon$ is studied in
\cite[Section~1.7]{vigneras:ens}, whose result we briefly recall below.


By \cite[Theorem~22.10]{reiner:mo}, the set of nonzero two-sided
fractional ideals of $\O$ forms a commutative multiplicative group
$\scrI(\O)$, which is a free abelian group generated by the nonzero prime
two-sided 
ideals of $\O$. Let $\scrP(\O)\subseteq \scrI(\O)$ be the subgroup of nonzero
principal two-sided fractional ideals of $\O$, and $\scrP(O_F)$ the
group of nonzero principal fractional $O_F$-ideals, identified with a subgroup
of $\scrP(\O)$ via $xO_F\mapsto x\O, \forall x\in F^\times$.  For any maximal order
$\O'$,  there is a
bijection (see \cite[Section~1.7]{vigneras:ens},
\cite[Lemma~III.5.6]{vigneras})
\begin{equation}
  \label{eq:37}
\Upsilon^{-1}(\dbr{\O'})\longleftrightarrow \scrI(\O')/\scrP(\O'). 
\end{equation}
The quotient group $\scrI(\O')/\scrP(\O')$ sits in a short exact sequence \cite[Theorem~55.22]{curtis-reiner:2}
\begin{equation}
  \label{eq:38}
  1\to \calN(\O')/(F^\times\O'^\times)\to \Pic(\O')\to
  \scrI(\O')/\scrP(\O')\to 1. 
\end{equation}
Here $\Pic(\O')$ denotes the Picard group $\scrI(\O')/\scrP(O_F)$, whose
cardinality can be calculated using the short exact sequence  \cite[Theorem~55.27]{curtis-reiner:2}
\begin{equation}
  \label{eq:39}
  1\to \Cl(O_F)\to \Pic(\O')\to \prod_{\grp\mid \grd(H)} (\zmod{2})\to
  0, 
\end{equation}
where $\Cl(O_F)$ denotes the ideal class group of $O_F$. 
It follows that 
\begin{equation}
  \label{eq:40}
\abs{\Upsilon^{-1}(\dbr{\O'})}=\frac{2^{\omega(H)}
  h(F)}{\abs{\calN(\O')/(F^\times\O'^\times)}}, 
\end{equation}
where $\omega(H)$ is the number of finite primes of $F$ that are
ramified in $H$.
\end{sect}

\begin{rem}\label{rem:free-act-odd-hF}
There is a natural action of $\Cl(O_F)$ on $\Cl(\O)$ as follows: 
\begin{equation}
  \label{eq:169}
  \Cl(O_F)\times \Cl(\O)\to \Cl(\O), \qquad ([\gra], [I])\mapsto [\gra
  I]. 
\end{equation}
Let $\ol{\Cl}(\O):=  \Cl(O_F)\backslash \Cl(\O)$ be the set of orbits
of this action. Clearly, $\Upsilon([I])=\Upsilon([\gra I])$, so
$\Upsilon$ factors through $\ol{\Cl}(\O)$. It is shown in
\cite[Section~3]{xue-yu:type_no} that 
\begin{itemize}
\item if $H$ splits at all finite places of $F$, then the induced map
  $\ol{\Cl}(\O)\to \Tp(H)$ is bijective; 
\item if $h(F)$ is odd, then the action in (\ref{eq:169}) is free. 
\end{itemize}
This is the gist of the proof of Proposition~\ref{1.1}, and we
refer to \cite[Section~3]{xue-yu:type_no} for details. 
\end{rem}

\begin{sect}\label{sec:general-strategy-cyclic}
We explain the strategy for calculating $h(G)=h(H, G)$ when $G=C_n$ is a
cyclic group of order $n>1$. Let $\scrB_n\subseteq \scrB$ be the finite subset of
CM $O_F$-orders $B$  such that
$B^\bs\simeq C_n$. For each fixed
  $B\in \scrB_n$, we define
\begin{equation}
  \label{eq:135}
  h(C_n, B):=\#\{[I]\in \Cl(\O)\mid \calO_l(I)^\bs \simeq C_n,
  \text{ and } \Emb(B, \calO_l(I))\neq \emptyset\}. 
\end{equation}
According to part (2) of
Proposition~\ref{prop:explicit-num-opt-embed}, 
\begin{equation}
  \label{eq:43}
h(C_n)=\sum_{B\in \scrB_n}h(C_n, B).   
\end{equation}


We divide the type set $\Tp(H)$ into two subsets: 
\[ \Tp^\circ(H):=\{\dbr{\O'}\in \Tp(H)\mid \O'^\bs \text{ is
    cyclic}\},\quad \text{and} \quad \Tp^{\natural}(H):=\Tp(H)\smallsetminus\Tp^\circ(H). \]
In the trace formula (\ref{eq:41}), the summation $\sum_{[I]\in
  \Cl(\calO)} m(B,\calO_l(I), \calO_l(I)^\times)$ is decomposed as two
parts accordingly: one sums over all $[I]\in \Cl(\calO)$ with $\dbr{\calO_l(I)}\in
\Tp^\circ(H)$, and the other with $\dbr{\calO_l(I)}\in
\Tp^\natural(H)$. 
Thanks to (\ref{eq:124}), 
\begin{equation}
  \label{eq:125}
\sum_{\substack{[I]\in \Cl(\O), \\ \dbr{\calO_l(I)}\in \Tp^\circ(H)}}
m(B, \calO_l(I), \calO_l(I)^\times)=2h(C_n, B). 
\end{equation}
On the other hand, note that if $\Upsilon([I])=\Upsilon([J])$, then
$\calO_l(I)\simeq \calO_l(J)$, so 
\[  m(B, \calO_l(I), \calO_l(I)^\times)=  m(B, \calO_l(J),
  \calO_l(J)^\times). \]
It follows that 
\begin{equation}
  \label{eq:126}
  \sum_{\substack{[I]\in \Cl(\O), \\ \dbr{\calO_l(I)}\in \Tp^{\natural}(H)}}
m(B, \calO_l(I), \calO_l(I)^\times)=  \sum_{\dbr{\O'}\in
  \Tp^\natural(H)}  \abs{\Upsilon^{-1}(\dbr{\O'})} m(B, \O',
 \O'^\times). 
\end{equation}
Combining (\ref{eq:41}), (\ref{eq:111}), (\ref{eq:40}), 
(\ref{eq:125}) and (\ref{eq:126}), we obtain the following
equation for each fixed $B\in \scrB_n$:
\begin{equation}
  \label{eq:42}
2h(C_n,
B)+2^{\omega(H)}h(F)\sum_{\dbr{\O'}\in \Tp^\natural(H)}\frac{m(B, \O',
  \O'^\times)}{\abs{\calN(\O')/(F^\times\O'^\times)}}=h(B)\prod_{\grp\vert\grd(H)} \left(1-\Lsymb{B}{\grp}\right )
\end{equation}
Therefore, to compute $h(C_n, B)$ (and hence $h(C_n)$), it is
crucial to classify all isomorphism classes of maximal orders with
\emph{noncyclic} reduced unit groups, and ideally, to write them down
as explicitly as possible. Sections~\ref{sec:minimal-g-orders}--\ref{sec:maximal-orders-part2} are devoted
to this task. Once this is done, then for any noncyclic group $G$, 
\begin{equation}
  \label{eq:107}
h(G)=\sum \frac{2^{\omega(H)}h(F)}{\abs{\calN(\O')/(F^\times\O'^\times)}},
\end{equation}
where the summation runs over all  $\dbr{\O'}\in \Tp^\natural(H)$
with $  \O'^\bs\simeq G$. 
\end{sect}

\begin{sect}\label{sect:mass-formula}
  Lastly, we compute $h(C_1)$ by the \emph{mass formula}. Recall that the 
  \emph{mass} of an arbitrary $O_F$-order $\calO\subset H$ is defined as 
  \begin{equation}
    \label{eq:112}
    \Mass(\calO):=\sum_{[I]\in \Cl(\calO)}\frac{1}{\abs{\calO_l(I)^\bs}}. 
  \end{equation}
The mass of a maximal order $\O$ can be calculated by the  mass formula  \cite[Corollaire~V.2.3]{vigneras}
\begin{equation}
  \label{eq:113}
\Mass(\O)= \frac{ h(F)|\zeta_F(-1)| }{2^{([F:\Q]-1)}} \prod_{\grp
  | \grd(H) } 
  (N(\grp)-1).
\end{equation}
See \cite[Lemma~5.1.2]{xue-yang-yu:ECNF} for the mass formula of an
arbitrary $O_F$-order in $H$. 
Once $h(G)$ is known for every nontrivial finite group $G$, then we have 
\begin{equation}
  \label{eq:114}
  h(C_1)=\Mass(\O)-\sum_{G\neq C_1}\frac{h(G)}{\abs{G}}.  
\end{equation}

Alternatively, one could compute $h(C_1)$ by 
\begin{equation}\label{eq:171}
h(C_1)=h(\O)-\sum_{G\neq C_1}h(G).   
\end{equation}
According to  Eichler's class number formula \cite[Corollaire~V.2.5]{vigneras},
\begin{equation}
  \label{eq:168}
h(\calO)=\Mass(\calO)+\frac{1}{2}\sum_{B\in
  \scrB}h(B)(1-w(B)^{-1})\prod_{\grp}m_\grp(B),  
\end{equation}
where $w(B)=[B^\times:O_F^\times]$ as in (\ref{eq:121}), and
$m_\grp(B)$ is the number of conjugacy classes of local optimal
embeddings at $\grp$ as in (\ref{eq:109}). A similar formula
\cite[Theorem~1.5]{xue-yang-yu:ECNF} holds for arbitrary $\Z$-orders
of full rank in $H$. When  $\calO=\O$ is maximal,  we have 
\begin{equation}
  \label{eq:170}
h(\O)=\frac{ h(F)|\zeta_F(-1)| }{2^{([F:\Q]-1)}} \prod_{\grp
  | \grd(H) } 
  (N(\grp)-1)+\frac{1}{2}\sum_{B\in
  \scrB}h(B)(1-w(B)^{-1})\prod_{\grp\vert \grd(H)}\left(1-\Lsymb{B}{\grp}\right).
\end{equation}
Applying (\ref{eq:168}) for  real quadratic fields $F$, Vign\'eras obtains explicit class
number formulas for all Eichler orders of square free level in 
\cite[Theorem~3.1]{vigneras:ens}.  

Comparing (\ref{eq:170}) with (\ref{eq:42}), we see that
(\ref{eq:114}) is more direct than (\ref{eq:171}) since it avoid
calculating the right hand side of (\ref{eq:42}) twice. However, the
two approaches share the same theoretical root (\ref{eq:41}), so the
difference is merely superficial.
\end{sect}

\section{Minimal $G$-orders}
\label{sec:minimal-g-orders}
\numberwithin{thmcounter}{subsection}

Throughout this section,  $F$ is a totally real field, and $H$
is a totally definite quaternion $F$-algebra. We mainly focus on
the case that $F=\Q(\sqrt{d})$ is a
real quadratic field with square-free $d\geq 6$. The cases $d\in \{2, 3,5\}$ are
treated separately in Section~\ref{case-p-2-3-5}.  The fundamental unit of
$F=\Q(\sqrt{d})$ is denoted by $\varepsilon$.  By definition, 
$\varepsilon>1$ for the canonical embedding  $F\hookrightarrow\R$.



\subsection{Vign\'eras unit index of $\calO$ and the fundamental unit
  of $F=\Q(\sqrt{d})$}
For an $O_F$-order $\calO$ in $H$, we set
$\calO^1:=\{u\in \calO^\times \mid \Nr(u)=1\}$. It is known \cite[Section~V.1]{vigneras} that
$\calO^1$ is a \textit{finite} normal subgroup of
$\calO^\times$.  Vign\'eras shows in
\cite[Theorem~6]{vigneras-crelle-1976} that 
$[\calO^\times:O_F^\times \calO^1]\in\{1,2,4\}$ for any arbitrary
totally real field $F$, so we call it  the \emph{Vign\'eras unit index} of $\calO$. 

\begin{lem}
  If $F=\Q(\sqrt{d})$ is a real quadratic field, then
  $[\calO^\times: O_F^\times \calO^1]\in \{1,2\}$. Moreover, if
  $\Nm_{F/\Q}(\varepsilon)=-1$, then $\calO^\times=O_F^\times \calO^1$.
\end{lem}

\begin{proof}
  Since $H$ is totally definite, we have
  $O_F^{\times 2}\subseteq \Nr(\calO^\times)\subseteq O_{F, +}^\times$, the group of totally
  positive units in $O_F^\times$. This gives rise to an embedding 
\begin{equation}
  \label{eq:3}
  \calO^\times/(O_F^\times\calO^1)\hookrightarrow O_{F, +}^\times/O_F^{\times
  2}. 
\end{equation}
When $F$ is a real quadratic field, the fundamental unit
$\varepsilon\in O_F^\times$ is totally positive if and only if
$\Nm_{F/\Q}(\varepsilon)=1$. We have
\begin{equation}
  \label{eq:134}
 O_{F,+}^\times=
\begin{cases}
  \dangle{\varepsilon} & \quad \text{if } \Nm_{F/\Q}(\varepsilon)=1,\\
  \langle\varepsilon^2\rangle & \quad \text{if }
  \Nm_{F/\Q}(\varepsilon)=-1, 
\end{cases} \qquad \text{while}\quad   O_F^{\times
  2}=\dangle{\varepsilon^2}. 
\end{equation}
The lemma follows directly from (\ref{eq:3}). 
\end{proof}

We set $\scrS=\{1, \varepsilon\}$ if $\Nm_{F/\Q}(\varepsilon)=1$, and
$\scrS=\{1\}$ otherwise. Thus $\scrS$ is a complete set of
representatives of $O_{F, +}^\times/O_F^{\times 2}$ for $F=\Q(\sqrt{d})$.




\begin{rem}
It is well-known \cite[Section~11.5]{Alaca-Williams-intro-ANT} that for $F=\Q(\sqrt{d})$, 
\begin{itemize}
\item $\Nm_{F/\Q}(\varepsilon)=1$ if $d$ is divisible by a prime $p\equiv
3\pmod{4}$;
\item $\Nm_{F/\Q}(\varepsilon)=-1$ if $d=p$ with a prime $p\equiv 1\pmod{4}$, 
  or $d=2p$ with a prime $p\equiv 5\pmod{8}$, or $d=p_1p_2$ with primes
  $p_1, p_2\equiv 1\pmod{4}$ and $\Lsymb{p_1}{p_2}=\Lsymb{p_2}{p_1}=-1$. 
\end{itemize}
Before the discussion leads us too far astray, we refer to the
classes of $d$ with known signs of $\Nm_{F/\Q}(\varepsilon)$ in
Corollaries~21.10 and 24.5 of \cite{Conner-Hurrelbrink}.
\end{rem}

\subsection{Structure of the reduced unit group of a CM-extension $K/F$}
\label{subsec:structure-CM-ext}
We first return to the general case where $F$ is an arbitrary totally
real number field. Let $K/F$ be a CM-extension, and $\bmu(K)$ be the group of
roots of unity in $K$. By
\cite[Theorem~4.12]{Washington-cyclotomic}, the Hasse unit index
  \begin{equation}
    \label{eq:103}
Q_{K/F}:=[O_K^\times: O_F^\times\bmu(K)]   
  \end{equation}
  is either 1 or 2.
  Following \cite[Section~13]{Conner-Hurrelbrink}, we call $K/F$ a
  \textit{CM-extension of type I} if $Q_{K/F}=1$, and of
  \textit{type II} if $Q_{K/F}=2$. Recall that an $O_F$-order in a
  CM-extension of $F$ is called a \emph{CM
$O_F$-order}.


  \begin{lem}\label{lem:cyclic-red-gp-CM-fields}
The reduced unit group  $O_K^\bs:=O_K^\times/O_F^\times$ is a cyclic
group of order
    $Q_{K/F}\abs{\bmu(K)}/2$. In particular,  $B^\bs:=B^\times/O_F^\times$ is
a cyclic subgroup of $O_K^\bs$ for every CM $O_F$-order $B$. 
  \end{lem}
  \begin{proof}
    Note that $\bmu(K)/\{\pm 1\}\subseteq O_K^\bs$ is  a cyclic subgroup 
    of index $Q_{K/F}\in \{1, 2\}$.  The lemma is
     true if one of the following conditions holds:
    \begin{itemize}
    \item $K/F$ is of type I;
    \item $K\neq F(\sqrt{-1})$ so that $\bmu(K)/\{\pm 1\}$ has \emph{odd}
       order. 
    \end{itemize}
    Now suppose that $K=F(\sqrt{-1})$ and $Q_{K/F}=2$. If
    $u\in O_K^\times$ represents an element of order $2$ in
    $O_K^\bs$, then it is purely imaginary,
    i.e.~$u\sqrt{-1}\in O_F^\times$. Thus $O_K^\bs$ contains a unique
    element of order $2$ represented by $\sqrt{-1}$, so it must be
    cyclic. 
  \end{proof}
  Since
  $[O_K^\bs: \bmu(K)/\{\pm 1\}]\in
  \{1,2\}$, if $\tilde{u}\in O_K^\bs$ has \emph{odd} order, then
  $\tilde{u}\in \bmu(K)/\{\pm 1\}$ and is represented by a root of
  unity.


From now on, we assume that $F=\Q(\sqrt{d})$ is a quadratic real field, where $d\in\bbN$ is a positive square free integer.  Since $[K:\Q]=4$,   the possible orders of
  $\bmu(K)$ are $2,4,6,8,10,12$ by
  \cite[Section~2.3]{xue-yang-yu:num_inv}. Moreover,
\begin{itemize}
\item $\abs{\bmu(K)}=8$ if and only if $F=\Q(\sqrt{2})$ and
  $K=\Q(\sqrt{2}, \sqrt{-1})$;
\item $\abs{\bmu(K)}=10$ if and only if $F=\Q(\sqrt{5})$ and
  $K=\Q(\zeta_{10})$; 
\item $\abs{\bmu(K)}=12$ if and only if $F=\Q(\sqrt{3})$ and
  $K=\Q(\sqrt{3},\sqrt{-1})$.
\end{itemize}
Here $\zeta_n$ denotes the primitive $n$-th root of unity $e^{2\pi i/n}$ for all positive $n\in \mathbb{N}$. 

Assume further that that $d\geq 6$ for the rest of Section~\ref{sec:minimal-g-orders}, so
$\abs{\bmu(K)}\in \{2, 4, 6\}$. Let $K/F$ be a CM-extension with
$[O_K^\times: O_F^\times]>1$. If $\bmu(K)=\{\pm 1\}$, then
$\Nm_{F/\Q}(\varepsilon)=1$ and $K=F(\sqrt{-\varepsilon})$ by 
\cite[Lemma~2.2]{xue-yang-yu:num_inv}. It follows that
$[O_K^\times:O_F^\times]>1$ if and only if
\begin{equation}
  \label{eq:6}
  K=\begin{cases}
     F(\sqrt{-1})\text{ or } F(\sqrt{-3}) &\quad \text{if } \Nm_{F/\Q}(\varepsilon)=-1;\\
     F(\sqrt{-1}), F(\sqrt{-\varepsilon}), \text{ or } F(\sqrt{-3})
     &\quad \text{if } \Nm_{F/\Q}(\varepsilon)=1.
   \end{cases}
\end{equation}
Clearly, $F(\sqrt{-\varepsilon})\neq F(\sqrt{-1})$ for all $F$. On
the other hand,  $F(\sqrt{-\varepsilon})= F(\sqrt{-3})$ if
$3\varepsilon\in F^{\times 2}$. 


By
Lemma~\ref{lem:cyclic-red-gp-CM-fields},
$O_K^\bs:=O_K^\times/O_F^\times$ is a cyclic group of order
$n\in \{2, 3, 4, 6\}$ and 
\[n=4 \text{ (resp.~$6$)} \quad \Leftrightarrow \quad K=F(\sqrt{-1})
\text{ (resp.~$F(\sqrt{-3})$) and } Q_{K/F}=2.\]
According to
\cite[Lemma~2]{MR0441914}, $F(\sqrt{-1})/F$ (resp.~$F(\sqrt{-3})/F$)
is of type II if and only if $2\varepsilon\in F^{\times 2}$ (resp.~$3\varepsilon\in F^{\times 2}$).  For example, we have
$2\varepsilon\in F^{\times 2}$ if $d=cp$ with $c\in \{1, 2\}$ and
$p\equiv 3\pmod{4}$ by \cite[Lemma~3]{MR1344833} or
\cite[Lemma~3.2(1)]{MR3157781}.  Similarly,
$3\varepsilon\in F^{\times 2}$ if $d=3p$ with $p\equiv 3\pmod{4}$ by
\cite[Lemma~3]{MR0441914}.  We introduce the following notation:
\begin{itemize}
\item   if   $2\varepsilon\in F^{\times 2}$, then we fix $\vartheta\in F^\times$
  such that $\varepsilon=2\vartheta^2$; 
\item if
  $3\varepsilon\in F^{\times 2}$, then we fix $\varsigma\in F^\times$
  such that $\varepsilon=3\varsigma^2$.
\end{itemize}
Note that in these two cases, $2$ (resp.~3) is ramified in $F/\Q$, and
$2\vartheta$ (resp.~$3\varsigma$) generates the unique prime of
$F$ over $2$ (resp.~$3$). Note that $2\varepsilon$ and
$3\varepsilon$ are simultaneously perfect squares in
$F^\times$ if and only if $d=6$.  Lastly, if
$\{2\varepsilon, 3\varepsilon\}\cap F^{\times 2}=\emptyset$, then
$\abs{O_K^\bs}\in \{2, 3\}$. For instance, if
$\Nm_{F/\Q}(\varepsilon)=-1$, then
$\{2\varepsilon, 3\varepsilon\}\cap F^{\times 2}=\emptyset$, and every CM-extension $K/F$ is of type I
by \cite[Lemma~13.6]{Conner-Hurrelbrink}. 

We select a special set of representatives of $O_K^\bs$. The
representative of the identity element is always chosen to be $1$. Let
$\tilde{u}\in O_K^\bs$ be a nontrivial element. After multiplying a
representative $u\in O_K^\times$ by a suitable element of
$O_F^\times$, we may assume that $n_u:=\Nm_{K/F}(u)$ belongs to the set
$\scrS\subseteq \{1, \varepsilon\}$. This determines $u$ uniquely up
to sign and we put $\Nm_{K/F}(\tilde{u}):=n_u$.   Let $P_u(x)=x^2-t_ux+n_u\in F[x]$
be the minimal polynomial  of $u$ over $F$. Changing $u$ to $-u$
switches the sign of $t_u$ and leaves $n_u$ invariant, so we put
$P_{\tilde{u}}(x)=x^2\pm t_ux+n_u$. In Table~\ref{tab:rep-elements},
we list all the possible $P_{\tilde{u}}(x)$ for each 
$\tilde{u}$.  The method is the
same as that of \cite[Lemma~2.2]{xue-yang-yu:num_inv}, hence omitted.

\begin{table}[htbp]
\centering
\caption{Minimal polynomials of $\tilde{u}$ when $d\geq 6$}
\label{tab:rep-elements}
\renewcommand{\arraystretch}{1.7}
\begin{tabular}{|>{$}c<{$}|>{$}c<{$}|>{$}c<{$}|>{$}c<{$}|}
\hline
\ord(\tilde{u}) & \text{Conditions} & P_{\tilde{u}}(x)\in F[x] &
                                                                 \text{roots
                                                                 of
                                                                 }P_{\tilde{u}}(x)
  \\
\hline
  \multirow{2}{*}{2}& \Nm_{K/F}(\tilde{u})=1 & x^2+1 & \pm \sqrt{-1}\\ \cline{2-4}
                  & \Nm_{F/\Q}(\varepsilon)=1, 
                    \Nm_{K/F}(\tilde{u})=\varepsilon& x^2+\varepsilon &
                                                                  \pm \sqrt{-\varepsilon}\\ \hline
3 & & x^2\pm x+1  & \pm \zeta_{3}^{\pm 1}\\ \hline
4& 2\varepsilon\in F^{\times 2}
                                    &x^2\pm 2\vartheta x+\varepsilon &
                                                                       \pm \sqrt{\pm \varepsilon\sqrt{-1}}=\pm
                                                                       \vartheta(1\pm
  \sqrt{-1})\\ 
\hline 
6& 3\varepsilon\in F^{\times 2}
                                    &x^2\pm 3\varsigma x+\varepsilon &
                                                                       \pm\sqrt{\varepsilon\zeta_6^{\pm
                                                                       1}}=\pm
                                                                       \varsigma
                                                                       \zeta_6^{\pm 1}\sqrt{-3}\\
                                                                       
\hline
\end{tabular}
\end{table}



\subsection{Finite noncyclic subgroups of $H^\times/O_F^\times$}
We keep the assumption that $F=\Q(\sqrt{d})$ with $d\geq 6$.  For every
$O_F$-order $\calO$ in $H$, the reduced unit group
$\calO^\bs=\calO^\times/O_F^\times$ is a \emph{finite} subgroup of
$H^\times/O_F^\times$, so we study the finite subgroups of
$H^\times/O_F^\times$. 

Let $\tilde{u}\in H^\times/O_F^\times$ be a nontrivial element. For
any $u\in H^\times$ representing $\tilde{u}$, the subfield
$F(u)\subset H$ depends only on $\tilde{u}$ and is denoted by
$F(\tilde{u})$. Since $H$ is totally definite, $F(\tilde{u})/F$ is a
CM-extension.  Note that $\tilde{u}$ has finite order if and only if
$u$ is a unit of $O_{F(\tilde{u})}$.




\begin{defn}\label{defn:minimal-g-orders}
Let $G$ be a finite noncyclic subgroup of $H^\times/O_F^\times$.  The
$O_F$-submodule $\calO$ spanned by a complete set of representatives of $G$
is an $O_F$-order in $H$ and called the \emph{minimal} $O_F$-order
attached to $G$,  or \emph{minimal $G$-order} for short. 
\end{defn}
Clearly, the minimal $G$-order $\calO$ does not depend on the choice of
representatives, and $G$ is naturally a subgroup of $\calO^\bs$.
We will show that in fact $G=\calO^\bs$ in Corollary~\ref{cor:units-of-minimal-orders}. For the
moment, let us note that $\ord(\tilde{u})\in \{2, 3, 4, 6\}$ for every
nontrivial $\tilde{u}\in G$ by 
Section~\ref{subsec:structure-CM-ext}. Applying the classification in
Section~\ref{sect:general-classification-grp}, we write down all possible finite
noncyclic subgroups of $H^\times/O_F^\times$ up to isomorphism in 
Table~\ref{tab:non-cyclic-red-gp}.  
\begin{table}[htbp]
 \centering
 \caption{Finite noncyclic  subgroups of $H^\times/O_F^\times$}\label{tab:non-cyclic-red-gp}
\renewcommand{\arraystretch}{1.3}
\begin{tabular}{|>{$}c<{$}|>{$}c<{$}|>{$}c<{$}|}
\hline
  F=\Q(\sqrt{d}), \ d\geq 6 & \ord(\tilde{u})>1 & \text{possible noncyclic
                                                       }G \\ \hline
  2\varepsilon\in F^{\times  2} & 2, 3, 4 & D_2, D_3,
                                              D_4,   A_4, S_4 \\ \hline
  3\varepsilon\in F^{\times  2} & 2, 3, 6 & D_2, D_3, D_6,
                                                 A_4  \\ \hline
\{2\varepsilon, 3\varepsilon\}\cap F^{\times  2}=\emptyset    & 2, 3 & D_2, D_3,  A_4 \\ 
\hline
\end{tabular}
\end{table}


\begin{defn}\label{def:strict-isom}
  For $s\in\{1,2\}$, let $H_s$ be a totally definite quaternion
  $F$-algebra, and $G_s$ be a finite noncyclic subgroup of
  $H_s^\times/O_F^\times$.  We say $G_1$ and $G_2$ are \emph{strictly
    isomorphic} if there is an $F$-isomorphism of quaternion algebras
\begin{equation}
  \label{eq:142}
\psi: H_1\to H_2
\quad \text{such that}\quad \widetilde{\psi}(G_1)=G_2, 
\end{equation}
where $\widetilde{\psi}: H_1^\times/O_F^\times\to H_2^\times/O_F^\times$
denotes the induced isomorphism. 
\end{defn}

By transport of structure, strictly isomorphic groups have isomorphic
 minimal $O_F$-orders. We characterize strict isomorphisms using the reduced norm map. 
 

For every finite subgroup $G$ in $H^\times/O_F^\times$,
the reduced norm map induces a
homomorphism 
\begin{equation}
  \label{eq:136}
\Nr: G\to
O_{F,+}^\times/O_F^{\times 2}. 
\end{equation}
Recall that $O_{F,+}^\times/O_F^{\times 2}$ is represented by
$\scrS$, which is either  $\{1, \varepsilon\}$ or $\{1\}$ depending on
whether $\Nm_{F/\Q}(\varepsilon)=1$ or not. Given $\tilde{u}\in G$, 
we may always
pick a representative $u\in O_{F(\tilde{u})}^\times$ of $\tilde{u}$
such that $\Nr(u)=\Nm_{F(\tilde{u})/F}(u)\in \scrS\subseteq \{1, \varepsilon\}$. Subsequently,
\emph{representatives} of $\tilde{u}$ refer exclusively to these ones.
We set $\Nr(\tilde{u}):=\Nr(u)$ and call it the \emph{reduced norm} of
$\tilde{u}$.  If $\widetilde{\psi}: G_1\to G_2$ is a strict
isomorphism as in (\ref{eq:142}), then clearly $\Nr\circ
\widetilde{\psi}=\Nr$.



\begin{defn}\label{def-of-kinds}
   Let $G$ be a finite noncyclic subgroup of $H^\times/O_F^\times$
   isomorphic to $D_n$ with $n\in \{2, 3\}$. We say
     $G$ is of the \emph{first kind} and write $G\simeq D_n^\dagger$  if every element of order $2$ in $G$
     has reduced norm $1$; otherwise, we say $G$ is of the
     \emph{second kind} and write $G\simeq D_n^\ddagger$.  
\end{defn}
\begin{rem}\label{rem:elements-of-order-2}
We explain why Definition~\ref{def-of-kinds} applies only to
$D_2$ and $D_3$ but not to any other groups in
Table~\ref{tab:non-cyclic-red-gp}. Clearly,  the reduced norm map in (\ref{eq:136}) is constant on each conjugacy
class of $G$. 
 If $G$ contains an element of
order $4$ or $6$, then $\scrS=\{1,
\varepsilon\}$ and $\Nr$ is surjective by
Table~\ref{tab:rep-elements}. 
We separate the rest of the discussion into cases: 

\begin{enumerate}
\item If $G\simeq A_4$, then $\Nr(G)=\{1\}$ since $A_4$ contains no
  subgroup of index $2$. 

\item Suppose that $G\simeq S_4$. Any two isomorphisms $G\simeq S_4$
  differ by a conjugation because all automorphisms of $S_4$ are
  inner. In particular, it makes sense to talk about \emph{cycle
    types} \cite[Section~2.3.1]{wilson-fin-simple-gp} of elements of $G$. There are two conjugacy classes of
  elements of order $2$ in $S_4$:
  \begin{itemize}
  \item the transpositions, and
    \item the permutations of  type $(2,2)$.
  \end{itemize}
The group $G$ has a unique subgroup $G'$ of index $2$ which
corresponds to $A_4$ for every isomorphism $G\simeq S_4$. We have $\ker(\Nr)=G'$ since $\Nr$ maps
  surjectively onto $O_{F,+}^\times/O_F^{\times 2}\simeq \zmod{2}$ and
  $\ker(\Nr)\supseteq G'$. 
 It follows that every transposition has reduced norm
  $\varepsilon$, and every permutation of type $(2,2)$ has reduced
  norm $1$. 
\item Suppose that $G\simeq D_n$ with $n\in \{2,3,4, 6\}$. We present
  $G$ as in (\ref{eq:130}): 
  \begin{equation}
    \label{eq:137}
G=\dangle{\tilde\eta, \tilde u\in G \mid
  \ord(\tilde\eta)=n, \ord(\tilde u)=2, \tilde\eta\tilde  u=\tilde u
    \tilde\eta^{-1}}.     
  \end{equation}
  First consider the case $n=2m$ with $m\in \{2,3\}$. There are 3
  conjugacy classes of elements of order $2$ in $G$:
\begin{itemize}
\item $\{\tilde{\eta}^m\}$, consisting of the unique nontrivial
  element of the center;
\item $\{\tilde{u}\tilde{\eta}^{2t}\mid 0\leq t\leq m-1\}$; 
\item $\{\tilde{u}\tilde{\eta}^{2t+1}\mid 0\leq t\leq m-1\}$.
\end{itemize}
We have $\Nr(\tilde{u})\neq \Nr(\tilde{u}\tilde{\eta})$ since
$\Nr(\tilde{\eta})=\varepsilon$.   Replacing
$\tilde{u}$ by $\tilde{u}\tilde{\eta}$ if necessary, we may and will assume
that 
\begin{equation}
  \label{eq:138}
\Nr(\tilde{u})=1\quad \text{and}\quad
\Nr(\tilde{u}\tilde{\eta})=\varepsilon. 
\end{equation}

If $G\simeq D_3$, then there is a unique conjugacy class of elements of
$2$, which has a uniform reduced norm (either $1$ or
$\varepsilon$). These two cases are distinguished by the notation
$D_3^\dagger$ and $D_3^\ddagger$.

Lastly, suppose that $G\simeq D_2$,  the Klein 4-group. We
write $G\simeq D_2^\dagger$ if every element of $G$ has reduced
norm $1$. If $G\not\simeq D_2^\dagger$, then there is an element with
reduced norm $\varepsilon$, which we choose as $\tilde{\eta}$. Once
again $\Nr(\tilde{u})\neq \Nr(\tilde{u}\tilde{\eta})$. Replacing
$\tilde{u}$ by $\tilde{u}\tilde{\eta}$ if necessary, we always assume
that (\ref{eq:138}) holds true in this case as well. 
\end{enumerate}
\end{rem}



\begin{prop}\label{prop:uniqueness-finit-subgp}
For $s\in\{1,2\}$, let $H_s$ and $G_s$ be as in
Definition~\ref{def:strict-isom}. Then $G_1$ and $G_2$ are strictly
isomorphic if and only if they are isomorphic as abstract groups and are of the same
kind (if applicable).   
\end{prop}
\begin{proof}
  We first prove the proposition under the
assumption that $G_s\simeq D_n$ with $n\in \{2,3,4,6\}$. Presented $G_s$ as in
(\ref{eq:137}),  and assume the generators $\{\tilde{\eta}_s,
\tilde{u}_s\}_{s=1,2}$ satisfy (\ref{eq:138}) if $G_s\simeq
D_2^\ddagger$ or $G_s\simeq D_n$ with $n\in \{4, 6\}$. By the assumption, we have $\ord(\tilde{\eta}_1)=\ord(\tilde{\eta}_2)$
and $\Nr(\tilde{\eta}_1)=\Nr(\tilde{\eta}_2)$. Let
$K_s:=F(\tilde{\eta}_s)$ for $s=1,2$. 
Thanks to
Table~\ref{tab:rep-elements}, there is an
$F$-isomorphism of CM-fields 
\begin{equation}
  \label{eq:140}
\psi_0: K_1\to
K_2 \quad \text{such that}\quad
\widetilde{\psi}_0(\tilde{\eta}_1)=\tilde{\eta}_2. 
\end{equation}
Let $u_s\in H_s^\times$ be a representative of $\tilde{u}_s$ and put
$c:=u_1^2=u_2^2\in \{-1, -\varepsilon\}$. We have
$u_sK_s u_s^{-1}=K_s$ since $\tilde{u}_s\tilde\eta_s\tilde{u}_s^{-1}=
    \tilde{\eta}_s^{-1}$. Hence conjugation by $u_s$ induces the unique
\textit{nontrivial} $F$-automorphism $y_s\mapsto \bar{y}_s$ on
$K_s$. It follows that  $H_s=K_s+K_su_s\simeq
\{K_s, c\}$ (notation as in the start of Section~\ref{sec:prelimiary}). 
Thus $\psi_0$
extends to an $F$-isomorphism 
\begin{equation}
  \label{eq:141}
\psi: H_1\to H_2\quad\text{such that}\quad \psi(u_1)=u_2.   
\end{equation}


Next, assume that $G_1$ and $G_2$ are both isomorphic to $A_4$. By
Remark~\ref{rem:elements-of-order-2}, we have
$A_4 \supset D_2^\dagger$. Applying the preceding proof to $D_2^\dagger$, we may identify
both $H_s$ with $H:=\{F(\sqrt{-1}), -1\}=\qalg{-1}{-1}{F}$ such that the
unique Sylow $2$-subgroup of $G_s$ is identified with $V:=\{\tilde{1},
\tilde{i}, \tilde{j}, \tilde{k}\}\subseteq H^\times/O_F^\times$. Put 
\begin{equation}\label{eq:108}
\xi:=(1+i+j+k)/2\in \qalg{-1}{-1}{\Q}\subset \qalg{-1}{-1}{F}.  
\end{equation}
The subgroup $G_0:=\dangle{\tilde{i}, \tilde{j}, \tilde{\xi}}\subset H^\times/O_F^\times$ is isomorphic to $A_4$. Let $G\subset
H^\times/O_F^\times$ be a subgroup such that $G\supset V$ and $G\simeq
A_4$. We show that $G=G_0$.
 The inner automorphisms of $G$
induce an embedding
\[G/V\hookrightarrow \Aut(V)\simeq S_3, \]
which identifies $G/V$ with $A_3\subset S_3$.  In
particular, there exists an element $\tilde{\xi}'\in G$ of order $3$
such that conjugation by $\tilde{\xi}'$ induces the cyclic permutation
$(\tilde{i}, \tilde{j}, \tilde{k})$. On the other hand, conjugation by $\tilde{\xi}$ also acts as $(\tilde{i},
\tilde{j}, \tilde{k})$. It follows that $\xi^{-1}\xi'$ commutes with
$i, j, k$ up to sign, and hence $\xi^{-1}\xi'\in F^\times\dangle{i, j,
  k}$, the subgroup of $H^\times$ generated by $F^\times$ and $\{i,j,k\}$. 
By Table~\ref{tab:rep-elements}, $\Nr(\xi')=1$ since
$\ord(\tilde{\xi}')=3$. We have $\xi^{-1}\xi'\in \{\pm 1, \pm
i, \pm j, \pm k\}$ and hence $\tilde{\xi}\in G$. This concludes the
proof for $A_4$.

Lastly, suppose that $G_1$ and $G_2$ are both isomorphic to $S_4$. By
Table~\ref{tab:non-cyclic-red-gp}, we have
$2\varepsilon\in F^{\times 2}$.  Since $S_4\supset A_4$, we may
identify each $H_s$ with $H=\qalg{-1}{-1}{F}$ in such a way that the unique
index 2 subgroup of $G_s$ is identified with $G_0\simeq A_4$ as above.  
It remains to show that there is a unique subgroup of
$H^\times/O_F^\times$ containing $G_0$ and isomorphic to
$S_4$.  Let $G$ be such a group. Then $\tilde{i}\in V\subset G$ is the
square of an element of order $4$, say $\tilde{\tau}\in G$.  We then
have $F(\tilde{\tau})=F(i)$ and $\tilde{\tau}$ is represented by
either $\vartheta(1+i)$ or $\vartheta(1-i)$ according to
Table~\ref{tab:rep-elements}. These two elements are mutually inverse
to each other modulo $O_F^\times$. Thus $G$ is generated by $G_0$ and
$\vartheta(1+i)O_F^\times$, hence uniquely determined. 
\end{proof}

In the case $G\simeq S_4$ and $G\supset G_0$, the same argument also shows that 
$\vartheta(1+j)O_F^\times\in G$ as well. This will be used when we
write down the minimal $G$-order in Section~\ref{sec:minimal-g-order-table}.

By Table~\ref{tab:non-cyclic-red-gp} and Proposition~\ref{prop:uniqueness-finit-subgp}, the strict
isomorphism classes of all possible 
finite
noncyclic subgroups of $H^\times/O_F^\times$ for all $H$ are
\begin{equation}
  \label{eq:139}
\calG^\natural:=\{D_2^\dagger, D_2^\ddagger, D_3^\dagger,
D_3^\ddagger, D_4, D_6, A_4, S_4\}.  
\end{equation}
Given two distinct groups $G, G'\in \calG^\natural$, we write $G< G'$ if $G$ is realizable as a subgroup of
$G'$. The relationships between these groups are illustrated in
Figure~\ref{fig:rel-gps}, where every pair $G< G'$ is
connected by a line. 
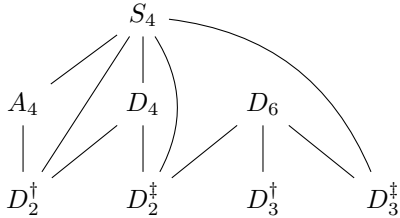
\begin{figure}[htbp]
 \centering
 \caption{Relationships between the groups in
   $\calG^\natural$}\label{fig:rel-gps}
\bigskip
\begin{tikzcd}
  & S_4\ar[ld, dash]\ar[d, dash]\ar[rrdd, dash, bend left] \ar[ldd,
  dash]\ar[dd,dash, bend left] &  & \\
A_4\ar[d, dash] & D_4\ar[d, dash]\ar[ld, dash] & D_6 \ar[rd, dash]\ar[d, dash]\ar[ld, dash] & \\
D_2^\dagger & D_2^\ddagger &D_3^\dagger & D_3^\ddagger 
\end{tikzcd}
\end{figure}

The diagram makes heavy use of Remark~\ref{rem:elements-of-order-2}.
For example, $D_3^\dagger$ is not realizable as a subgroup of
$S_4$. Otherwise $D_3^\dagger< A_4$, but $A_4$
does not contain any subgroup of index $2$.  The group $D_6$ has
a unique conjugacy class of subgroups of order $4$, namely its Sylow
$2$-subgroups. Each Sylow 2-subgroup contains the unique nontrivial
element of the center, which has reduced norm
$\varepsilon$. Hence $D_6$ contains subgroups strictly isomorphic to
$D_2^\ddagger$ but none to $D_2^\dagger$. Lastly, let $\{\tilde{u}, \tilde{\eta}\}$ be
  generators of $D_4$ 
as in (\ref{eq:137}) with $\Nr(\tilde{u})=1$. Since
$\Nr(\tilde{\eta}^2)=1$, we have  $\dangle{\tilde{u},
  \tilde{\eta}^2}\simeq D_2^\dagger$ and  $\dangle{\tilde{u}\tilde{\eta},
  \tilde{\eta}^2}\simeq D_2^\ddagger$. We leave the discussion of the
remaining cases to the interested readers. 

\begin{lem}\label{lem:strict-isom-subgp-conj}
  Let $G$ be a finite noncyclic subgroup of $H^\times/O_F^\times$,
  and $G_1, G_2$ be two noncyclic subgroups of $G$. If $G_1$ and
  $G_2$ are strictly isomorphic, then they are conjugate in $G$. 
\end{lem}
\begin{proof}
Suppose that $G_1$ and $G_2$ are strictly isomorphic. First, we have $G_1=G_2$ in the following cases: 
\begin{itemize}
\item $G\simeq A_4, D_4$; 
\item $G\simeq S_4$ and $G_s\simeq A_4$ (or $D_2^\dagger$) for   $s=1,2$; 
\item $G\simeq D_6$ and $G_s\simeq D_3^\dagger$ (or $D_3^\ddagger$)
  for  $s=1,2$.
\end{itemize}
Next, if $G\simeq S_4$
(resp.~$D_6$) and $G_s\simeq D_4$ (resp.~$D_2^\ddagger$), then both
$G_s$ are Sylow $2$-subgroups of $G$, and hence conjugate. It remains
to handle the cases that $G\simeq S_4$, and $G_s\simeq D_2^\ddagger$ or
$D_3^\ddagger$. Note that a noncyclic order $4$ subgroup of $S_4$ 
is strictly
isomorphic to $D_2^\ddagger$ if and only if it is not normal, and such
subgroups form a single conjugacy class. There is a single
conjugacy class of subgroups of order $6$ in $S_4$, and each such subgroup is
strictly isomorphic to $D_3^\ddagger$. 
\end{proof}


\subsection{Structure of the minimal $G$-orders for noncyclic finite
  groups $G$}
\label{sec:minimal-g-order-table}

\begin{defn}\label{def:min-g-ord-up-isom}
For each $G\in \calG^\natural$ as in \eqref{eq:139}, an $O_F$-order $\calO\subset H$ is
called a 
\emph{minimal $G$-order} if there exists $G'\subset
H^\times/O_F^\times$ such that $G'\simeq G$ and $\calO$ is a minimal
$G'$-order.  
\end{defn}
By Proposition~\ref{prop:uniqueness-finit-subgp}, a minimal
$G$-order is uniquely determined up
to $O_F$-isomorphism.  The proof of
Proposition~\ref{prop:uniqueness-finit-subgp} also provides a method
for writing down a minimal $G$-order explicitly. For example, if $G$ is
dihedral of order $2n$ and presented as in (\ref{eq:137}), then the
$O_F$-order 
\begin{equation}
  \label{eq:44}
\scrO_n:=O_F+O_Fu+O_F\eta+O_Fu \eta,
\end{equation}
is a minimal $G$-order. Take
$n=6$. Then by (\ref{eq:44}) and Table~\ref{tab:rep-elements}, 
  \begin{equation}\label{eq:63}
\scrO_6=O_F+O_Fi+O_F\varsigma
  (3+j)/2+O_F\varsigma i(3+j)/2\subset \qalg{-1}{-3}{F},    
  \end{equation}
where $3\varsigma^2=\varepsilon$.  Since
$O_{F(j)}=O_F+O_F\varsigma (3+j)/2$ by Lemma~\ref{lem:ord-6-max-ord}, we may  write
\begin{equation}
  \label{eq:143}
  \scrO_6=O_{F(j)}+iO_{F(j)}. 
\end{equation}

As before, we use superscripts $\dagger$ or $\ddagger$ to distinguish the two kinds,
e.g.~$\scrO_2^\dagger$ denotes a minimal $D_2^\dagger$-order.
Let 
\begin{equation}
  \label{eq:149}
H_{2,\infty}=\qalg{-1}{-1}{\Q}, \qquad 
H_{3, \infty}=\qalg{-1}{-3}{\Q},
\end{equation}
 and $\xi\in H_{2, \infty}$ be the element in \eqref{eq:108}. We define
\begin{align}
  \gro_2&:=\Z[i,j]=\Z+\Z i+\Z j+\Z k\subset
          H_{2,\infty},\\
  \bbo_2&:=\Z+\Z i+\Z j+\Z \xi\subset H_{2,\infty},\label{eq:146}\\
  \gro_3&:=\Z[i, (1+j)/2]=\Z+\Z
  i+\Z\frac{1+j}{2}+\Z\frac{i(1+j)}{2}\subset H_{3, \infty}.\label{eq:60}
\end{align}
Then $\scrO_n^\dagger=\gro_n\otimes_\Z O_F$ for $n=2,3$, and
$\calO_{12}:=\bbo_2\otimes_\Z O_F$ is a minimal $A_4$-order. 
  It is well-known that both $\bbo_2$ and $\gro_3$ are maximal orders in their
respective quaternion $\Q$-algebras, and $[\bbo_2:\gro_2]=2$.

By the proof of Proposition~\ref{prop:uniqueness-finit-subgp}, if
$H^\times/O_F^\times$ contains a subgroup isomorphic to $D_4$ or $S_4$, then 
$2\varepsilon\in F^{\times 2}$ and $H$ may be identified with
$\qalg{-1}{-1}{F}$.  Fix $\vartheta\in F^\times$ such that
$\varepsilon=2\vartheta^2$ as in Section~\ref{subsec:structure-CM-ext}. If 
$\alpha\in H$ satisfies $\alpha^2=-1$, then we put
$\sqrt{\varepsilon\alpha}=\vartheta(1+\alpha)$ since $(\pm
\vartheta(1+\alpha))^2=\varepsilon \alpha$. Thus 
\begin{equation}
  \label{eq:145}
  \scrO_4=O_F+O_Fi+O_F\sqrt{\varepsilon
    j}+O_Fi\sqrt{\varepsilon j}. 
\end{equation}
Straightforward calculation shows that 
\begin{equation}
  \label{eq:144}
\O_{24}:=O_F+O_F\sqrt{\varepsilon i}+O_F\sqrt{\varepsilon j}+O_F\xi  
\end{equation}
is closed under multiplication,   and hence an $O_F$-order. Note that
$\sqrt{\varepsilon k}=(2\vartheta-\sqrt{\varepsilon j})\xi\in
\O_{24}$. Clearly,
$\O_{24}$ contains the minimal $A_4$-order $\calO_{12}$.  Thus it is
 a minimal $S_4$-order since $\sqrt{\varepsilon i}\in \O_{24}$. See
Definition-Lemma~\ref{defn:O12-O4I} for a simpler
 presentation of $\O_{24}$. 

The deduction of the explicit forms of minimal $O_F$-orders attached
to $D_2^\ddagger$ or
$D_3^\ddagger$ are left to the interested reader.
 In Table~\ref{tab:minimal-G-orders}, we list for each $G\in \calG^\natural$ a
representative in the unique isomorphism class of minimal $G$-orders
and calculate its discriminant (see
Lemmas~\ref{lem:gen-lem-scro2II} and \ref{lem:gen-lem-scro3II}).
The existence of a minimal $G$-order determines $H$
 uniquely up to isomorphism.


\begin{table}[htbp]
\centering
\caption{Minimal $G$-orders for $d\geq 6$}\label{tab:minimal-G-orders}
\renewcommand{\arraystretch}{1.5}
\begin{tabular}{|>{$}c<{$}|>{$}c<{$}|>{$}c<{$}|>{$}c<{$}|>{$}c<{$}|}
\hline
G & \text{condition on }\varepsilon & H & \text{minimal $G$-order } & \grd(\calO)\\
\hline
D_2^\dagger& & \qalg{-1}{-1}{F} & \scrO_2^\dagger:=O_F[i,j] & 4O_F\\
\hline
D_2^\ddagger&\Nm_{F/\Q}(\varepsilon)=1 & \qalg{-1}{-\varepsilon}{F} &
                                                                     \scrO_2^\ddagger:=O_F[i,j]
                            & 4O_F\\
\hline
D_3^\dagger& & \qalg{-1}{-3}{F} & \scrO_3^\dagger:=O_F[i,(1+j)/2]
                            & 3O_F\\
\hline
D_3^\ddagger&\Nm_{F/\Q}(\varepsilon)=1 & \qalg{-\varepsilon}{-3}{F} & \scrO_3^\ddagger:=O_F[i,(1+j)/2]
                            & 3O_F\\
\hline
D_4 & 2\varepsilon\in F^{\times 2} & \qalg{-1}{-1}{F}
                    &\scrO_4:=O_F+O_Fi+O_F\sqrt{\varepsilon j}+O_Fi\sqrt{\varepsilon j} & 2O_F\\
\hline
D_6 & 3\varepsilon\in F^{\times 2} &\qalg{-1}{-3}{F} & \scrO_6:=O_{F(j)}+iO_{F(j)}& O_F\\
\hline
A_4 &  & \qalg{-1}{-1}{F}&\calO_{12}:=O_F+O_Fi+O_Fj+O_F\xi & 2O_F\\
\hline 
S_4 & 2\varepsilon\in F^{\times 2} &
                                     \qalg{-1}{-1}{F}&\O_{24}:=O_F+O_F\sqrt{\varepsilon
                                                       i}+O_F\sqrt{\varepsilon j}+O_F\xi
                                                   & O_F\\
\hline
\end{tabular}
\end{table}

We make a few observations based on
Table~\ref{tab:minimal-G-orders}. 
\begin{cor}\label{cor:mini-D6-S4-maximal}
The minimal $O_F$-orders  attached to $D_6$ and $S_4$  are 
 maximal. 
\end{cor}
\begin{proof}
Let $\O$ be either $\scrO_6$ or $\O_{24}$. We have $\grd(\O)=O_F$, and
hence $\O$ is maximal. 
\end{proof}
\begin{cor}\label{cor:quat-alg-with-non-cyclic-gp}
   If $H$ is not isomorphic to any of the
  quaternion algebras in Table~\ref{tab:minimal-G-orders}, (e.g.~$\grd(H)\nmid 6O_F$), then
  $\calO^\bs$ is \emph{cyclic} for every  $O_F$-order
  $\calO\subset H$.
\end{cor}

\begin{proof}
  If $G:=\calO^\bs$ is noncyclic, then $\calO$ contains a
  minimal $G$-order. Thus $H$ is necessarily isomorphic to
  one of the quaternion algebras listed in
  Table~\ref{tab:minimal-G-orders}.
\end{proof}

\begin{cor}\label{cor:units-of-minimal-orders}
If $\calO$ is the minimal $O_F$-order attached to a finite noncyclic subgroup $G$ of
$H^\times/O_F^\times$, then $\calO^\bs=G$.  

\end{cor}
\begin{proof}
Let $G':=\calO^\bs$. 
Clearly, $G\subseteq G'$, and $\calO$ is also 
a minimal $G'$-order since by definition $\calO$ is spanned over $O_F$ by
representatives of $G\subseteq G'$. 
  In particular, minimal $O_F$-orders attached to $G$ and
  $G'$ have
  the same discriminants. Combining Figure~\ref{fig:rel-gps} and
Table~\ref{tab:minimal-G-orders}, we see that this is possible only if
$G=G'$. 
\end{proof}

By Corollary~\ref{cor:units-of-minimal-orders},  if $\calO$ is a minimal $O_F$-order attached to a noncyclic group $G\in \calG^\natural$, then the Vign\'eras unit index 
\begin{equation}
    [\calO^\times: O_F^\times\calO^1]=\begin{cases}
      1 \qquad \text{if } G\in \{D_2^\dagger, D_3^\dagger, A_4\},\\
    2 \qquad \text{if } G\in \{D_2^\ddagger, D_3^\ddagger, D_4, S_4, D_6\}.\\
    \end{cases}
\end{equation}

\subsection{Normalizers of minimal orders}
\label{subsec:norm-minim-orders}
\begin{lem}
Let $\calO$ be one of the minimal $G$-orders in
Table~\ref{tab:minimal-G-orders}.  The kernel of the natural
homomorphism $\varphi:   \calN(\calO)\to \Aut(\calO^\bs)$ is given as follows
\[\ker(\varphi)=
\begin{cases}
  F^\times\calO^\times &\quad \text{if }  \calO^\bs\simeq D_2;\\
  F^\times\<j\>&\quad \text{if }  \calO^\bs\simeq D_n \text{ with } n\in \{3, 4, 6\};\\
  F^\times &\quad \text{if }  \calO^\bs\simeq A_4 \text{ or } S_4. \\
\end{cases}
\]  
\end{lem}
\begin{proof}
  Let $\tilde{v}\in \calO^\bs$ be a nontrivial element 
  represented by $v\in \calO^\times$, and
  $x\in \ker(\varphi)\subseteq \calN(\calO)$. Since
  $x\tilde{v}x^{-1}=\tilde{v}$, we have $xF(v)x^{-1}=F(v)$, so either
  $xvx^{-1}=v$ or $xvx^{-1}=\bar{v}$. The latter case is possible only
  if $\bar{v}/v\in O_F^\times$. On the other hand, $\bar{v}/v$ is a
  root of unity in the CM-field $F(v)$. So necessarily $\bar{v}=-v$,
  i.e.~$\ord(\tilde{v})=2$. Therefore, if $\calO^\bs$ contains an
  element $\tilde{\eta}$ with $\ord(\tilde{\eta})>2$, then $\ker(\varphi)\subseteq
  F(\tilde{\eta})^\times$. In particular, if $\calO^\bs\simeq A_4$ or $S_4$,
  then there exist $\tilde{\eta}_1, \tilde{\eta}_2\in \calO^\bs$ such
  that  $F(\tilde{\eta}_1)\neq F(\tilde{\eta}_2)$ and
  $\min\{\ord(\tilde{\eta}_1), \ord(\tilde{\eta}_2)\}\geq 3$. Hence
  $F^\times\subseteq \ker(\varphi)\subseteq F(\tilde{\eta}_1)^\times\cap
  F(\tilde{\eta}_2)^\times=F^\times$ in this case. 

  Next, suppose that $\calO=O_F+O_Fi+O_F\eta+O_Fi \eta$ is a minimal
  $D_n$-order with $i^2\in \{-1, -\varepsilon\}$ and
  $n=\ord(\tilde{\eta})\geq 3$.  Note that $F(\eta)=F(j)$ for all $\calO$
 in this
  case, so we have $\ker(\varphi)\subseteq F(j)^\times$. Given any $x\in
  \ker(\varphi)$, if $xix^{-1}=i$, then $x\in F(i)^\times\cap
  F(j)^\times=F^\times$, otherwise $xix^{-1}=-i$, and hence $x\in
  F^\times j$.  It follows that $\ker(\varphi)=F^\times\cup F^\times
  j=F^\times\<j\>$ in this case. 

  Lastly, suppose that $\calO^\bs\simeq D_2$. Then $\calO=O_F[i, j]$
  in either $\qalg{-1}{-1}{F}$ or $\qalg{-1}{-\varepsilon}{F}$. For every $x\in \ker(\varphi)$, we have $xix^{-1}=\pm i$ and
  $xjx^{-1}=\pm j$.  There exists a suitable element in
  $\{1, i, j, k\}$ whose product with $x$ commutes with both $i$ and
  $j$, and hence lies in $F^\times$. Therefore,
  $\ker(\varphi)=F^\times\calO^\times$.
\end{proof}

Since we can write down $\Aut(G)$ for every $G\in \calG^\natural$, it is a simple exercise to work
out the normalizer $\calN(\calO)$ for each minimal $G$-order $\calO$.
For $\calO^\bs\simeq D_n$ with $n\in \{4, 6\}$, the non-central
elements of order 2 fall into two conjugacy classes with distinct
reduced norms by Remark~\ref{rem:elements-of-order-2}. Hence the image of $\varphi$ coincides with the inner
automorphism group $\Inn(\calO^\bs)$ in this case.  For a set $S\subset H^\times$, let $F^\times\calO^\times\<S\>$ be
the subgroup of $H^\times$ generated by $F^\times\calO^\times$ and $S$. Then
\begin{align}
\label{eq:56}  \calN(\scrO_2^\dagger)&=F^\times(\scrO_2^\dagger)^\times\< 1+i, 1+j\>, & \calN(\scrO_2^\ddagger)&=F^\times(\scrO_2^\ddagger)^\times\<1+i\>,     \\
\calN(\scrO_3^\dagger)&=F^\times(\scrO_3^\dagger)^\times\< j\>, & \calN(\scrO_3^\ddagger)&=F^\times(\scrO_3^\ddagger)^\times\< j\>,\\
 \label{eq:51}   \calN(\calO_{12})&=F^\times\calO_{12}^\times\<1+i\>,& &\\ 
\calN(\calO)&=F^\times\calO^\times \text{ if }\calO\simeq \scrO_4,  \scrO_6
              \text{ or }\O_{24}. \label{eq:98}
\end{align}

%


Let us denote
$\ol\calN(\calO):=\calN(\calO)/F^\times\calO^\times$ for any
order $\calO\subset H$.  Then $\ol\calN(\scrO_2^\dagger)\simeq S_3$,
and $\ol\calN(\calO')\simeq \zmod{2}$ if $\calO'\in \{\scrO_2^\ddagger, \scrO_3^\dagger, \scrO_3^\ddagger,
\calO_{12}\}$.  The natural action of 
$\calN(\calO)$ on the set $\grS(\calO)$ of maximal orders containing $\calO$ descends to $\ol\calN(\calO)$. The number of orbits is
denoted by 
\begin{equation}
  \label{eq:72}
\beth(\calO):=\abs{\ol\calN(\calO)\backslash
                                    \grS(\calO)}=\abs{\calN(\calO)\backslash\grS(\calO)}. 
\end{equation}

\begin{lem}\label{lem:conj-max-orders}
  Let $\calO\subset H$ be a minimal $G$-order with $G\in \calG^\natural$, and $\O,\O'\in \grS(\calO)$ be two 
  maximal $O_F$-orders containing $\calO$. If $\O$ and $\O'$ are
  $H^\times$-conjugate, then there there
  exists $x\in \calN(\calO)$ such that $\O'=x\O x^{-1}$.
\end{lem}
\begin{proof}
  By the assumption, there there exists $y\in H^\times$ such that
  $\O=y\O'y^{-1}$. Applying Lemma~\ref{lem:strict-isom-subgp-conj} to
  the groups 
  $\calO^\bs\subseteq \O^\bs$ and
  $y\calO^\bs y^{-1}\subseteq \O^\bs$, we see that there exists
  $u\in \O^\times$ such that
  $u\calO^\bs u^{-1}=y\calO^\bs y^{-1}$, or equivalently, $u\calO
  u^{-1}= y\calO y^{-1}$. Take $x=y^{-1}u$. Then $x\in \calN(\calO)$
  and $x\O x^{-1}=\O'$.  
\end{proof}
Therefore, if $\calO$ is a minimal $G$-order, then
$\beth(\calO)$ is the number of conjugacy classes of
maximal orders containing $\calO$.

\section{Refined type numbers for noncyclic reduced unit groups: part I}
\label{sec:maximal-orders}
\numberwithin{thmcounter}{section}
Let $H$ be a totally definite quaternion algebra over a real quadratic
field $F=\Q(\sqrt{d})$ with square-free $d\geq 6$. The refined type number
$t(G)=t(H, G)$ is defined in (\ref{eq:106}) and counts the number of
conjugacy classes of maximal orders in $H$ with reduced unit group
$G$. We compute $t(G)$ for each finite noncyclic group $G$ in
Figure~\ref{fig:rel-gps}, starting from the groups on the top and
working downward. The current
section treats the cases
$G\in \{S_4, D_6,  A_4, D_4, D_2^\dagger, D_3^\dagger\}$. The remaining cases $G\in \{D_2^\ddagger, D_3^\ddagger\}$ are
postponed to the next section.

By Corollary~\ref{cor:quat-alg-with-non-cyclic-gp}, if $t(G)\neq 0$ for
some $G\in \calG^\natural$ as in \eqref{eq:139}, then we may identify $H$ with  one of the quaternion $F$-algebras in
Table~\ref{tab:minimal-G-orders}. Assume that this is the case
throughout this section.  Let $\O\subset H$ be a maximal
order with $\O^\bs\simeq G$.  After a
suitable $H^\times$-conjugation, $\O$ contains the minimal
$G$-order $\calO$ in Table~\ref{tab:minimal-G-orders}.  To
compute $t(G)$, we classify the set $\grS(\calO)$ of maximal orders
containing $\calO$,  and compute  the number  $\beth(\calO)$ of
$\calN(\calO)$-conjugacy classes. It then follows from
Lemma~\ref{lem:conj-max-orders} that 
\begin{equation}
  \label{eq:147}
t(G)=\beth(\calO)-\sum_{G< G'} t(G'). 
\end{equation}
The orbits in $\calN(\calO)\backslash \grS(\calO)$ consisting of maximal
orders with reduced unit groups $G'>G$ have already been
accounted for in a previous iteration. After
eliminating those, we write down a representative $\O\in \grS(\calO)$
for each remaining orbit in $\calN(\calO)\backslash
\grS(\calO)$. Necessarily, such
an $\O$ is  intermediate to $\calO\subseteq
\calO^\vee$ (the dual lattice of $\calO$). By construction, $\O^\times=\calO^\times$. Since $\calO$ is
spanned by $\calO^\times$ and every $x\in \calN(\O)$ normalizes $\O^\times$, we have 
\begin{equation}
  \label{eq:148}
\calN(\O)\subseteq \calN(\calO).
\end{equation}
 This provides a
way to work out $\calN(\O)$ explicitly by applying the results of
Section~\ref{subsec:norm-minim-orders}. We then 
apply  (\ref{eq:107}) to compute $h(G)$. Recall that $\omega(H)$
denotes the number of finite primes of $F$ that are ramified in $H$.





The computation of $t(S_4)$ and $t(D_6)$  is made easiest since
minimal $O_F$-orders attached to $S_4$ and $D_6$  turn out to be
maximal by Corollary~\ref{cor:mini-D6-S4-maximal}.

\begin{prop}\label{prop:S4-unique-order}
We have 
 \begin{align}
  \label{eq:46}
  t(S_4)&=
  \begin{cases}
    1 &\quad\text{if } H=\qalg{-1}{-1}{F} \text{ and } 2\varepsilon\in F^{\times 2} 
    ,\\
    0 &\quad\text{otherwise;}
  \end{cases}\\
 h(S_4)&=h(F)t(S_4). 
\end{align}
When $2\varepsilon\in F^{\times 2}$, every maximal order in $\qalg{-1}{-1}{F}$ with reduced unit
group $S_4$ is conjugate to $\O_{24}$ in (\ref{eq:144}), and
$\calN(\O_{24})=F^\times\O_{24}^\times$. 
\end{prop}
\begin{proof}
The normalizer of $\O_{24}$ is calculated in (\ref{eq:98}).   Suppose that
$2\varepsilon\in F^{\times 2}$ and $H=\qalg{-1}{-1}{F}=H_{2,
  \infty}\otimes_\Q F$, where $H_{2,\infty}=\qalg{-1}{-1}{\Q}$ as in
(\ref{eq:149}).  Note that  $2$ is ramified in $F$, and the quaternion
$\Q$-algebra $H_{2, \infty}$ is
ramified only at $2$ and $\infty$. It follows that $\omega(H)=0$, and hence
$h(S_4)=h(F)t(S_4)$.
\end{proof}
\begin{rem}\label{rem:suborders-of-O24}
Suppose that $2\varepsilon\in F^{\times 2}$, and fix $\vartheta\in F^\times$ such that
$\varepsilon=2\vartheta^2$ as in Section~\ref{subsec:structure-CM-ext}. 
From the proof of Lemma~\ref{lem:strict-isom-subgp-conj},  $O_F[i,
j]\subset \qalg{-1}{-1}{F}$ is the unique minimal $D_2^\dagger$-suborder of
  $\O_{24}$. On the other hand, $O_F[i, \vartheta(j+k)]$ is a minimal
  $D_2^\ddagger$-suborder of $\O_{24}$, and any other minimal
  $D_2^\ddagger$-suborder of $\O_{24}$ is $\O_{24}^\times$-conjugate
  to it. We have
  $O_F[i, \vartheta(j+k)]\not\subset O_F+(2\vartheta)\O_{24}$ since
  the latter order does not contain $\vartheta(j+k)$.  
  \end{rem}

\begin{prop}\label{prop:D6-max-ord-unique}
We have 
\begin{align}
  t(D_6)&=
  \begin{cases}
    1 &\qquad \text{if } H=\qalg{-1}{-3}{F} \text{ and } 3\varepsilon\in F^{\times 2},\\
    0 &\qquad \text{otherwise;}
  \end{cases}\\
h(D_6)&=h(F)t(D_6).
\end{align}
When $3\varepsilon\in F^{\times 2}$, every maximal order in $\qalg{-1}{-3}{F}$ with reduced unit
group $D_6$ is conjugate to $\scrO_6$ in (\ref{eq:143}), and
$\calN(\scrO_6)=F^\times\scrO_6^\times$. 
\end{prop}
\begin{proof}
The proof of Proposition~\ref{prop:S4-unique-order} applies,  mutatis
mutandis, to here as well. 
\end{proof}

  We introduce two maximal orders in $\qalg{-1}{-1}{F}$ when $2$
  is ramified in $F$.   The first order provides a simplified form of $\O_{24}$ when $2\varepsilon\in F^{\times
    2}$ and is also used for the calculation of $t(A_4)$, and the second one will be used for $t(D_4)$ and $t(D_2^\dagger)$. 


\begin{def-lem}\label{defn:O12-O4I}
  Suppose that $d\equiv 2, 3\pmod{4}$ so that $2$ is ramified in
  $F$, and hence $H=\qalg{-1}{-1}{F}$ splits at all finite places of $F$. 
  
  \begin{enumerate}[(i)]
  \item Let $\grp$ be the unique dyadic prime of $F$,  $K:=F(i)\subset
H$, and 
\begin{align}
  \label{eq:53}
B&:=
\begin{cases}
  O_K & \quad \text{if } d\equiv 2\pmod{4},\\
 O_F+\grp O_K &\quad  \text{if } d\equiv 3\pmod{4}, \\
\end{cases}\\
  \label{eq:74}
\O_{12}&:=B+B\xi, \quad \text{where}\quad \xi=(1+i+j+k)/2\in H.
\end{align}
Then $\O_{12}$
 is the unique maximal order containing the minimal $A_4$-order
$\calO_{12}\subseteq H$. If  $2\varepsilon\in
F^{\times 2}$, then $\O_{12}$ coincides with $\O_{24}$ in
(\ref{eq:144}), and 
\begin{equation}
\label{eq:150}
 \O_{12}^\bs\simeq
  \begin{cases}
    S_4\quad &\text{if } 2\varepsilon\in
F^{\times 2},\\
    A_4\quad &\text{otherwise.}
  \end{cases}
\end{equation}

\item Let
  $L$ (resp.~$\O_4^\dagger$) be the following subfield (resp.~$O_F$-order) of $H$: 
  \begin{align}
    \label{eq:49}
L:&=
\begin{cases}
  F((1+\sqrt{d}i+j)/2)\simeq F(\sqrt{-(d+1)})\qquad &\text{if }
  d\equiv 2\pmod{4},\\
F(j)\simeq F(\sqrt{-1}) \qquad &\text{if }
  d\equiv 3\pmod{4},
\end{cases}\\
\O_4^\dagger:&=O_L+iO_L, \qquad \text{or more explicitly}\label{eq:152}\\
\O_4^\dagger:&= 
\begin{cases}
  O_F+O_Fi+O_F\frac{1+\sqrt{d}i+j}{2}+O_F\frac{-\sqrt{d}+i+k}{2}\quad
&\text{if } d\equiv 2\pmod{4}, \\  
O_F+O_Fi+O_F\frac{\sqrt{d}+j}{2}+O_F\frac{\sqrt{d}i+k}{2}\quad &\text{if } d\equiv 3\pmod{4}.
\end{cases}
  \end{align}
Then $\O_4^\dagger$ is a maximal order containing the minimal
$D_2^\dagger$-order $\scrO_2^\dagger=O_F[i, j]$, and it is not
$H^\times$-conjugate to $\O_{12}$ in (\ref{eq:74}).   We have 
\begin{equation}
\label{eq:151}
 (\O_4^\dagger)^\bs\simeq
  \begin{cases}
    D_4\quad &\text{if } 2\varepsilon\in
F^{\times 2},\\
    D_2^\dagger\quad &\text{otherwise.}
  \end{cases}
\end{equation}

\item  When $2\varepsilon\in F^{\times 2}$, the maximal orders
  $\O_{12}$ and $\O_4^\dagger$ are denoted as $\O_{24}$ and $\O_8$
  respectively to emphasize their
   reduced unit groups. Both $\O_{24}$ and $\O_8$ contain the minimal $D_4$-order
 $\scrO_4=O_F[i, \sqrt{\varepsilon j}]$. 
  \end{enumerate}
\end{def-lem}


\begin{proof}[Proof of Definition-Lemma~\ref{defn:O12-O4I}]
Let $H_{2,
  \infty}=\qalg{-1}{-1}{\Q}$, the unique quaternion $\Q$-algebra (up to
isomorphism) ramified
exactly at $2$ and $\infty$. We have $H=H_{2, \infty}\otimes_\Q F$, which 
splits at all finite places of
$F$ since $2$ is ramified in $F$ by
the assumption. 

From Table~\ref{tab:orders-K1},  the order $B$ is the unique $O_F$-order in $K$ containing $O_F[i]$ with $\chi(B, O_F[i])=\grp$.
  Straightforward calculation shows that $\O_{12}=B+B\xi$ is closed under
  multiplication, hence an $O_F$-order. Since
  $\calO_{12}=O_F[i]+O_F[i]\xi$, we have
  $\grd(\O_{12})=\grp^{-2}\grd(\calO_{12})=O_F$ by
  Table~\ref{tab:minimal-G-orders}.  Thus, $\O_{12}$ is a maximal
  order. If $2\varepsilon\in F^{\times 2}$, then
  $B=O_F[\sqrt{\varepsilon i}]$ by Lemma~\ref{lem:OF-eta-K1}, which implies that
  $\O_{12}=\O_{24}$ and hence $\O_{12}^\bs\simeq S_4$; otherwise
  $\O_{12}^\bs\simeq A_4$ since $\O_{12}\supset \calO_{12}$. It
  is clear from the expression of $\O_{24}$ in (\ref{eq:144}) that
  $\O_{24}\supset \scrO_4$. 

  Recall that $\calO_{12}=\bbo_2\otimes_\Z O_F$, where $\bbo_2$ is the
  maximal $\Z$-order in  $H_{2, \infty}$ defined in (\ref{eq:146}). 
Thus $\calO_{12}\otimes\Z_\ell$ is maximal in
  $H\otimes\Q_\ell$ for every prime $\ell\neq 2$ since 
  $H_{2, \infty}$  splits at $\ell$.  As both $H_{2,\infty}$ and $F$  are ramified at
  $2$, there is a unique maximal order of $H\otimes\Q_2$ containing
  $\calO_{12}\otimes \Z_2$ by Lemma~\ref{lem:ram-ext}. Therefore, $\O_{12}$
  is the unique maximal order containing $\calO_{12}$, and there is no other $O_F$-order intermediate to $\calO_{12}\subset \O_{12}$.

For (ii),  one calculates directly that $\O_4^\dagger=O_L+iO_L$ is
closed under multiplication, hence an $O_F$-order. Let
$B':=O_F[\sqrt{d}i+j]\subset O_L$ if $d\equiv 2\pmod{4}$,  and
$B':=O_F[j]\subset O_L$ if
$d\equiv 3\pmod{4}$. Then $\scrO_2^\dagger=B'+iB'$ and $\chi(O_L:
B')=2O_F$. Thus 
$\grd(\O_4^\dagger)=(2O_F)^{-2}\grd(\scrO_2^\dagger)=O_F$ by 
  Table~\ref{tab:minimal-G-orders}.  It follows that $\O_4^\dagger$ is
  a maximal order containing $\scrO_2^\dagger$. Note that $\O_{12}\cap
  F(i)=B\not\simeq O_F[\sqrt{-1}]$, so by symmetry, results of the same form hold true if $i$ is replaced by
  $j$ or $k$. On the other hand,  $\O_4^\dagger\cap F(i)=O_F[i]\simeq O_F[\sqrt{-1}]$. Thus
  $\O_4^\dagger$ is not $H^\times$-conjugate to $\O_{12}$. 

  Suppose that $2\varepsilon\in F^{\times 2}$. Clearly,
  $\scrO_4\subseteq \O_4^\dagger$ if $d\equiv 3\pmod{4}$. By
  Lemma~\ref{lem:OF-eta-K1}, if $d\equiv 2\pmod{4}$, then
  $O_F[\sqrt{\varepsilon j}]=O_{F(j)}$, which has a $\Z$-basis
  $\{1,\sqrt{d}, j, (\sqrt{d}+\sqrt{d}j)/2\}$. Thus,
  $\O_4^\dagger\supset \scrO_4$ in this case as well. We have
  $(\O_4^\dagger)^\bs\not\simeq S_4$ since $\O_4^\dagger$ is not
  conjugate to $\O_{12}$. Now (\ref{eq:151}) follows directly from
  Table~\ref{tab:non-cyclic-red-gp} and Figure~\ref{fig:rel-gps}. 
\end{proof}




For each prime $p\in \bbN$, we write $\Lsymb{F}{p}$ for the Artin
symbol \cite[p.~94]{vigneras}:
\begin{equation}\label{eq:165}
\Lsymb{F}{p}=
  \begin{cases}
    1 \quad&\text{if $p$ splits in } F,\\
    0 \quad&\text{if $p$ ramifies in } F,\\
   -1 \quad&\text{if $p$ is inert in } F.\\
  \end{cases}
\end{equation}

\begin{prop}\label{prop:type-num-A4}
We have
\begin{align}
  t(A_4)&=
  \begin{cases}
    1 &\quad\text{if } H=\qalg{-1}{-1}{F} \text{ and }
    2\varepsilon\not\in F^{\times 2}, \\
    0 &\quad\text{otherwise};
\end{cases}\\
h( A_4)&=\left(\frac{1}{2}+\frac{1}{2}\Lsymb{F}{2}+\Lsymb{F}{2}^2\right)h(F)t(A_4).\label{eq:153}
\end{align}  
Suppose that $H=\qalg{-1}{-1}{F}$ and $2\varepsilon\not\in F^{\times
  2}$. We set $\O_{12}=B+B\xi$ as in (\ref{eq:74}) if $d\not \equiv
1\pmod{4}$,  $\O_{12}=\calO_{12}$ if $d\equiv 1\pmod{8}$, and 
\begin{equation}
  \label{eq:52}
\O_{12}=O_F+O_F\frac{(1+\sqrt{d})-2i+(1-\sqrt{d})j}{4}+O_Fj+O_F\xi\quad
\text{ if } d\equiv 5 \bmod8.
\end{equation}
Then every maximal order $\O\subset H$ with $\O^\bs\simeq A_4$ is
$H^\times$-conjugate to $\O_{12}$, and 
\begin{equation}
  \label{eq:54}
  \calN(\O_{12} )=
  \begin{cases}
F^\times\O_{12}^\times   &\qquad \text{if } d\equiv 5\pmod{8},\\
F^\times\O_{12}^\times\<1+i\> &\qquad \text{otherwise}. 
  \end{cases}
\end{equation}
\end{prop}
\begin{proof}
Suppose that $t(A_4)\neq 0$, and $\O\subset H$ is a maximal order
with $\O^\bs\simeq A_4$. By Table~\ref{tab:minimal-G-orders}, we may
identify $H$ with $\qalg{-1}{-1}{F}$ in such a way that $\O\supseteq
\calO_{12}$. 
 Necessarily, $2\varepsilon \not\in F^{\times
    2}$,  otherwise $\O=\O_{24}$, which has reduced unit group
  $S_4$. 

  Conversely, suppose that $H=\qalg{-1}{-1}{F}$ and
  $2\varepsilon\not\in F^{\times 2}$.   Then $t(S_4)=0$, and $t(A_4)=\beth(\calO_{12})$. In other words, $\O^\bs\simeq A_4$ for every maximal order $\O$
  containing $\calO_{12}=\bbo_2\otimes_\Z O_F$. According to
  (\ref{eq:51}) and (\ref{eq:148}),
  \begin{equation}\label{eq:nO12}
  \calN(\O)\subseteq
  \calN(\calO_{12})=F^\times\calO_{12}^\times\<1+i\>=F^\times\O^\times\<1+i\>.
  \end{equation}
We show that
  $\calN(\calO_{12})$ acts transitively on 
  $\grS(\calO_{12})$ so that $t(A_4)=1$.

  If $F$ splits at $2$, then $H$ is ramified at the two dyadic primes of $F$, and $\calO_{12}$ is already a maximal order. We have 
  $\omega(H)=2$, and $h(A_4)=2h(F)$ in this case.

  In the remaining two cases, there is a unique dyadic prime $\grp$ of
  $F$, and $H$ splits at all finite primes of $F$,
  i.e.~$\omega(H)=0$. If $F$ is ramified at $2$, then $\O_{12}$
  in (\ref{eq:74}) is the unique maximal order containing $\calO_{12}$
  by Definition-Lemma~\ref{defn:O12-O4I}.  Necessarily
  $\calN(\O_{12})=\calN(\calO_{12})$, and $h(A_4)=\frac{1}{2}h(F)$.

If $F$ is inert
  at $2$,  then $e_\grp(\calO_{12})=1$ and
  $\calO_{12}$ is an Eichler order of level $\grp=2O_F$ by
  Lemma~\ref{lem:unram-ext}. Let $\O$ and $\O'$ be the two maximal
  orders containing $\calO_{12}$.  Note that
  $(1+i)\in \calN(\calO_{12})$ , but $(1+i)\not\in\calN(\O)$ since it
  is odd at $\grp$ (see
  Section~\ref{sect:parity-of-element}). Therefore,
  $(1+i)\O (1+i)^{-1}=\O'$, and $\calN(\O)=F^\times\O^\times$ by
  \eqref{eq:nO12}. It follows that $t(A_4)=1$ and $h(A_4)=h(F)$. 
A direct calculation shows that the dual basis of $\{1, i, j,
\xi\}\subset \qalg{-1}{-1}{F}$ is 
\[\left\{\frac{1+k}{2},\quad \frac{-i+k}{2}, \quad \frac{-j+k}{2},
  \quad -k   \right\},\]
and the order in (\ref{eq:52}) is a maximal order intermediate to
$\calO_{12}\subset \calO_{12}^\vee$, so it coincides with either $\O$
or $\O'$.  

Summarizing the above three cases and interpolating, we obtain
(\ref{eq:153}). 
\end{proof}

Similar to $\calO_{12}=\bbo_2\otimes_\Z O_F$, the minimal $D_3^\dagger$-order
$\scrO_3^\dagger$ is also obtained from the maximal
$\Z$-order $\gro_3\subset H_{3, \infty}=\qalg{-1}{-3}{\Q}$ in
(\ref{eq:60}) by extending the scalars from $\Z$ to $O_F$. By
Figure~\ref{fig:rel-gps}, $D_3^\dagger$ is realizable as a subgroup of
$D_6$ but not $S_4$. 

\begin{prop}\label{prop:type-num-D3I}
 We have 
\begin{align}
\label{eq:59}
    t(D_3^\dagger)&=
  \begin{cases}
    1 &\quad\text{if } 
    H\simeq\qalg{-1}{-3}{F} \text{ and }
    3\varepsilon\not\in F^{\times 2}  
    ,\\ 0 &\quad\text{otherwise.}
\end{cases}\\
    h(D_3^\dagger)&=\left(\frac{1}{2}+\frac{1}{2}\Lsymb{F}{3}+\Lsymb{F}{3}^2\right)h(F)t(D_3^\dagger).
\end{align}
Suppose that $H=\qalg{-1}{-3}{F}$ and $3\varepsilon\not\in F^{\times
  2}$. We set $\O_6^\dagger=O_{F(j)}+iO_{F(j)}$ if $d \equiv
0\pmod{3}$,  $\O_6^\dagger=\scrO_3^\dagger$ if $d\equiv 1\pmod{3}$,
and 
\begin{equation}\label{eq:101}
  \O_6^\dagger:=O_F+O_F\frac{i+k}{2}+O_F\frac{1+j}{2}+O_F\frac{\sqrt{d}j+k}{3}\quad\text{if
  }
  d\equiv 2\pmod{3}. 
\end{equation}
Every maximal order $\O\subset H$ with $\O^\bs\simeq D_3^\dagger$ is
$H^\times$-conjugate to $\O_6^\dagger$, and 
\[\calN(\O_6^\dagger)=
\begin{cases}
  F^\times(\O_6^\dagger)^\times\qquad &\text{if } d\equiv 2\pmod{3},\\
  F^\times(\O_6^\dagger)^\times\<j\>\qquad &\text{otherwise.}
\end{cases}
\]
\end{prop}
\begin{proof}
By (\ref{eq:61}), the dual lattice of $\scrO_3^\dagger$ is 
\[  (\scrO_3^\dagger)^\vee=O_F\frac{3+j}{6}+O_F\frac{-3i+k}{6}+O_F\frac{j}{3}+O_F\frac{k}{3}.\]
When $d\equiv 2\pmod{3}$, one checks directly that
$\scrO_3^\dagger+O_F\frac{\sqrt{d}j+k}{3}\subseteq
(\scrO_3^\dagger)^\vee$ is a maximal order and coincides with the one
in (\ref{eq:101}).  Similar proof as that of 
Proposition~\ref{prop:type-num-A4} shows that  $\grS(\scrO_3^\dagger)=\{\O_6^\dagger,
j\O_6^\dagger j^{-1}\}$ if 
$d\equiv 2\pmod{3}$, and 
$\grS(\scrO_3^\dagger)=\{\O_6^\dagger\}$ otherwise. 
The rest of the proof is almost the same, hence omitted.
\end{proof}

\begin{prop}\label{prop:max-D4}
  We have  
  \begin{align}
  t(D_4)&=
  \begin{cases}
    1 &\quad\text{if } H=\qalg{-1}{-1}{F} \text{ and }  2\varepsilon\in F^{\times 2}, 
    \\
    0 &\quad\text{otherwise};
\end{cases}\\
      h(D_4)&=h(F)t(D_4). 
  \end{align}
In particular, 
\begin{equation}
  \label{eq:81}
  t(D_4)=t(S_4) \qquad \text{for  all }  H. 
\end{equation}
Suppose that $H=\qalg{-1}{-1}{F}$ and $2\varepsilon\in F^{\times
  2}$.
Every maximal order $\O\subset H$ with $\O^\bs\simeq D_4$ is
$H^\times$-conjugate to $\O_8=\O_4^\dagger$ in (\ref{eq:152}), and 
\begin{equation}
  \label{eq:55}
\calN(\O_8)=F^\times\O_8^\times. 
\end{equation}
\end{prop}


\begin{proof}
  By Table~\ref{tab:minimal-G-orders}, there exists a minimal
  $D_4$-order only if $H=\qalg{-1}{-1}{F}$ and
  $2\varepsilon\in F^{\times 2}$. We show that
  $t(D_4)=1$ in this case. It is shown in
  Definition-Lemma~\ref{defn:O12-O4I} that 
  $\grS(\scrO_4)\supseteq \{\O_{24}, \O_8\}$, and $\O_8$ is not
  $H^\times$-conjugate to $\O_{24}$. 
 Let $\grp=(2\vartheta) O_F$ be the
  unique dyadic prime of $F$, where
  $2\vartheta^2=\varepsilon$.  By
  \cite[Corollary~1.6]{Brzezinski-1983},
  $\scrO_4=O_F[i, \sqrt{\varepsilon j}]$ is a Bass order since
  $\grd(\scrO_4)=2O_F=\grp^2$ is cube-free. Note that
  $i^2\equiv (\sqrt{\varepsilon j})^4\equiv 1 \pmod{\grp}$. The same
  proof as that of Lemma~\ref{lem:gen-lem-scro2II} shows that
  $\scrO_4$ is maximal at every finite place of $F$ coprime
  to $2$, and the Eichler invariant $e_\grp(\scrO_4)=0$. It then
  follows from Lemma~\ref{lem:bass-order-eichler-inv=0} that
  $\grS(\scrO_4)=\{\O_{24}, \O_8\}$. 
  Since $\O_8$ is not $H^\times$-conjugate to $\O_{24}$, we have 
  $\O_8^\bs=\scrO_4^\bs\simeq D_4$. Therefore, 
\[F^\times\O_8^\times\subseteq \calN(\O_8)\subseteq
  \calN(\scrO_4)=F^\times\scrO_4^\times=F^\times\O_8^\times. \]
It follows that (\ref{eq:55}) holds,  and $h(D_4)=h(F)t(D_4)$
since $\omega(H)=0$. 
\end{proof}

\begin{rem}
The proof of Proposition~\ref{prop:max-D4} shows that $\O_{24}$ and
$\O_8$ intersect at an Eichler order $\scrE$ of level
$\grp=(2\vartheta)O_F$, and $\scrE$ is also the unique minimal
overorder of $\scrO_4$.  It is not hard to write it down explicitly: 
\begin{equation}\label{eq:24}
\scrE=O_F[i, \sqrt{\varepsilon j}, \xi-\sqrt{\varepsilon
  i}]=O_F+O_Fi+O_F\sqrt{\varepsilon j}+O_F(\xi-\sqrt{\varepsilon
  i}). 
\end{equation}
\end{rem}

\begin{prop}\label{prop:type-num-D2I}
 We have
 \begin{align}
  t(D_2^\dagger)&=
  \begin{cases}
    1 &\quad\text{if } H=\qalg{-1}{-1}{F}, \left(\frac{F}{2}\right)=0 \text{ and }  2\varepsilon\not\in F^{\times 2},\\
    0 &\quad\text{otherwise};
  \end{cases}\\
   h(D_2^\dagger)&=\frac{1}{2}h(F)t(D_2^\dagger)
 \end{align}
Suppose that $H=\qalg{-1}{-1}{F}, \left(\frac{F}{2}\right)=0$
  and  $2\varepsilon\not\in F^{\times 2}$. 
Every maximal order $\O\subset H$ with 
  $\O^\bs\simeq D_2^\dagger$ is $H^\times$-conjugate to $\O_4^\dagger$
in  (\ref{eq:152}), and  
\[\calN(\O_4^\dagger)=F^\times(\O_4^\dagger)^\times\<1+i\>.\]
\end{prop}



\begin{proof}
  We focus on the case that $H=\qalg{-1}{-1}{F}$ since
  $t(D_2^\dagger)=0$ otherwise.  Thanks to
  Lemma~\ref{lem:gen-lem-scro2II}, $\scrO_2^\dagger=O_F[i,j]$ is a Gorenstein
  order maximal at every prime $\ell\neq 2$. We study the set of  maximal
  orders in $H$ containing $\scrO_2^\dagger$.

If $\Lsymb{F}{2}=1$,  then $t(D_2^\dagger)=0$. Indeed, $H$ is ramified at
  the two dyadic primes of $F$ in this case, and $\calO_{12}$ is the unique
  maximal order containing $\scrO_2^\dagger$ with
  $\calO_{12}^\bs\simeq A_4$. Suppose that $\Lsymb{F}{2}\neq 1$ so that there is a   unique dyadic
  prime $\grp$ of $F$. Then $e_\grp(\scrO_2^\dagger)=0$ by
  Lemma~\ref{lem:gen-lem-scro2II}. It follows from
  \cite[Proposition~4.1]{Brzezinski-1983} that there is a unique
  minimal overorder $\calO$ of $\scrO_2^\dagger$, and
  $\chi(\calO, \scrO_2^\dagger)=\grp$.  In particular, $\grS(\calO)=\grS(\scrO_2^\dagger)$, and 
  $\calO_{12}\supseteq \calO$. Since $\chi(\calO_{12}, \scrO_2^\dagger)=2O_F$ by Table~\ref{tab:minimal-G-orders},  we see that $\calO_{12}=\calO$ if $\Lsymb{F}{2}=-1$, forcing $\grS(\scrO_2^\dagger)=\grS(\calO_{12})$.   Therefore,
  $t(D_2^\dagger)\neq 0$ only if $\Lsymb{F}{2}=0$, which we
  assume for the rest of the proof. In this case, $\grd(\calO)=\chi(\calO, \scrO_2^\dagger)^{-1}\grd(\scrO_2^\dagger)=\grp^3$ by Table~\ref{tab:minimal-G-orders}. It follows from \cite[Proposition~4.1]{Brzezinski-1983}  again that $e_\grp(\calO)=0$ since $\grd(\calO)\neq \grp$.

Note that $O_F[\sqrt{-1}]$ is not maximal in $F(\sqrt{-1})$ (see Table~\ref{tab:orders-K1}).  According to Lemma~\ref{lem:scro2II-Bass-criterion}, $\scrO_2^\dagger$ is not Bass, and hence $\calO$ is not Gorenstein by \cite[Proposition~1.12]{Brzezinski-1983}.  Now it follows from Lemma~\ref{lem:non-Gorenstein-order} that $\aleph(\scrO_2^\dagger)=\aleph(\calO)=4$. More precisely, the subtree $\grT(\scrO_2^\dagger\otimes \Z_2)$ of the  Bruhat-Tits
  tree of $H\otimes \Q_2$ is a star centered at $\Gor(\calO)\otimes \Z_2$ with $3$
external vertices. Here $\Gor(\calO)$ denotes the Gorenstein saturation of $\calO$, which is a maximal order as shown in Lemma~\ref{lem:non-Gorenstein-order}.
According to Definition-Lemma~\ref{defn:O12-O4I},
 $\grS(\scrO_2^\dagger)\supseteq \{\O_{12}, \O_4^\dagger\}$. By \eqref{eq:56},   $\O_{12}$ is fixed under the action of $\ol\calN(\scrO_2^\dagger)$ on
$\grS(\scrO_2^\dagger)$. We claim that  $\beth(\scrO_2^\dagger)=2$, that is,  $\ol\calN(\scrO_2^\dagger)$ acts transitively on the set $\grS'(\scrO_2^\dagger):=\grS(\scrO_2^\dagger)\smallsetminus
\{\O_{12}\}$. It then follows that $\Gor(\calO)=\O_{12}$.

In fact, 
$\xi=(1+i+j+k)/2\in \calN(\scrO_2^\dagger)$ acts transitively on
$\grS'(\scrO_2^\dagger)$. Otherwise, $\xi$ lies in the
normalizer of one of its members, say $\O'$.  Then $\xi\in \O'$ by
Lemma~\ref{lem:max-order-normalizer}, and hence $\O'$ contains
$\calO_{12}=\scrO_2^\dagger+O_F\xi$ as well. But this contradicts the fact that $\O_{12}$ is the unique maximal overorder of $\calO_{12}$.  Since $\xi$
generates the only nontrivial normal proper subgroup of
$\ol\calN(\scrO_2^\dagger)\simeq S_3$, the action of
$\ol\calN(\scrO_2^\dagger)$ on $\grS'(\scrO_2^\dagger)$ identifies
$\ol\calN(\scrO_2^\dagger)$ with the  full symmetric  group on
$\grS'(\scrO_2^\dagger)$. 




Lastly, if $2\varepsilon\in F^{\times 2}$, then the orbits
$\ol\calN(\scrO_2^\dagger)\backslash \grS(\scrO_2^\dagger)$ are
represented by $\O_{12}=\O_{24}$ and $\O_8=\O_4^\dagger$, with
$\O_{24}^\bs\simeq S_4$ and $\O_8^\bs\simeq D_4$ respectively. Thus
there are no maximal orders $\O$ with $\O^\bs\simeq D_2^\dagger$ in
this case.  If $2\varepsilon\not\in F^{\times 2}$, then
$\O_{12}^\bs\simeq A_4$, and $(\O_4^\dagger)^\bs\simeq  D_2^\dagger$  since
 $\O_4^\dagger $ is not conjugate to $\O_{12}$. 
Hence $\calN(\O_4^\dagger)\subseteq \calN(\scrO_2^\dagger)$. One verifies
directly that $(1+i)\in \calN(\O_4^\dagger)$, but
$\xi\not\in\calN(\O_4^\dagger)$ as demonstrated. It follows that
\[\calN(\O_4^\dagger)=F^\times(\scrO_2^\dagger)^\times\<1+i\>=F^\times(\O_4^\dagger)^\times\<1+i\>.\]
We conclude that $h(D_2^\dagger)=h(F)/2$ since $\omega(H)=0$. 
\end{proof}


%



\section{Refined type numbers for noncyclic reduced unit groups: part II}
\label{sec:maximal-orders-part2}

We keep the notation and assumptions of
Section~\ref{sec:maximal-orders} and study the refined type numbers
and class numbers for $D_2^\ddagger$ or $D_3^\ddagger$. In particular, $F=\Q(\sqrt{d})$ with square-free $d\geq 6$.  We assume further that $\Nm_{F/\Q}(\varepsilon)=1$ throughout this section,  otherwise $t(D_2^\ddagger)=t(D_3^\ddagger)=0$.  

\numberwithin{thmcounter}{subsection}
\subsection{Maximal orders containing $\scrO_2^\ddagger$}
In this subsection, $H$ denotes the quaternion algebra $\qalg{-1}{-\varepsilon}{F}$.
We first write down the finite ramified places of
$H$.  By Lemma~\ref{lem:fund-unit-d=1mod8},
if $d\equiv 1\pmod{8}$, then $\varepsilon$ is of the form
$a+b\sqrt{d}\in \Z[\sqrt{d}]$ with $a$ odd and $b$ divisible by $4$.

\begin{lem}\label{lem:ram-places-D2II}
  The quaternion algebra
  $H=\qalg{-1}{-\varepsilon}{F}$ splits at all finite nondyadic primes
  of $F$. If $d\not\equiv 1\pmod{8}$, then $H$ splits at the unique dyadic
  prime of $F$ as well. When $d\equiv 1\pmod{8}$, $H$ is ramified at
  the two dyadic primes of $F$ if and only if
  $\varepsilon=a+b\sqrt{d}$ with $a\equiv 1\pmod{4}$. 
\end{lem}

\begin{proof}

By Lemma~\ref{lem:gen-lem-scro2II},  $H$ splits at all finite
nondyadic primes of $F$. First suppose that $\Lsymb{F}{2}\neq 1$  so that there is a unique dyadic prime $\grp\subset O_F$. Since the number
of ramified places of $H$ is even, $H$ necessarily splits at $\grp$ as
well. Hence $H$ splits at all finite primes of $F$ in this
case. 

Now suppose that $d\equiv 1\pmod{8}$. We have $F_\grp=\Q_2$ for every dyadic prime $\grp$, and by
\cite[Corollary~V.3.3]{Serre_local}, 
\[\Nm_{\Q_2(\sqrt{-1})/\Q_2}(\Z_2[\sqrt{-1}]^\times)=\{u\in
\Z_2^\times\mid  u\equiv 1\pmod{4}\}. \]
It follows that the Hilbert symbol $(-1, -\varepsilon)_\grp=1$ if and
only if $v_\grp(\varepsilon+1)\geq 2$, or equivalently,  $a\equiv 3\pmod{4}$. 
\end{proof}

Let $\scrO_2^\ddagger=O_F[i, j]\subset H$
be the minimal $D_2^\ddagger$-order in Table~\ref{tab:minimal-G-orders}. Recall that
$\calN(\scrO_2^\ddagger)=F^\times(\scrO_2^\ddagger)^\times\<1+i\>$ by
(\ref{eq:56}), so
$\ol\calN(\scrO_2^\ddagger)=\calN(\scrO_2^\ddagger)/F^\times(\scrO_2^\ddagger)^\times$
is a cyclic group of order $2$ generated by $1+i$. By
Lemma~\ref{lem:conj-max-orders}, two distinct maximal orders
containing $\scrO_2^\ddagger$ are $H^\times$-conjugate if and only if
they are conjugate by $1+i$. 
As defined in 
  (\ref{eq:72}),  $\beth(\scrO_2^\ddagger)=\abs{\ol\calN(\scrO_2^\ddagger)\backslash \grS(\scrO_2^\ddagger)}$.
According to
\cite[63:3]{o-meara-quad-forms}, $F(j)/F$ is unramified at every
dyadic prime of $F$ if and only if $-\varepsilon$ is congruent to a
square modulo $4O_F$. If this is the case, then $O_F[j]=O_F+2O_{F(j)}$ by
(\ref{eq:64}), and hence  
\begin{equation}\label{eq:62}
\O_j:=O_{F(j)}+iO_{F(j)}  
\end{equation}
is a \emph{maximal} order because
$\grd(\O_j)=\chi(\O_j:
\scrO_2^\ddagger)^{-1}\grd(\scrO_2^\ddagger)=(2O_F)^{-2}\cdot
4O_F=O_F$.
If there exists $\varsigma \in F^\times$ such that
$\varepsilon=3\varsigma^2$, then $(j/\varsigma)^2=-3$, and we may
identify $H=\qalg{-1}{-\varepsilon}{F}$ with $H'=\qalg{-1}{-3}{F}$. In
such case, $F(j)=F(j')\simeq F(\sqrt{-3})$, which is unramified at every
dyadic prime of $F$, and $\O_j$ is identified with $\scrO_6$
in (\ref{eq:143}).

\begin{prop}\label{prop:scroII-unified-proof}
   Write $\varepsilon=a+b\sqrt{d}$ with
  $a,b\in \bbN$ if $\varepsilon\in \Z[\sqrt{d}]$. We have the following table for
  $\aleph(\scrO_2^\ddagger)$ and $\beth(\scrO_2^\ddagger)$ in 
  $H=\qalg{-1}{-\varepsilon}{F}$
\begin{center}
\renewcommand{\arraystretch}{1.2}
  \begin{tabular}{|>{$}c<{$}|>{$}c<{$}|>{$}c<{$}|>{$}c<{$}|}
\hline
  d  &a &\aleph(\scrO_2^\ddagger) &
                                    \beth(\scrO_2^\ddagger)\\
\hline
d\equiv 1\pmod{8}& a\equiv
  1\pmod{4} & 1 & 1\\
\hline
d\equiv 1\pmod{8} & a\equiv
  3\pmod{4} & 4 & 2\\
\hline
d\equiv 5\pmod{8} & & 2 & 1\\
\hline
d\equiv 3\pmod{4} & a\text{ is even} & 2 & 2\\
\hline 
\multicolumn{2}{|c|}{\text{otherwise}} & 4 &3\\
\hline
  \end{tabular}
\end{center}
\end{prop}
\begin{proof}
By Lemma~\ref{lem:gen-lem-scro2II}, the Eichler invariant
$e_\grp(\scrO_2^\ddagger)=0$ for every dyadic prime $\grp$ of
$F$. Moreover, $\grd(\scrO_2^\ddagger)=4O_F$, and $\scrO_2^\ddagger$ is maximal at
all finite nondyadic primes of $F$.   It
follows that 
$\aleph(\scrO_2^\ddagger)=\prod_{\grp\vert
  (2O_F)}\aleph_\grp(\scrO_2^\ddagger)$. 

If $d\equiv 1\pmod{8}$ and $a\equiv 1\pmod{4}$, then $H$ is ramified
at the two dyadic primes of $F$. Hence there is a unique maximal order
containing $\scrO_2^\ddagger$, and it is necessarily normalized by
$1+i$.

If $d\equiv 1\pmod{8}$ and $a\equiv 3\pmod{4}$, then $H$ splits at the
two dyadic primes of $F$. By \cite[Corollary~1.6]{Brzezinski-1983},
$\scrO_2^\ddagger$ is a Bass order because $\grd(\scrO_2^\ddagger)$ is cube-free.  Then
Lemma~\ref{lem:bass-order-eichler-inv=0} shows that
$\aleph_\grp(\scrO_2^\ddagger)=2$ for each dyadic prime $\grp$ of $F$,
and hence $\aleph(\scrO_2^\ddagger)=2\cdot 2=4$. Since $(1+i)$ is odd
at every dyadic $\grp$, it does not belong to $\calN(\O)$ for any $\O\in
\grS(\scrO_2^\ddagger)$. Thus conjugation by $(1+i)$ separates
$\grS(\scrO_2^\ddagger)$ into two pairs of maximal orders.  

When $d\equiv 5\pmod{8}$, $\grp=2O_F$ is the unique dyadic prime
in $F$. The same line of argument as the previous case applies here and
produces the desired result. 
 
Now suppose that $2$ is ramified in $F$, i.e.~$d\equiv 2,
3\pmod{4}$. By Table~\ref{tab:orders-K1},  $O_F[i]$ is a proper suborder of $O_{F(i)}$. It then follows from
Lemma~\ref{lem:scro2II-Bass-criterion} that $\scrO_2^\ddagger$ is a
Bass order if and only if $O_F[\sqrt{-\varepsilon}]$ coincides with
the ring of integers of $F(\sqrt{-\varepsilon})$. The latter condition
holds if and only if $d\equiv 3\pmod{4}$ and $a$ is even by 
Lemmas~\ref{lem:criterion-ram-d=3mod4} and
\ref{lem:criterion-ram-d=2mod4}.   Assume that this is the case so that 
$\scrO_2^\ddagger$ is Bass, and we apply
Lemma~\ref{lem:bass-order-eichler-inv=0} again to obtain
$\aleph(\scrO_2^\ddagger)=2$. By Section~\ref{sec:orders-in-K1-K3}, 
we have $\chi(O_{F(i)}, O_F[i])=2O_F$, so 
\begin{equation}\label{eq:70}
\O_i:=O_{F(i)}+jO_{F(i)} 
\end{equation}
is one of the two maximal orders containing
$\scrO_2^\ddagger$. Furthermore, 
\[(1+i)\O_i(1+i)^{-1}=O_{F(i)}+kO_{F(i)}=i(O_{F(i)}+jO_{F(i)} )=\O_i.\]
Hence $\ol\calN(\scrO_2^\ddagger)$ acts trivially on
$\grS(\scrO_2^\ddagger)$.


In the remaining cases, $\scrO_2^\ddagger$ is not a Bass order. Let $\grp$ be the unique dyadic
prime of $F$.  The $\grp$-adic completion $(\scrO_2^\ddagger)_\grp=\scrO_2^\ddagger\otimes \Z_2$ is a Gorenstein order of Eichler invariant $0$ as shown in
Lemma~\ref{lem:gen-lem-scro2II}. By
\cite[Proposition~1.12]{Brzezinski-1983}, there exists a unique
minimal overorder $\calO$ of $(\scrO_2^\ddagger)_\grp$, and $\calO$ is non-Gorenstein
since $(\scrO_2^\ddagger)_\grp$ is not Bass.  We have $\grd(\calO)=\grp^3O_{F_\grp}$ by
\cite[Proposition~4.1]{Brzezinski-1983}.  Now it follows from 
Lemma~\ref{lem:non-Gorenstein-order} that
$\aleph((\scrO_2^\ddagger)_\grp)=\aleph(\calO)=4$. In fact, $\grT(\calO)$ is a star
centered at the Gorenstein saturation $\Gor(\calO)$ with 3 exterior
vertices. The symmetry forces $(1+i)\in \calN(\Gor(\calO))$. To
obtain $\beth(\scrO_2^\ddagger)=3$, it is enough to show that
conjugation by $(1+i)$ acts non-trivially on $\grS(\calO)$.  

Firstly suppose
that $d\equiv 2\pmod{4}$, so $(1+i)/\sqrt{d}$ is integral over
$O_{F_\grp}$.  If $(1+i)/\sqrt{d}$ acts trivially on $\grS(\calO)$,
then $(1+i)/\sqrt{d}\in \O$ for every $\O\in \grS(\calO)$ by
Lemma~\ref{lem:max-order-normalizer}. Thus
\[\frac{1+i}{\sqrt{d}}\in \bigcap_{\O\in \grS(\calO)}\O=\calO. \]
However, $(1+i)/\sqrt{d}$ generates the ring of integers of
$F_\grp(i)$ over $O_{F_\grp}$. In light of
\cite[Proposition~1.11]{Brzezinski-crelle-1990}, this contradicts the
fact that $\calO$ is non-Gorenstein.  

Lastly, suppose that
$d\equiv 3\pmod{4}$ and $a$ is odd. By
Lemma~\ref{lem:criterion-ram-d=3mod4}, $F(j)/F$ is unramified at the
dyadic prime of $F$, so 
the order $\O_j$ in (\ref{eq:62})
 is a maximal order
containing $\scrO_2^\ddagger$. We claim that 
\begin{equation}
  \O_k:=(1+i)\O_j(1+i)^{-1}=O_{F(k)}+iO_{F(k)}  
\end{equation}
is distinct from $\O_j$. Indeed, 
\[\O_j\cap F(k)=O_F+k((j^{-1}O_{F(j)})\cap
F)=O_F+k(O_{F(j)}\cap F)=O_F+kO_F\neq O_{F(k)}. \]
The proposition is proved. 
\end{proof}

\begin{cor}\label{cor:scO2typeII}
  Let $H=\qalg{-1}{-\varepsilon}{F}$. Then $t(D_2^\ddagger)=0$ if $d=6$. For $d\geq 7$,
\begin{gather}
t(D_2^\ddagger)+t(D_4)+t(S_4)+t(D_6)=\beth(\scrO_2^\ddagger), \label{eq:82}\\
\text{where either} \quad t(D_6)=0 \quad \text{or}\quad t(S_4)=t(D_4)=0.\label{eq:83}
\end{gather}
Particularly, if $\{2\varepsilon, 3\varepsilon\}\cap F^{\times
  2}=\emptyset$, then 
\begin{equation}
  \label{eq:84}
  t(D_2^\ddagger)=\beth(\scrO_2^\ddagger). 
\end{equation}
\end{cor}
\begin{proof}
  For every $\O\in \grS(\scrO_2^\ddagger)$, we have
  $\O^\bs\in \{D_2^\ddagger, D_4, S_4, D_6\}$ by Figure~\ref{fig:rel-gps}.  Hence (\ref{eq:82})
  follows from \eqref{eq:147}, and it holds for all $d\geq 6$.  Formula (\ref{eq:83})
  holds because $2\varepsilon$ and $3\varepsilon$ cannot be 
  perfect squares simultaneously  in $F$ when $d\geq 7$. When $\{2\varepsilon, 3\varepsilon\}\cap F^{\times
  2}=\emptyset$,  every term other than $t(D_2^\ddagger)$ on
the left side of (\ref{eq:82}) is zero, so \eqref{eq:82} simplifies to \eqref{eq:84}.

 Lastly, if $d=6$, then $H=\qalg{-1}{-\varepsilon}{F}$ splits at all finite primes of $F$. Hence $H\simeq \qalg{-1}{-1}{F}\simeq \qalg{-1}{-3}{F}$, so $t(S_4)=t(D_4)=t(D_6)=1$. It follows from Proposition~\ref{prop:scroII-unified-proof} and \eqref{eq:82}  that $t(D_2^\ddagger)=0$.
\end{proof}

By Corollary~\ref{cor:scO2typeII}, we  may assume that $d\geq 7$ for the rest of the discussion.

\begin{lem}\label{lem:red-unit-gp-Oj}
 Assume that $F(j)/F$ is unramified at every dyadic prime of
  $F$ so that $\O_j$ in (\ref{eq:62}) is a maximal order containing
  $\scrO_2^\ddagger$. Then $\O_j^\bs$ is isomorphic to neither $D_4$ nor $S_4$. 
\end{lem}
\begin{proof}
  Clearly, the lemma holds if $2\varepsilon\not\in F^{\times
    2}$, so assume that $2\varepsilon\in F^{\times
    2}$. Suppose that $\O_j^\bs$ is isomorphic to either
  $D_4$ or $S_4$. Then there exists an element $\tilde{v}\in \O_j^\bs$ of
  order $4$ such that $\tilde{v}^2\in
  (\scrO_2^\ddagger)^\bs$. However, $\tilde{i}\in
  (\scrO_2^\ddagger)^\bs$ is the unique element of order $2$ with
  $\Nr(\tilde{i})=1$. So we must have $\tilde{v}^2=\tilde{i}$. This
  leads to a contradiction since $\O_j\cap F(i)=O_F[i]\not\simeq
  O_F[\sqrt{\varepsilon i}]$ by Lemma~\ref{lem:OF-eta-K1}. 
\end{proof}
\begin{sect}\label{sec:d=1mod8scro2II}
  We keep the notation and assumptions of
  Proposition~\ref{prop:scroII-unified-proof}. 
  If $d\equiv 1\pmod{8}$ and $a\equiv 1\pmod{4}$, then 
  \begin{equation}
    \O=O_F+O_Fi+O_Fj+O_F\frac{1+i+j+k}{2}\subset H=\qalg{-1}{-\varepsilon}{F}
  \end{equation}
is the unique maximal order containing $\scrO_2^\ddagger$. Since 
$\omega(H)=2$ by Lemma~\ref{lem:ram-places-D2II}, the reduced unit
group $\O^\bs$ cannot be isomorphic to $S_4$, $D_4$, nor $D_6$, otherwise $H$ 
 splits at all finite primes of $F$. 
Therefore, we have
\begin{alignat}{2}
\O^\bs&\simeq D_2^\ddagger,&\qquad  \calN(\O)&=F^\times\O^\times\<1+i\>, \quad \text{and}\\
t(D_2^\ddagger)&=1,  &\qquad h(D_2^\ddagger)&=2h(F). 
\end{alignat}
Note that $-\varepsilon$ coincides with $3$ in $O_F/4O_F\simeq \zmod{4}\times\zmod{4}$ in this case. So $F(j)/F$
is ramified at both the dyadic primes of $F$.

Suppose that $d\equiv 1\pmod{8}$ and $a\equiv 3\pmod{4}$. Then
$-\varepsilon\equiv 1\pmod{4O_F}$, and $F(j)/F$
is unramified at every dyadic prime of $F$. We
leave it to the reader to check that the following are two distinct maximal
orders containing $\scrO_2^\ddagger$:
\begin{align}
  \O_j&=O_{F(j)}+iO_{F(j)}=O_F+O_Fi+O_F\frac{1+j}{2}+O_F\frac{i+k}{2},\\
  \begin{split}\label{eq:66}
  \O&=O_F+O_Fi+O_F\frac{(-1+\sqrt{d})+(1+\sqrt{d})i+2j}{4}\\&\phantom{=}+O_F\frac{(1+\sqrt{d})+(-1+\sqrt{d})i+2k}{4}. 
  \end{split}
  \end{align}
  We have $(1+i)\O_j(1+i)^{-1}=\O_k$, which coincides with neither
  $\O_j$ nor $\O$. Since $2$ splits in $F$, $\O_j^\bs$ cannot be
  isomorphic to $S_4$ or $D_4$, and the same for $\O^\bs$. It
  can happen that $3\varepsilon\in F^{\times 2}$ (e.g.~when $d=3p>9$
  with $p\equiv 3\pmod{8}$), in which case
  $\qalg{-1}{-\varepsilon}{F}=\qalg{-1}{-3}{F}$, and $\O_j$ coincides
  with $\scrO_6\subseteq \qalg{-1}{-3}{F}$ in (\ref{eq:63}). On the other hand,
  $\O^\bs\not\simeq D_6$ since $\O$ is not $H^\times$-conjugate to
  $\O_j$. Therefore,
\begin{align}
\label{eq:68}  \O_j^\bs&=
  \begin{cases}
    D_2^\ddagger\qquad\text{if } 3\varepsilon\not\in F^{\times 2},\\
    D_6\qquad\text{if } 3\varepsilon\in F^{\times 2},\\
  \end{cases} &  \calN(\O_j)&=F^\times\O_j^\times;\\
\O^\bs&=D_2^\ddagger,  &  \calN(\O)&=F^\times\O^\times. 
\end{align}
By Lemma~\ref{lem:ram-places-D2II}, $\omega(H)=0$ in this case, and hence
\begin{equation}
  \label{eq:154}
  t(D_2^\ddagger)=  \begin{cases}
    2\qquad&\text{if } 3\varepsilon\not\in F^{\times 2},\\
    1\qquad&\text{if } 3\varepsilon\in F^{\times 2}, \\
  \end{cases}\quad\text{and}\quad h(D_2^\ddagger)=h(F)t(D_2^\ddagger). 
\end{equation}
\end{sect}

\begin{sect}
Now suppose that $d\equiv 5\pmod{8}$ so that $\grp=2O_F$ is the unique
dyadic prime of $F$.  If $F(j)$ is unramified at $\grp$, then
$\grS(\scrO_2^\ddagger)=\{\O_j, \O_k\}$. Once again, it can happen
that $3\varepsilon\in F^{\times 2}$ (e.g.~when $d=3p$ with
$p\equiv 7\pmod{8}$), in which case $\O_j$ coincides with
$\scrO_6\subset \qalg{-1}{-3}{F}$ as before. Moreover, 
(\ref{eq:68}) still holds for $\O_j$. 

Suppose that $F(j)/F$
is ramified at $\grp$. Then $O_F[j]=O_{F(j)}$ by
Lemma~\ref{lem:d=1mod4ram-max}. 
According to Lemma~\ref{lem:CM-ext-fund-unit-d=5mod8}, there are three subcases to
consider.
\begin{enumerate}[(i)]
\item If
$\varepsilon=a+b\sqrt{d}\in \Z[\sqrt{d}]$, then $a\equiv 1\pmod{4}$
and $4\mid b$. Hence $\O$ in (\ref{eq:66}) is a maximal order
containing $\scrO_2^\ddagger$.
\item  If $\varepsilon=(a+b\sqrt{d})/2$ with
$a$ odd and $b\equiv 1\pmod{4}$, then 
 \begin{equation}
   \label{eq:65}
   \O=O_F+O_Fi+O_F\frac{(-1+\sqrt{d})+2i+2j}{4}+O_F\frac{2+(-1+\sqrt{d})i+2k}{4}
 \end{equation}
is a maximal order containing $\scrO_2^\ddagger$.
\item  If $\varepsilon=(a+b\sqrt{d})/2$ with $a$ odd and
$b\equiv 3\pmod{4}$, then 
\begin{equation}
  \label{eq:67}
   \O=O_F+O_Fi+O_F\frac{(1+\sqrt{d})+2i+2j}{4}+O_F\frac{2+(1+\sqrt{d})i+2k}{4}
\end{equation}
is a maximal order containing $\scrO_2^\ddagger$. 
\end{enumerate}
In all three cases, we have 
\begin{equation}
  \label{eq:69}
  \O^\bs\simeq D_2^\ddagger, \quad \text{and}\quad \calN(\O)=F^\times\O^\times.
\end{equation}
Summarizing, if $d\equiv 5\pmod{8}$, then $\omega(H)=0$, and 
\begin{equation}
  \label{eq:155}
  t(D_2^\ddagger)=
  \begin{cases}
    1\qquad&\text{if } 3\varepsilon\not\in F^{\times 2},\\
    0 \qquad&\text{if } 3\varepsilon\in F^{\times 2},\\
  \end{cases}\quad \text{and}\quad
  h(D_2^\ddagger)=h(F)t(D_2^\ddagger). 
\end{equation}
\end{sect}

\begin{sect}
Next, suppose that $d\equiv 3\pmod{4}$ and $\varepsilon=a+b\sqrt{d}$
with $a$ even. Then $F(j)/F$ is ramified at the unique dyadic prime
$\grp$ of
$F$, and $O_{F(j)}=O_F[j]$ by
Lemma~\ref{lem:criterion-ram-d=3mod4}. We have
$\grS(\scrO_2^\ddagger)=\{\O_i, \O\}$, where $\O_i$ is defined in
(\ref{eq:70}),  and $ \O=O_{F(j)}+O_{F(j)}\frac{1+i+(1+\sqrt{d})j}{2}$. More explicitly, 
\begin{equation}
  \label{eq:71}
  \O=O_F+O_F\frac{1+i+(1+\sqrt{d})j}{2}+O_Fj+O_F\frac{(1+\sqrt{d})+j+k}{2}. 
\end{equation}
Clearly, $F(j)\not\simeq F(\sqrt{-3})$ since the latter is unramified
at $\grp$. Hence $3\varepsilon\not\in F^{\times 2}$. On the other
hand, it is possible that $2\varepsilon\in F^{\times 2}$ (e.g.~when
$d=p$ with $p\equiv 3\pmod{4}$).
The reduced unit groups and normalizers are given by the following table
\begin{center}
\renewcommand{\arraystretch}{1.3}
  \begin{tabular}{*{5}{|>{$}c<{$}}|}
\hline
& \O_i^\bs & \calN(\O_i)&\O^\bs & \calN(\O)\\
\hline
2\varepsilon\not\in F^{\times 2} & D_2^\ddagger &
                                                  F^\times\O_i^\times\<1+i\>&
    D_2^\ddagger& F^\times\O^\times\<1+i\>\\
\hline
2\varepsilon\in F^{\times 2} & D_4 & F^\times\O_i^\times
                                                  & S_4 & F^\times\O^\times\\
\hline
\end{tabular}
\end{center}
In this case we have
\begin{equation}
  \label{eq:157}
  t(D_2^\ddagger)=
  \begin{cases}
    2\qquad &\text{if } 2\varepsilon\not\in F^{\times 2},\\
    0 \qquad &\text{if } 2\varepsilon\in F^{\times 2},\\
  \end{cases}\quad\text{and}\quad h(D_2^\ddagger)=h(F)t(D_2^\ddagger)/2. 
\end{equation}
\end{sect}

\begin{sect}\label{sect:D2-typeII-lastcase}
  We treat the remaining cases where one of the following is
  true:
  \begin{itemize}
  \item $d\equiv 2\pmod{4}$ and $d>6$;
  \item $d\equiv 3\pmod{4}$ and $\varepsilon=a+b\sqrt{d}$ with $a$ odd.
  \end{itemize}
Let $\grp$ be the unique dyadic prime of $F$. Then
$\grT(\scrO_2^\ddagger\otimes\Z_2)$ is a star on which
conjugation by $(1+i)$ acts as a reflection, and
$\aleph(\scrO_2^\ddagger)=4$. For simplicity, let 
$\varrho=(1+i+j+k)/2\in \qalg{-1}{-\varepsilon}{F}$, and $B$ be the
order of $F(i)\subset \qalg{-1}{-\varepsilon}{F}$ defined in (\ref{eq:53}). 
We claim that 
\begin{equation}
  \label{eq:73}
  \O_0=B+B\varrho
\end{equation}
is the maximal order in $\grS(\scrO_2^\ddagger)$ corresponding to the center of the star
$\grT(\scrO_2^\ddagger\otimes\Z_2)$. Clearly,  $\varrho
  i=-i\varrho+(-1+i)$. For any $x+yi\in B$ with $x, y\in F$, we
  have $y(-1+i)\in B$, and hence $\varrho(x+yi)=(x-yi)\varrho+y(-1+i)\in
  \O_0$. Therefore, $\O_0$ is an order containing $\scrO_2^\ddagger$,
  and it is maximal because $\chi(\O_0, \scrO_2^\ddagger)=4O_F$. Let $\pi=\sqrt{d}$ if
  $d\equiv 2\pmod{4}$, and 
  $\pi=1+\sqrt{d}$ if $d\equiv 3\pmod{4}$. Then $\grp=(2, \pi)$, and
  the norm of $(1+i)\pi/2\in B$ over $F$ is a $\grp$-adic
  unit. It follows that $\scrO_2^\ddagger\subseteq O_F+\grp
  \O_0$. Thus $\grT(\scrO_2^\ddagger\otimes\Z_2)$ is
  centered at $(\O_0)_\grp$. Although the expression looks similar, $\O_0$ is
  \emph{not}  isomorphic to $\O_{12}$ in (\ref{eq:74}).  Indeed, if $2\varepsilon\not\in F^{\times 2}$, then $\O_{12}^\bs\simeq A_4$, which does not contain any subgroup strictly isomorphic to $D_2^\ddagger$; if $2\varepsilon\in F^{\times 2}$, then $\O_{12}=\O_{24}$, and $\O_0\not\simeq \O_{24}$
  by
  Remark~\ref{rem:suborders-of-O24}. 

 First suppose that $d\equiv 3\pmod{4}$ and $\varepsilon=a+b\sqrt{d}$
 with $a$ odd. Then $\grS(\scrO_2^\ddagger)=\{\O_i, \O_j, \O_k,
 \O_0\}$, and $(1+i)\O_j(1+i)^{-1}=\O_k$. Note that
 $2\varepsilon\not\in F^{\times 2}$ in this case, otherwise
 $F(j)\simeq \Q(\sqrt{d}, \sqrt{-2})$ is totally ramified over
 $\Q$ at $2$, contradicting to Lemma~\ref{lem:criterion-ram-d=3mod4}. On the other hand, it is
 possible that $3\varepsilon\in F^{\times 2}$ (e.g.~$d=3p$ with
 $p\equiv 1\pmod{4}$). The reduced unit groups and normalizers are given by the following table
\begin{center}
\renewcommand{\arraystretch}{1.3}
  \begin{tabular}{*{7}{|>{$}c<{$}}|}
\hline
&\O_j^\bs&\calN(\O_j)&\O_0^\bs & \calN(\O_0)& \O_i^\bs & \calN(\O_i)  \\
\hline
3\varepsilon\not\in F^{\times 2} & D_2^\ddagger&
                                                 \multirow{2}{*}{$F^\times\O_j^\times$} & \multirow{2}{*}{$D_2^\ddagger$} &
                                                  \multirow{2}{*}{$F^\times\O_0^\times\<1+i\>$}
 &\multirow{2}{*}{$D_2^\ddagger$} & \multirow{2}{*}{$F^\times\O_i^\times\<1+i\>$}\\

\cline{1-2}
3\varepsilon\in F^{\times 2} &D_6  & & &  &  & \\
\hline
\end{tabular}
\end{center}
In this case, we have 
\begin{equation}
  \label{eq:158}
  t(D_2^\ddagger)=
  \begin{cases}
    3\qquad &\text{if } 3\varepsilon\not\in F^{\times 2},\\
    2 \qquad &\text{if } 3\varepsilon\in F^{\times 2},\\
  \end{cases}\quad\text{and}\quad h(D_2^\ddagger)=  \begin{cases}
    2h(F) \qquad &\text{if } 3\varepsilon\not\in F^{\times 2}.\\
    h(F) \qquad &\text{if } 3\varepsilon\in F^{\times 2},\\
  \end{cases}
\end{equation}
For the remaining subsection,  assume that $d\equiv 2\pmod{4}$ and $d>6$.  We then have 
$\varepsilon=a+b\sqrt{d}$ with $a$ odd and $b$ even. Hence 
\begin{equation}
  \label{eq:75}
  \O_i':=O_F+O_F\frac{1+i+\sqrt{d}j}{2}+O_Fj+O_F\frac{\sqrt{d}+j+k}{2}
\end{equation}
is an order containing $\scrO_2^\ddagger$. Moreover,
$(1+i)\O_i'(1+i)^{-1}=\O_i'$. Note that $\O_i'\neq \O_0$ since
$(1+i+j+k)/2\not\in \O_i'$.  

First suppose further that $F(j)/F$ is
unramified at $\grp$ (see Lemma~\ref{lem:criterion-ram-d=2mod4}). Then
$\grS(\scrO_2^\ddagger)=\{\O_i', \O_0, \O_j, \O_k\}$. 
It can happen that $2\varepsilon\in F^{\times 2}$ (e.g.~when $d=2p$ with
$p\equiv 3\pmod{4}$) or $3\varepsilon\in F^{\times 2}$ (e.g.~when
$d=78$ or $222$) in this case.  By
Lemma~\ref{lem:red-unit-gp-Oj}, $\O_j^\bs$ is  isomorphic to neither
$D_4$ nor $S_4$.  When $2\varepsilon\in F^{\times 2}$, we must have
$\{\O_0^\bs, \O_i'^\bs\}=\{D_4, S_4\}$. It has already been remarked
that $\O_0\not\simeq \O_{24}$. Thus $\O_0^\bs\simeq D_4$ and
$\O_i'^\bs\simeq S_4$ when $2\varepsilon\in F^{\times 2}$. The reduced
unit groups and normalizers are given by the following table
\begin{center}
\renewcommand{\arraystretch}{1.3}
  \begin{tabular}{*{7}{|>{$}c<{$}}|}
\hline
&\O_j^\bs&\calN(\O_j)&\O_0^\bs & \calN(\O_0)& \O_i'^\bs & \calN(\O_i')  \\
\hline
\{2\varepsilon, 3\varepsilon\}\cap F^{\times 2}=\emptyset & D_2^\ddagger&
                                                 \multirow{3}{*}{$F^\times\O_j^\times$} & \multirow{2}{*}{$D_2^\ddagger$} &
                                                  \multirow{2}{*}{$F^\times\O_0^\times\<1+i\>$}
 &\multirow{2}{*}{$D_2^\ddagger$} & \multirow{2}{*}{$F^\times\O_i'^\times\<1+i\>$}\\

\cline{1-2}
3\varepsilon\in F^{\times 2} &D_6  & & &  &  & \\
\cline{1-2}\cline{4-7}
2\varepsilon\in F^{\times 2} & D_2^\ddagger  & & D_4 & F^\times\O_0^\times & S_4 & F^\times\O_i'^\times\\
\hline
\end{tabular}
\end{center}
In this case, we have 
\begin{equation}
  \label{eq:159}
 t(D_2^\ddagger)=
  \begin{cases}
    3 \quad &\text{if } \{2\varepsilon, 3\varepsilon\}\cap F^{\times 2}=\emptyset,\\
    2\quad &\text{if } 3\varepsilon\in F^{\times 2},\\
    1\quad &\text{if } 2\varepsilon\in F^{\times 2},\\
  \end{cases}\quad h(D_2^\ddagger)=  \begin{cases}
    2h(F) \quad &\text{if } \{2\varepsilon, 3\varepsilon\}\cap F^{\times 2}=\emptyset,\\
    h(F)\quad &\text{otherwise.}
\end{cases}
\end{equation}

Lastly, suppose that $d\equiv 2\pmod{4}$ and $F(j)/F$ is ramified at
$\grp$.  By Lemma~\ref{lem:criterion-ram-d=2mod4}, if $4\mid b$, then
$a\equiv 1\pmod{4}$, and hence 
\begin{equation}
  \label{eq:76}
  \O_j'=O_F+O_Fi+O_F\frac{\sqrt{d}+i+j}{2}+O_F\frac{1+\sqrt{d}i+k}{2}
\end{equation}
is a maximal order containing $\scrO_2^\ddagger$; if $b\equiv
2\pmod{4}$, then $a\equiv 3\pmod{4}$, and hence 
\begin{equation}
  \label{eq:77}
  \O_j'=O_F+O_Fi+O_F\frac{1+\sqrt{d}+\sqrt{d}i+j}{2}+O_F\frac{\sqrt{d}+(1+\sqrt{d})i+k}{2}
\end{equation}
is a maximal order containing $\scrO_2^\ddagger$. Let
$\O_k'=(1+i)\O_j'(1+i)^{-1}$. It is straightforward to check that
$\O_j'\neq \O_k'$ in both cases. So
$\grS(\scrO_2^\ddagger)=\{\O_0, \O_i', \O_j', \O_k'\}$. Since
$F(\sqrt{-3})/F$ is unramified at $\grp$, we have
$3\varepsilon\not\in F^{\times 2}$ in this case. On the other hand, it
is possible that $2\varepsilon\in F^{\times 2}$ (e.g.~when $d=2p$ for
some prime $p\equiv 1\pmod{8}$ and $p$ \emph{not} of the form
$x^2+32y^2$ for any $x,y\in \bbZ$. See \cite[Corollary~24.5]{Conner-Hurrelbrink}).
The reduced
unit groups and normalizers are given by the following table
\begin{center}
\renewcommand{\arraystretch}{1.3}
  \begin{tabular}{*{7}{|>{$}c<{$}}|}
\hline
&\O_j'^\bs&\calN(\O_j')&\O_0^\bs & \calN(\O_0)& \O_i'^\bs & \calN(\O_i')  \\
\hline
2\varepsilon\not\in F^{\times 2} & \multirow{2}{*}{$D_2^\ddagger$}&
                                                 \multirow{2}{*}{$F^\times\O_j'^\times$}
                     & D_2^\ddagger&
                                                 F^\times\O_0^\times\<1+i\>& D_2^\ddagger&
                                                 F^\times\O_i'^\times\<1+i\>\\

\cline{1-1}\cline{4-7}
2\varepsilon\in F^{\times 2} &  & & D_4 & F^\times\O_0^\times & S_4 & F^\times\O_i'^\times\\
\hline
\end{tabular}
\end{center}
In this case, we have 
\begin{equation}
  \label{eq:160}
    t(D_2^\ddagger)=
  \begin{cases}
    3\qquad &\text{if } 2\varepsilon\not\in F^{\times 2},\\
    1 \qquad &\text{if } 2\varepsilon\in F^{\times 2},\\
  \end{cases}\quad\text{and}\quad h(D_2^\ddagger)=  \begin{cases}
    2h(F) \qquad &\text{if } 2\varepsilon\not\in F^{\times 2},\\
    h(F) \qquad &\text{if } 2\varepsilon\in F^{\times 2}.\\
  \end{cases}
\end{equation}
Combining \eqref{eq:158}, \eqref{eq:159}, \eqref{eq:160}, we see that under the assumption of Section~\ref{sect:D2-typeII-lastcase} on $d$ and $\varepsilon$, \begin{equation}
    h(D_2^\ddagger)=\begin{cases}
      2h(F) \quad&\text{if } \{2\varepsilon, 3\varepsilon\} \cap F^{\times 2}=\emptyset,\\
      h(F) \quad&\text{otherwise.}
    \end{cases}
\end{equation} 
\end{sect}

\subsection{Maximal orders containing $\scrO_3^\ddagger$}
Throughout this subsection, $H$ denotes the quaternion algebra $\qalg{-\varepsilon}{-3}{F}$.  We study  the maximal orders in $H$ containing the minimal $D_3^\ddagger$-order 
$\scrO_3^\ddagger=O_F[i, \frac{1+j}{2}]$. By
\cite[Corollary~1.6]{Brzezinski-1983}, $\scrO_3^\ddagger$ is always a
Bass order since $\grd(\scrO_3^\ddagger)=3O_F$ is cube-free.  We
first determine the finite ramified primes of
$H=\qalg{-\varepsilon}{-3}{F}$. Write
$\varepsilon=\frac{a+b\sqrt{d}}{2}$, where $a,b$ are positive integers
such that $a\equiv b\pmod{2}$. If $d\equiv 1\pmod{3}$ and
$\Nm_{F/\Q}(\varepsilon)=1$, then $3\mid b$, which is immediately seen
by taking both sides of $a^2-b^2d=4$ modulo $3$. Therefore,
$\varepsilon\equiv \pm 1\pmod{3O_F}$ in this case.

\begin{lem}\label{lem:alg-ramified-places-D3-II}
  Let $H=\qalg{-\varepsilon}{-3}{F}$. Then $H$ splits at all finite places
  of $F$ coprime to $3$. If $d\not\equiv 1\pmod{3}$, then
  $H$ splits at the unique prime of $F$ above $3$ as well. When $d\equiv
  1\pmod{3}$, $H$
   splits at the two primes of $F$ above $3$ if and only if 
   $\varepsilon \equiv -1 \pmod{3O_F}$. 
\end{lem}
\begin{proof}
   Since $\grd(H)$ divides $\grd(\scrO_3^\ddagger)=3O_F$, $H$ splits at
  all finite places of $F$ coprime to $3$.  If $F$ has a unique prime
   $\grq$ above $3$, i.e.~$d\not\equiv 1\pmod{3}$, then $H$ splits
  at $\grq$ as well because it splits  at an even number of
  places of $F$. 
  
  Lastly, suppose that $d\equiv 1\pmod{3}$. Let $\grq$
  be a prime of
  $F$ above $3$. Then $F_\grq=\Q_3$. By Hensel's lemma, the quadratic form 
  $-\varepsilon x^2-3y^2$ represents $1$ with $x,y\in \Q_3$ if and
  only if $\varepsilon\equiv -1\pmod{3}$. Therefore, the Hilbert symbol $(-\varepsilon,
  -3)_\grq=1$ (i.e.~$H$ splits at $\grq$) if and only if
  $\varepsilon\equiv -1\pmod{3}$. 
\end{proof}

\begin{prop}\label{prop:scrO3-typeII}
  Let 
  $H=\qalg{-\varepsilon}{-3}{F}$,  and $\grq$ be a prime of $F$
  above $3$.   The values of 
  $\aleph(\scrO_3^\ddagger)$ and $\beth(\scrO_3^\ddagger)$ are listed in the
  following table
\begin{center}
\renewcommand{\arraystretch}{1.2}
  \begin{tabular}{|>{$}c<{$}|>{$}c<{$}|>{$}c<{$}|>{$}c<{$}|}
\hline
  d\geq 6  &\varepsilon &\aleph(\scrO_3^\ddagger) &\beth(\scrO_3^\ddagger)\\                                 
\hline
\multirow{2}{*}{$d\equiv 0\pmod{3}$} & \varepsilon\equiv 1 \pmod{\grq} & 1 & 1\\
\cline{2-4}
 & \varepsilon\equiv
  -1\pmod{\grq} & 3 & 2\\
\hline
\multirow{2}{*}{$d\equiv 1\pmod{3}$} & \varepsilon\equiv 1 \pmod{3O_F} & 1 & 1\\
\cline{2-4}
 & \varepsilon\equiv -1 \pmod{3O_F} & 4 & 2\\
\hline 
d\equiv 2\pmod{3}& & 2 &1\\
\hline
  \end{tabular}
\end{center}
\end{prop}
\begin{proof}
  First suppose that $3\mid d$ so that $\grq=(3, \sqrt{d})$. Since $F(j)\simeq F(\sqrt{-3})$, we have 
  $O_F[(1+j)/2]=O_F+\grq O_{F(j)}$ by Table~\ref{tab:orders-K3}.
  Hence $\O_j=O_{F(j)}+iO_{F(j)}$ is a maximal order containing $\scrO_3^\ddagger$
  with $j\in \calN(\O_j)$.  By Lemma~\ref{lem:gen-lem-scro3II},
  $\scrO_3^\ddagger$ is maximal at all finite places of $F$ coprime to
  $\grq$, and
\[e_\grq(\scrO_3^\ddagger)=
 \begin{cases}
  -1           &\text{if} \quad \varepsilon \equiv 1\pmod{\grq};\\
  1 &\text{if} \quad \varepsilon \equiv 2\pmod{\grq}.
 \end{cases}\]
If $e_\grq(\scrO_3^\ddagger)=-1$, then
$\aleph_\grq(\scrO_3^\ddagger)=1$ by
\cite[Corollary~3.2]{Brzezinski-1983}.  Thus $\O_j$ is the unique
maximal order containing $\scrO_3^\ddagger$.  Suppose that
$e_\grq(\scrO_3^\ddagger)=1$ next. Then
$\scrO_3^\ddagger$ is an Eichler order of level $3O_F=\grq^2$ by
\cite[Corollary~2.2]{Brzezinski-1983}. Hence
$\aleph(\scrO_3^\ddagger)=3$. Let $\O$ and $\O'$ be the remaining two maximal orders distinct
from $\O_j$ that contain
$\scrO_3^\ddagger$. Then $\O\cap
F(j)=O_F[(1+j)/2]$.
Otherwise, we have $O_{F(j)}\subseteq \O$, and hence
$\O\subseteq \O_j$, which contradicts our assumption. For
simplicity, write $R=O_{F_\grq}$. By \cite[Theorem~II.3.2]{vigneras},
there exists an isomorphism $\O\otimes_{O_F}R\simeq M_2(R)$ such
that $(1+j)/2$ is identified with $
\begin{pmatrix}
  0& 1 \\ 1& 1
\end{pmatrix}$. Then $j$ is identified with $
\begin{pmatrix}
  -1 & 2 \\ 2 &1
\end{pmatrix}$, which does not normalize $M_2(R)$. It follows that
$j\not\in \calN(\O)$. Therefore, $j\O j^{-1}=\O'$, and hence
$\beth(\scrO_3^\ddagger)=2$. 

Next, suppose that $d\equiv 1\pmod{3}$. By
Lemma~\ref{lem:alg-ramified-places-D3-II}, if
$\varepsilon\equiv 1\pmod{3O_F}$, then $H$ is ramified at the two
places of $F$ above $3$ and splits at all other finite places. Hence
$\grd(H)=3O_F=\grd(\scrO_3^\ddagger)$, which implies that
$\scrO_3^\ddagger$ is maximal.  Suppose next that
$\varepsilon\equiv -1\pmod{3O_F}$. Then $H$ splits at all finite places
of $F$.  For any prime $\grq$ of $F$ above $3$, we have
$\grd((\scrO_3^\ddagger)_\grq)=3 O_{F_\grq}=\grq O_{F_\grq}$.  It
follows that $(\scrO_3^\ddagger)_\grq$ is an Eichler order of level
$\grq O_{F_\grq}$, and hence $\aleph_\grq(\scrO_3^\ddagger)=2$.
Therefore,
$\aleph(\scrO_3^\ddagger)=\prod_{\grq\mid
  3O_F}\aleph_\grq(\scrO_3^\ddagger)=2\cdot 2=4$.
By Section~\ref{sect:parity-of-element}, $j\not\in \calN(\O)$ for
any maximal order $\O$ containing $\scrO_3^\ddagger$ since it is odd
at $\grq$. Hence conjugation by $j$ acts on $\grS(\scrO_3^\ddagger)$ as the product of two
disjoint transpositions, and
$\beth(\scrO_3^\ddagger)=2$. 

Lastly,  the calculation of
$\aleph(\scrO_3^\ddagger)$ and $\beth(\scrO_3^\ddagger)$ for $d\equiv 2\pmod{3}$ is similar to
the case $d\equiv 1\pmod{3}$ and $\varepsilon\equiv -1\pmod{3O_F}$ above,
hence omitted.
\end{proof}

\begin{cor}\label{cor:D3typeII}
Let 
  $H=\qalg{-\varepsilon}{-3}{F}$. Then $t(D_3^\ddagger)=0$ if $d=6$. For $d\geq 7$, 
  \begin{gather}
t(D_3^\ddagger)+t(S_4)+t(D_6)=\beth(\scrO_3^\ddagger),\label{eq:200}\\    
\text{where either} \quad t(S_4)=0 \quad \text{or}\quad t(D_6)=0. 
  \end{gather}
Particularly, if $\{2\varepsilon, 3\varepsilon\}\cap F^{\times
  2}=\emptyset$, then 
\begin{equation}
  \label{eq:85}
t(D_3^\ddagger)=\beth(\scrO_3^\ddagger). 
\end{equation}
\end{cor}
The proof is similar to that of Corollary~\ref{cor:scO2typeII}, 
hence omitted.

\begin{sect}
 Assume that $3\mid d$ with $d>6$. By the proof of   Proposition~\ref{prop:scrO3-typeII},  $\grS(\scrO_3^\ddagger)\supseteq \{\O_j\}$, and $j\in \calN(\O_j)$. 
If $3\varepsilon\in F^{\times 2}$, then $\O_j^\bs \simeq D_6$ since $O_{F(j)}^\bs\simeq \zmod{6}$ in this case.
For example, if $d=3p$ with
$p>3$ and $p\equiv 3\pmod{4}$, then $3\varepsilon\in F^{\times 2}$ by
\cite[Lemma~3(a)]{MR0441914}. Both cases $\varepsilon\equiv \pm 1 \pmod{\grq}$ may occur when $3\varepsilon\in F^{\times 2}$, as shown by the examples $d=21$ and $d=33$.
On the other hand, $\O_j^\bs$ is never isomorphic to  $S_4$. Indeed, $F(\xi)\cap \O_{24}=O_F[\xi]\neq O_{F(\xi)}$ in $\qalg{-1}{-1}{F}$ by \eqref{eq:144}, so $\O_j\not\simeq \O_{24}$.  Therefore, $\O_j^\bs\simeq D_3^\ddagger$ if $3\varepsilon\not\in F^{\times 2}$, in which case $\calN(\O_j)=F^\times\O_j^\times\dangle{j}$. 

Let us write
\begin{equation}
  \label{eq:79}
 \varepsilon=a+b\sqrt{d}\quad \text{with}\quad a, b\in \frac{1}{2}\Z\quad\text{and}\quad
a\equiv b\pmod{\Z}.  
\end{equation}
First assume that $a\equiv 1\pmod{\frac{3}{2}\Z}$ so that $\varepsilon\equiv 1\pmod{\grq}$, then $\grS(\scrO_3^\ddagger)=\{\O_j\}$. It follows that
\begin{equation}
  \label{eq:162}
    t(D_3^\ddagger)=
  \begin{cases}
    1\qquad &\text{if }  3\varepsilon\not\in  F^{\times 2},\\
    0 \qquad &\text{if } 3\varepsilon\in  F^{\times 2},
  \end{cases}\quad\text{and}\quad
  h(D_3^\ddagger)=\frac{1}{2}h(F)t(D_3^\ddagger). 
\end{equation}
Next, assume that $a\equiv -1\pmod{\frac{3}{2}\Z}$. We
define
\begin{gather}
  \label{eq:78}
\O=\scrO_3^\ddagger+O_F\delta=O_F+O_Fi+O_F\frac{1+j}{2}+O_F\delta, \qquad\text{where}\\
\delta=\begin{cases}
\frac{-3i+2j+k}{6}\qquad &\text{if } b\equiv
0\pmod{\frac{3}{2}\Z};\\
\frac{-3i+2(1+\sqrt{d})j+k}{6}\qquad &\text{if } b\equiv
1\pmod{\frac{3}{2}\Z};\\
\frac{-3i+2(-1+\sqrt{d})j+k}{6}\qquad &\text{if } b\equiv
2\pmod{\frac{3}{2}\Z}.
\end{cases}
\end{gather}
Then $\O$ is a maximal order containing $\scrO_3^\ddagger$ and
distinct from $\O_j$. It follows from the proof of 
Proposition~\ref{prop:scrO3-typeII} that
$\grS(\scrO_3^\ddagger)=\{\O, \O_j, j\O j^{-1}\}$. 

In this case, it is possible that $2\varepsilon\in F^{\times 2}$,  as demonstrated by the example
$d=66$ and $\varepsilon=65+8\sqrt{66}$. Since $d>6$ by our assumption, $2\varepsilon$ and $3\varepsilon$ are not simultaneously perfect squares in $F$.  The reduced
unit groups and normalizers are given by the following table
\begin{center}
\renewcommand{\arraystretch}{1.3}
  \begin{tabular}{*{5}{|>{$}c<{$}}|}
\hline
&\O_j^\bs&\calN(\O_j)&\O^\bs & \calN(\O)\\
\hline
\{2\varepsilon, 3\varepsilon\}\cap F^{\times 2}=\emptyset & \multirow{2}{*}{$D_3^\ddagger$}&
                                                 \multirow{2}{*}{$F^\times\O_j^\times\<j\>$} & D_3^\ddagger &
                                                  \multirow{3}{*}{$F^\times\O^\times$}\\
\cline{1-1}\cline{4-4}
2\varepsilon\in F^{\times 2} &  & & S_4 &  \\
\cline{1-4}
3\varepsilon\in F^{\times 2} & D_6  &F^\times\O_j^\times & D_3^\ddagger &  \\
\hline
\end{tabular}
\end{center}
It follows that when $3\mid d$ with $d>6$ and $\varepsilon\equiv -1\pmod{\grq}$
\begin{equation}
  \label{eq:161}
    t(D_3^\ddagger)=
  \begin{cases}
    2\qquad &\text{if } \{2\varepsilon, 3\varepsilon\}\cap  F^{\times 2}=\emptyset,\\
    1 \qquad &\text{otherwise,}
  \end{cases}\quad h(D_3^\ddagger)=  \begin{cases}
    3h(F)/2 \qquad &\text{if } \{2\varepsilon, 3\varepsilon\}\cap F^{\times 2}=\emptyset,\\
    h(F)/2 \qquad &\text{if } 2\varepsilon\in F^{\times 2},\\
    h(F) \qquad &\text{if } 3\varepsilon\in F^{\times 2}.\\
  \end{cases}
\end{equation}
\end{sect}

\begin{sect}\label{sect:maximal-D3II-d=1mod3}
Assume that $d\equiv 1\pmod{3}$. If $\varepsilon\equiv
1\pmod{3O_F}$, then $H$ is ramified at the two primes of $F$
above $3$, and $\scrO_3^\ddagger$ is maximal. We have 
\begin{equation}
  \label{eq:163}
      t(D_3^\ddagger)=1, \quad \text{and}\quad
      h(D_3^\ddagger)=2h(F). 
\end{equation}

Next, assume that  $\varepsilon\equiv -1\pmod{3O_F}$ so that $\omega(H)=0$. We
  define
  \begin{align}
        \O&=O_F+O_Fi+O_F\frac{1+j}{2}+O_F\frac{-3i+2j+k}{6};\label{eq:100}\\
    \O'&=O_F+O_Fi+O_F\frac{1+j}{2}+O_F\frac{-3i+2\sqrt{d}j+k}{6}.\label{eq:110}
  \end{align}
Then $\O\neq \O'$ and $j\O j^{-1}\neq \O'$. Therefore,
$\grS(\scrO_3^\ddagger)=\{\O,\O', j\O j^{-1}, j\O' j^{-1}\}$. Note that $3\varepsilon\not\in F^{\times 2}$ since
$F/\Q$ is unramified at $3$. On
the other hand, it is possible that $2\varepsilon\in F^{\times 2}$
(e.g.~$d=p$ with $p\equiv 7\pmod{12}$ or $d=2p$ with $p\equiv 11
\pmod{12}$). Suppose that this is the case. Then $\{\O^\bs,
\O'^\bs\}=\{S_4,D_3^\ddagger\}$. Indeed, we \emph{cannot} have
$\O^\bs\simeq \O'^\bs\simeq S_4$ since $\O$ and $\O'$ are not
$H^\times$-conjugate.  Write
$\varepsilon=2\vartheta^2$ and $2\vartheta=x+y\sqrt{d}\in O_F$ with
$x,y\in \frac{1}{2}\Z$ and $x\equiv y\pmod{\Z}$. Note that either $x$
or $y$ lies in $\frac{3}{2}\Z$, and 
\[\left(\frac{j}{3}\pm \frac{2\vartheta
    k}{3\varepsilon}\right)^2=-1,\qquad i\left(\frac{j}{3}\pm \frac{2\vartheta
    k}{3\varepsilon}\right)=-\left(\frac{j}{3}\pm \frac{2\vartheta
    k}{3\varepsilon}\right)i.\]
If $y\in \frac{3}{2}\Z$, then $\frac{j}{3}\pm \frac{2\vartheta
  k}{3\varepsilon}\in \O$ for a suitable choice of sign depending on
$(x\bmod{\frac{3}{2}\Z})$. Hence $\O^\bs\simeq S_4$ in this
case. Similarly, if $x\in \frac{3}{2}\Z$, then  $\frac{j}{3}\pm \frac{2\vartheta
  k}{3\varepsilon}\in \O'$ for a suitable choice of sign depending on
$(y\bmod{\frac{3}{2}\Z})$. Hence $\O'^\bs\simeq S_4$ in this
case.
The reduced
unit groups and normalizers are summarized in the following table
\begin{center}
\renewcommand{\arraystretch}{1.3}
  \begin{tabular}{*{6}{|>{$}c<{$}}|}
\hline
&&\O^\bs&\calN(\O)&\O'^\bs & \calN(\O')\\
\hline
2\varepsilon\not\in F^{\times 2}& & \multirow{2}{*}{$D_3^\ddagger$}&
                                                 \multirow{3}{*}{$F^\times\O^\times$} & D_3^\ddagger &
                                                  \multirow{3}{*}{$F^\times\O'^\times$}\\
\cline{1-2}\cline{5-5}
2\varepsilon\in F^{\times 2} &x\in \frac{3}{2}\Z &  & & S_4 &  \\
\cline{1-3}\cline{5-5}
2\varepsilon\in F^{\times 2}&y\in \frac{3}{2}\Z & S_4  & &D_3^\ddagger&  \\
\hline
\end{tabular}
\end{center}
It follows that when $d\equiv 1\pmod{3}$ and $\varepsilon\equiv -1\pmod{3O_F}$
\begin{equation}
  \label{eq:164}
   t(D_3^\ddagger)=
   \begin{cases}
2\quad&\text{if } 2\varepsilon\not\in F^{\times 2}, \\
1 \quad&\text{if } 2\varepsilon\in F^{\times 2} 
   \end{cases}\quad\text{and}\quad    h(D_3^\ddagger)=h(F)t(D_3^\ddagger). 
\end{equation}
\end{sect}

\begin{cor}\label{cor:type-D3II-d-prime-1mod3}
  Suppose that $d=p$ is a prime congruent to $7$ modulo $12$ so that
  $p\equiv 1\pmod{3}$ and $2\varepsilon\in F^{\times 2}$. Let $\O$ and
  $\O'$  be the two maximal order defined in (\ref{eq:100}) and (\ref{eq:110}). If
  $p\equiv 7\pmod{24}$, then $\O^\bs=D_3^\ddagger$, otherwise $\O'^\bs=D_3^\ddagger$. 
\end{cor}
\begin{proof}
  Write $\varepsilon=2\vartheta^2$ and $2\vartheta=x+y\sqrt{p}$ as in
  Section~\ref{sect:maximal-D3II-d=1mod3}. Since $p\equiv 3\pmod{4}$
  and $2\vartheta\in O_F$, we have $x, y\in \Z$. Clearly, 
  $\Nm_{F/\Q}(2\vartheta)^2= \Nm_{F/\Q}(2\varepsilon)=4$. On the other
  hand, $\Nm_{F/\Q}(2\vartheta)=x^2-y^2p$
  is a quadratic residue modulo $p$. It follows that
  $\Nm_{F/\Q}(2\vartheta)=2\Lsymb{2}{p}$. Reducing modulo $3$ on both
  sides of $x^2-y^2p=2\Lsymb{2}{p}$, we see that $x\equiv 0\pmod{3}$
  if and only if $\Lsymb{2}{p}=1$, i.e. $p\equiv 7\pmod{8}$. The
  corollary follows from the classification in Section~\ref{sect:maximal-D3II-d=1mod3}. 
\end{proof}
\begin{sect}
Let $d\equiv 2\pmod{3}$,  and $\varepsilon=a+b\sqrt{d}$ be as in
(\ref{eq:79}). Then either $a$ or $b$ lies in $\frac{3}{2}\Z$. We define 
\begin{gather}
\O=\scrO_3^\ddagger+O_F\delta=O_F+O_Fi+O_F\frac{1+j}{2}+O_F\delta, \qquad\text{where}\\
\delta=
\begin{cases}
\frac{-3i+2\sqrt{d}j+k}{6}\qquad &\text{if } a\equiv 1\pmod{\frac{3}{2}\Z}\text{
  and }b\equiv
0\pmod{\frac{3}{2}\Z};\\
\frac{-3i+2j+k}{6}\qquad &\text{if } a\equiv 2\pmod{\frac{3}{2}\Z}\text{
  and }b\equiv
0\pmod{\frac{3}{2}\Z};\\
\frac{-3i+2(1+\sqrt{d})j+k}{6}\qquad &\text{if } a\equiv 0\pmod{\frac{3}{2}\Z}\text{
  and }b\equiv
1\pmod{\frac{3}{2}\Z};\\
\frac{-3i+2(-1+\sqrt{d})j+k}{6}\qquad &\text{if } a\equiv 0\pmod{\frac{3}{2}\Z}\text{
  and }b\equiv
2\pmod{\frac{3}{2}\Z}.
\end{cases}
  \end{gather}
Then  $\O$ is the unique maximal order
containing $\scrO_3^\ddagger$ up to conjugation by $j$. It is possible
that $2\varepsilon\in F^{\times 2}$ (e.g.~when $d=2p$ with $p\equiv
7\pmod{12}$). On the other hand, $3\varepsilon\not\in F^{\times 2}$
since $F/\Q$ is unramified at $3$. We have 
\begin{align}
  \label{eq:80}
\O^\bs&\simeq
\begin{cases}
  D_3^\ddagger \qquad&\text{if } 2\varepsilon\not\in F^{\times 2};\\
  S_4\qquad&\text{if } 2\varepsilon\in F^{\times 2},
\end{cases}\qquad\text{and}\qquad \calN(\O)=F^\times\O^\times;\\
 t(D_3^\ddagger)&=
  \begin{cases}
 1 \qquad&\text{if } 2\varepsilon\not\in F^{\times 2};\\
 0 \qquad&\text{if } 2\varepsilon\in F^{\times 2},
  \end{cases}\qquad\text{and}\qquad
           h(D_3^\ddagger)=h(F)t(D_3^\ddagger). 
\end{align}
\end{sect}

\section{Quadratic $O_F$-orders in CM-fields}
\label{sec:quadratic-orders}
Let $F=\Q(\sqrt{d})$ be a real quadratic field with square-free
$d\geq 6$, and $K/F$ be a CM-extension of $F$. Given an $O_F$-order $B$
in $K$, we denote $[B^\times: O_F^\times]$ by  $w(B)$, and the conductor $\chi(O_K, B)$ by $\grf_B$. The class number $h(B)$ can be calculated
using the following formula  in \cite[p.~75]{vigneras:ens}:
\begin{equation}
  \label{eq:156}
h(B)=\frac{h(K)\Nm_{F/\Q}(\grf_B)}{[O_K^\times:B^\times]}\prod_{\grp\mid
  \grf_B}\left(1-\frac{1}{\Nm_{F/\Q}(\grp)}\Lsymb{K}{\grp}\right), 
\end{equation}
where $\Lsymb{K}{\grp}$ is the Artin symbol (cf.~(\ref{eq:165}) or see
\cite[p.~94]{vigneras}). For simplicity, we set $h(m)=h(\Q(\sqrt{m}))$
for any non-square $m\in \Z$.

The CM-extensions $K/F$ with $w(O_K)>1$ have been listed in
(\ref{eq:6}).  As explained in Section~\ref{sec:general-stra}, to
compute $h(H, C_n)$ it is necessarily to classify all $B$
in $K$ with $w(B)>1$. When $K=F(\sqrt{-1})$ or $F(\sqrt{-3})$, this is
carried out in \cite[Chapter~3]{vigneras:ens}, whose results are
recalled in Section~\ref{sec:orders-in-K1-K3}. When
$\Nm_{F/\Q}(\varepsilon)=1$,  the orders $B$ in $F(\sqrt{-\varepsilon})$ with $B\supseteq O_F[\sqrt{-\varepsilon}]$ are studied in 
Section~\ref{sec:orders-in-Ke}.




\subsection{Orders in $F(\sqrt{-1})$ and $F(\sqrt{-3})$}\label{sec:orders-in-K1-K3}

First, assume that $K=F(\sqrt{-1})$. Then by
\cite[Proposition~3.3]{vigneras:ens} (see also \cite{MR1544516}), we have 
\begin{equation}
  \label{eq:102}
  h(K)=\frac{1}{2}Q_{K/F}h(d)h(-d),
\end{equation}
where the Hasse unit index $Q_{K/F}=2$ if $2\varepsilon\in F^{\times 2}$,
and $Q_{K/F}=1$ otherwise. 

Write
$\grp$ for the unique dyadic prime of $F$ if $2$ is ramified in $F$ (i.e.~$d\not\equiv 1\pmod{4}$). If further $d\equiv
3\pmod{4}$, then we give the order $B$ in (\ref{eq:53}) a more
specialized notation:
\begin{equation}
  \label{eq:97}
B_{1,2}:=O_F+\grp O_K=\Z[\sqrt{-1},
\alpha_d],\qquad \text{where}\quad \alpha_d=(1+\sqrt{-1})(1+\sqrt{d})/2.
\end{equation}
The $O_F$-orders $B\subseteq O_K$ with $w(B)>1$
are summarized in Table~\ref{tab:orders-K1}.


\begin{lem}\label{lem:OF-eta-K1}
  Suppose that there exists $\vartheta\in F^\times$ such that
  $\varepsilon=2\vartheta^2$. Set
$\eta:=\vartheta(1+\sqrt{-1})\in K$ so that $\eta^2=\varepsilon\sqrt{-1}$. If
$d\equiv 2\pmod{4}$, then $O_F[\eta]=O_K$; if  
$d\equiv 3\pmod{4}$, then $O_F[\eta]=B_{1,2}$. 
\end{lem}
\begin{proof}
Suppose first that $d\equiv 2\pmod{4}$.   Then
  $2\vartheta\equiv \sqrt{d}\pmod{2O_F}$ since both sides represent the
  unique nontrivial nilpotent element in $O_F/2O_F$. Hence
  $\vartheta(1+\sqrt{-1})\equiv (\sqrt{d}+\sqrt{-d})/2
  \pmod{O_F[\sqrt{-1}]}$.  By \cite[Exercise~II.42(b),
p.~51]{MR0457396}), we have $O_F[\eta]=O_K$. When $d\equiv 3\pmod{4}$, the same proof as that of
\cite[Proposition~3.3]{xue-yang-yu:num_inv} shows that
$O_F[\eta]=B_{1,2}$. 
\end{proof}


  \begin{table}[htbp]
\renewcommand{\arraystretch}{1.5}
\centering
\caption{$O_F$-orders $B$ with $w(B)>1$ in
  $K=F(\sqrt{-1})$.} \label{tab:orders-K1}
  \begin{tabular}{|c|>{$}c<{$}|>{$}c<{$}|>{$}c<{$}|>{$}c<{$}|}
\hline
   $d$ & B  &\grf_B & w(B)&h(B)/h(d)\\
\hline
$d\equiv 1\pmod{4}$  &O_K  & O_F & 2 & \frac{1}{2}h(-d)\\
\hline
 \multirow{2}{*}{$d\equiv 2\pmod{4}$} & O_K  &O_F
                        &2Q_{K/F} &  \frac{1}{2}Q_{K/F}h(-d)\\
\cline{2-5}
& O_F[\sqrt{-1}]  & \grp & 2 &     h(-d)\\
\hline

 \multirow{3}{*}{$d\equiv 3\pmod{4}$}& O_K & O_F & 2Q_{K/F} &  \frac{1}{2}Q_{K/F}h(-d)\\
\cline{2-5}
 & B_{1,2} & \grp & 2Q_{K/F} & \frac{1}{2}Q_{K/F}h(-d)\left(2-\Lsymb{\Q(\sqrt{-d})}{2}\right)\\
\cline{2-5}

& O_F[\sqrt{-1}] & 2O_F &2 & h(-d)\left(2-\Lsymb{\Q(\sqrt{-d})}{2}\right)\\
 \hline
  \end{tabular}
  \end{table}

Next, assume that $K=F(\sqrt{-3})$. Then by
\cite[Proposition~3.4]{vigneras:ens}, we have 
  \begin{equation}
    \label{eq:104}
    h(K)=\frac{1}{2}Q_{K/F}h(d)h(-3d),
  \end{equation}
  where the Hasse index $Q_{K/F}=2$ if $3\varepsilon\in F^{\times 2}$,
  and $Q_{K/F}=1$ otherwise. In particular, if $3\nmid d$, then
  $3\varepsilon\not\in F^{\times 2}$ and $Q_{K/F}=1$, in which case $O_K$ is the only
  $O_F$-order in $K$ with nontrivial reduced unit group. Suppose next
  that $3\mid d$. Clearly,   $w(B)\mid (3Q_{K/F})$ for every 
  $O_F$-order  $B$ in $K$. The $O_F$-orders $B$ with
  $3\mid w(B)$ are summarized in Table~\ref{tab:orders-K3}. If
  $3\varepsilon\not \in F^{\times 2}$, this exhausts all orders $B$ in
  $K$ with $w(B)>1$. If $3\varepsilon\in F^{\times 2}$ and $B\subset
  K$ is an $O_F$-order with $w(B)>1$ and $3\nmid w(B)$, then
  $K=F(\sqrt{-\varepsilon})$, $w(B)=2$ and $B\supseteq
  O_F[\sqrt{-\varepsilon}]$. Such orders are classified in
  Section~\ref{sec:orders-in-Ke}.

  \begin{table}[htbp]
\renewcommand{\arraystretch}{1.5}
\centering
\caption{$O_F$-orders $B$  in
  $K=F(\sqrt{-3})$ with $3\mid w(B)$.} \label{tab:orders-K3}
  \begin{tabular}{|c|>{$}c<{$}|>{$}c<{$}|>{$}c<{$}|>{$}c<{$}|}
\hline
   $d$ & B &\grf_B & w(B)&h(B)/h(d)\\
\hline
$3\nmid d$  &O_K  &O_F& 3 & \frac{1}{2}h(-3d)\\
\hline
 \multirow{2}{*}{$3\mid d$} & O_K & O_F  &3Q_{K/F} & \frac{1}{2}Q_{K/F}h(-d/3)\\
\cline{2-5}
& O_F[\zeta_6]  & \grq=(3,\sqrt{d}) & 3 &                                                                   \frac{1}{2}\left(3-\Lsymb{\Q(\sqrt{-d/3}}{3}\right)h(-d/3)\\
\hline
  \end{tabular}
  \end{table}

\subsection{Orders in $F(\sqrt{-\varepsilon})$.} \label{sec:orders-in-Ke}
Throughout this subsection, we assume that
$\Nm_{F/\Q}(\varepsilon)=1$ and denote $F(\sqrt{-\varepsilon})$ by
$L$. By \cite[Lemma~3]{MR0441914}, there exists a pair of square-free
positive integers $\{r, s\}$ such that $rs\in \{d, 4d\}$ and
$\{r\varepsilon, s\varepsilon\}\subset  F^{\times 2}$. Hence 
\begin{equation}
  \label{eq:167}
  L=\Q(\sqrt{d}, \sqrt{-r},
\sqrt{-s})=\Q(\sqrt{d}, \sqrt{-r})=\Q(\sqrt{d},
\sqrt{-s}). 
\end{equation}
In particular, $L/F$ is either ramified at every dyadic prime of $F$
or none. According to \cite[Proposition~3.1]{vigneras:ens}, $\{r,s\}$ can be obtained from  $\{\Tr_{F/\Q}(\varepsilon)\pm 2\}$ by stripping away all the perfect square factors, and
\begin{equation}
  \label{eq:166}
  h(L)=h(d)h(-r)h(-s). 
\end{equation}

If $3\varepsilon \in F^{\times
  2}$, then $3\mid d$, and $L=F(\sqrt{-3})$. In this case,
$O_L^\times/O_F^\times=\dangle{\tilde{\eta}}\simeq \zmod{6}$, where $\eta=
\varsigma(3+\sqrt{-3})/2\in O_L^\times$ with
$3\varsigma^2=\varepsilon$ (see Table~\ref{tab:rep-elements}). We will show that $O_F[\eta]=O_L$ in
Lemma~\ref{lem:ord-6-max-ord}. Note that $F(\sqrt{-3})/F$ is unramified
at every dyadic prime of $F$. The $O_F$-orders $B$ with $3\mid w(B)$
have been covered in Table~\ref{tab:orders-K3}, so we focus on those
with $2\mid w(B)$, or equivalently $B\supseteq
O_F[\sqrt{-\varepsilon}]$. By Lemma~\ref{lem:ord-6-max-ord} below, these two
classes of orders overlap only at $O_L$.

If  $3\varepsilon \not\in F^{\times
  2}$, then $\bmu(L)=\{\pm 1\}$, and 
$O_L^\times/O_F^\times$ is a cyclic group of order $2$ generated by
the image of $\sqrt{-\varepsilon}$, so any $O_F$-order
$B\subset O_L$ with $w(B)>1$ contains
$O_F[\sqrt{-\varepsilon}]$.






Let $\grf$ be the conductor of $O_F[\sqrt{-\varepsilon}]\subseteq
O_L$. 
Then by \cite[Proposition~III.5]{Serre_local}, 
\begin{align}
\grf^2\grd_{O_L/O_F}&=\grd_{O_F[\sqrt{-\varepsilon}]/O_F}=4O_F, \label{eq:32}
\\
\grd_{O_F[\sqrt{-\varepsilon}]/\Z}&=(\grd_{O_F/\Z})^2\Nm_{F/\Q}(4O_F)=
                                    \begin{cases}
                                      2^4d^2 \qquad &\text{if }
                                      d\equiv 1\pmod{4},\\
                                      2^8d^2 \qquad &\text{otherwise.}
                                    \end{cases}
 \label{eq:31}
\end{align}
In particular, $L/F$ is unramified at all finite \textit{nondyadic}
primes of $F$. By \cite[63:3]{o-meara-quad-forms}, $L/F$ is unramified
at a dyadic prime $\grp$
if and only if $-\varepsilon$ is a square in $(O_F/4O_F)_\grp$. 
If this is the case for one (or equivalently, all) $\grp$,  then $\grd_{O_L/O_F}=O_F$,
and hence
\begin{gather}
\grf=2O_F,\qquad
O_F[\sqrt{-\varepsilon}]=O_F+2O_L,  \quad \text{and} \label{eq:64}\\
h(O_F[\sqrt{-\varepsilon}])=
\begin{cases}
4h(L)\prod_{\grp
   \mid (2O_F)}\left(1-\frac{1}{\Nm(\grp)}\Lsymb{L}{\grp}\right)\quad &\text{if }
 3\varepsilon\not\in F^{\times 2},\\[7pt]
\left(2+\Lsymb{F}{2}\right)h(d)h(-d/3)  \quad &\text{if }
 3\varepsilon\in F^{\times 2}. \\
\end{cases}
\end{gather}
When $L/F$ is ramified,  it will be shown   (in Lemmas~\ref{lem:d=1mod4ram-max},\ \ref{lem:criterion-ram-d=3mod4} and \ref{lem:criterion-ram-d=2mod4}) that
\begin{equation}\label{eq:conductor}
    \grf=\begin{cases}
      (2,\sqrt{d}) \qquad &\text{if } d\equiv 2\pmod{4},\\
      O_F \qquad &\text{otherwise.}
    \end{cases}
\end{equation}

If $O_F[\sqrt{-\varepsilon}]$ is non-maximal, then
$O_F[\sqrt{-\varepsilon}]\subseteq O_F+\grp O_L$ for every dyadic prime
of $F$. When $\Lsymb{F}{2}\neq -1$, we have $O_F/\grp\simeq \F_2$, and
hence
\begin{equation}
  \label{eq:90}
h(O_F+\grp O_L)=
\begin{cases}
\left(2-\Lsymb{L}{\grp}\right)h(L)\quad &\text{if }
3\varepsilon\not\in F^{\times 2},\\[5pt]
h(d)h(-d/3)  \quad &\text{if }
3\varepsilon\in F^{\times 2}.
\end{cases}
\end{equation}
There is a unique $O_F$-order of the form $O_F+\grp O_L$ if $2$ is ramified
in $F$, and two such orders if $2$ splits in $F$. 


\begin{lem}\label{lem:ord-6-max-ord}
  Suppose that there exists $\varsigma\in F^\times$ such that
  $\varepsilon=3\varsigma^2$. Then $O_F[\eta]=O_L$, where  $\eta=\varsigma
  (3+\sqrt{-3})/2\in O_L^\times$. 
\end{lem}
\begin{proof}
  Necessarily $3\mid d$ since $F$ is ramified at $3$. By
  Table~\ref{tab:orders-K3}, the conductor of $O_F[\zeta_6]$ coincides
  with the unique prime ideal $\grq$ of $O_F$ above $3$.  As $O_F[\eta]$ contains both
  $O_F[\sqrt{-\varepsilon}]$ and $O_F[\zeta_6]$, the conductor of
  $O_F[\eta]$ divides both $\grf=2O_F$ and $\grq$, so it
  must be $O_F$.   
 \end{proof}


 The rest of this subsection is devoted to a case-by-case proof of \eqref{eq:conductor} and working out the explicit criteria on $\varepsilon$ for $L/F$ to be ramified.



\begin{lem}\label{lem:fund-unit-d=1mod8}
Suppose that $d\equiv 1\pmod{8}$.
  Then $\varepsilon$ is of the form $a+b\sqrt{d}\in \Z[\sqrt{d}]$ with $a$ 
  odd and $b$ divisible by $4$. Moreover, if $a\equiv
 3\pmod{4}$, then 
$L/F$ is unramified at every dyadic prime of
 $F$; otherwise, $L/F$ is ramified at
  \emph{both} the dyadic primes of $F$. 
\end{lem}
\begin{proof}
  When $d\equiv 1\pmod{8}$, $O_F^\times=\Z[\sqrt{d}]^\times$ by
  \cite[Lemma~4.1]{xue-yang-yu:num_inv}. In particular,
  $\varepsilon\in \Z[\sqrt{d}]$, so we may write
  $\varepsilon=a+b\sqrt{d}$ with $a,b\in \bbN$ and $a^2-b^2d=1$. The
  first part of the lemma is obtained by taking modulo 8 on both sides
  of $a^2-b^2d=1$. The second part follows directly from
  \cite[63:3]{o-meara-quad-forms} by noting that $O_F/4O_F\simeq
  \zmod{4}\times \zmod{4}$. 
\end{proof}

\begin{lem}\label{lem:d=1mod4ram-max}
  Suppose that $d\equiv 1\pmod{4}$,  and $L/F$ is ramified at the dyadic
  primes of $F$. Then $O_L=O_F[\sqrt{-\varepsilon}]$. 
\end{lem}
\begin{proof}
Let $\{r, s\}$ be the pair of square-free positive integers in
(\ref{eq:167}).
By
\cite[Lemma~3]{MR0441914}, $rs=d$ since $d\equiv 1\pmod{4}$. Hence
  $L/F$ is ramified at the dyadic primes of $F$ if and only if $r\equiv
  1\pmod{4}$, in which case
\[\grd(O_L/\Z)=d\cdot
(-4r)\cdot(-4s)=2^4d^2=\grd(O_F[\sqrt{-\varepsilon}]/\Z)\]
by  \cite[Exercise~II.42(f), p.52]{MR0457396} and
(\ref{eq:31}). 
Therefore, $O_L=O_F[\sqrt{-\varepsilon}]$ if $L/F$ is ramified at
every dyadic prime of $F$. 
\end{proof}



Next, suppose that $d\equiv 5\pmod{8}$. It is a
classical problem of Eisenstein to characterize those $d$ such that
$\varepsilon\in \Z[\sqrt{d}]$.  If $\varepsilon\in \Z[\sqrt{d}]$, then
we write $\varepsilon=a+b\sqrt{d}$ as before, otherwise  write 
$\varepsilon=(a+b\sqrt{d})/2$ with both $a$ and $b$ odd. It has been
shown \cite{Stevenhagen, Alperin} that the number of $d$ is infinite in each
case. 

\begin{lem}\label{lem:CM-ext-fund-unit-d=5mod8}
  Suppose that $d\equiv 5\pmod{8}$.  Then $L=F(\sqrt{-\varepsilon})$
  is ramified at $\grp=2O_F$ if and only if one of the following
  conditions holds:
  \begin{enumerate}[(i)]
  \item $a\equiv
    1\pmod{4}$ and $4\mid b$ if $\varepsilon=a+b\sqrt{d}\in \Z[\sqrt{d}]$;
\item $a\equiv 3\pmod{4}$ if $\varepsilon=(a+b\sqrt{d})/2$ with
  both $a$ and $b$ odd. In this case, $d, a, b$ fall into one of the
  subcases listed below:
  \end{enumerate}
\begin{center}
\renewcommand{\arraystretch}{1.2}
  \begin{tabular}{|>{$}c<{$}|>{$}c<{$}|>{$}c<{$}|}
\hline
 d\pmod{16} & a\pmod{8} & b\pmod{8}\\
\hline
\multirow{2}{*}{$5$} & 3 & \pm 1\\
\cline{2-3}
& 7 & \pm 3\\
\hline
\multirow{2}{*}{$13$} & 7 & \pm 1\\
\cline{2-3}
&3 &\pm 3\\
\hline
  \end{tabular}
\end{center}

\end{lem}
\begin{proof}

First suppose that $\varepsilon=a+b\sqrt{d}\in \Z[\sqrt{d}]$. Reducing
both sides of $a^2-b^2d=1$ modulo $8$, we find that $4\mid b$ and
$a$ is odd. If $a\equiv 3\pmod{4}$, then $-\varepsilon\equiv
1\pmod{4O_F}$, and hence $L/F$ is
unramified at $\grp=2O_F$. On the other hand, $(O_F/4O_F)^\times\simeq
\zmod{3}\times(\zmod{2})^2$. So the nontrivial element
$-1\in (O_F/4O_F)^\times$ of order $2$ cannot be a perfect square. Thus if $a\equiv
1\pmod{4}$, then $L/F$ is ramified at $\grp$.


Next, suppose that $\varepsilon=(a+b\sqrt{d})/2$ with both $a$ and $b$
odd. Let $\{r, s\}$ be as in (\ref{eq:167}) so that $\{r\varepsilon, s\varepsilon\}\subset
F^{\times 2}$. Recall that $rs=d$ as mentioned in 
Lemma~\ref{lem:d=1mod4ram-max}. 
Write $r\varepsilon=(\frac{x+y\sqrt{d}}{2})^2$ with both $x,y$
odd.  Then $ra=\frac{x^2+y^2d}{2}\equiv 3\pmod{4}$. Therefore $r\equiv
1\pmod{4}$ if and only if $a\equiv 3\pmod{4}$. Now suppose
further that $d\equiv 5\pmod{16}$ and $b\equiv \pm 1\pmod{8}$. Taking
modulo 16 on both sides of $a^2-b^2d=4$, we get $a\equiv
\pm 3\pmod{8}$. It follows that $L/F$ is ramified at $\grp$ if and
only if $a\equiv 3\pmod{8}$ in this case. The remaining
cases are treated similarly and hence omitted. 
\end{proof}



We now study the cases where $\Lsymb{F}{2}=0$, that is,
$d\not\equiv 1\pmod{4}$. Let $\grp$ be the unique dyadic prime of $F$.

\begin{lem}\label{lem:criterion-ram-d=3mod4}
  Suppose that $d\equiv 3\pmod{4}$. Write $\varepsilon=a+b\sqrt{d}$
  with $a,b\in \bbN$. If $a$ is even, then
  $L=F(\sqrt{-\varepsilon})$ is ramified at $\grp$,  and
  $O_L=O_F[\sqrt{-\varepsilon}]$; otherwise, 
  $L/F$ is unramified at all finite
  places of $F$.
\end{lem}

\begin{proof}

Let $\{r, s\}$ be the pair of square-free positive integers as in
(\ref{eq:167}).  Clearly, $a$ and $b$ have opposite parity. 


  First, suppose that $a$ is even. By \cite[Lemma~3(b)]{MR0441914},
  $rs=4d$, and hence $L/F$ is ramified above $\grp$
  (In fact, $L/\Q$ is totally ramified above $2$). It follows from
  \cite[Exercise~II.42(f), p.~52]{MR0457396} that
  $\grd_{O_L/\Z}=(4d)\cdot 4(-2r)\cdot 4(-2s)=2^8d^2$. Comparing with
  \eqref{eq:31}, we get $O_L=O_F[\sqrt{-\varepsilon}]$.

  Next, suppose that $a$ is odd. Then $rs=d$ by
  \cite[Lemma~3(a)]{MR0441914}. Without lose of generality, we may
  assume that $r\equiv 3\pmod{4}$ so that $\Q(\sqrt{-r})$ is
  unramified at $2$. Therefore, $L=F(\sqrt{-r})$ is unramified at
  $\grp$, and thus unramified at all finite primes.  
\end{proof}

\begin{rem}
  By \cite[Theorem~1.1]{MR3157781}, we have $2\mid a$ when $d=p$ is a
  prime congruent to $3$ modulo $4$. Thus
  $O_F[\sqrt{-\varepsilon}]=O_L$ in this case (see
  \cite[Proposition~2.6]{xue-yang-yu:num_inv}).  Suppose that $d=pp'$
  is a product of primes with $p\equiv 3\pmod{4}$ and
  $p'\equiv 1\pmod{4}$. By \cite[Corollary~18.6]{Conner-Hurrelbrink},
  the 2-primary subgroup of $\Cl(O_F)$ is a nontrivial cyclic group in
  this case, and it is a cyclic group of order $2$ if either
  $\Lsymb{p}{p'}=-1$ or $\Lsymb{2}{p'}=-1$ (see \cite[Table~1]{Kim-Ryu-2012}). 
  If $\Lsymb{p}{p'}=-1$, then
  $2\mid a$. If $\Lsymb{p}{p'}=1$ and $\Lsymb{2}{p'}=-1$, then
  $2\nmid a$.  For the remaining case where $d=pp'$ with
  $\Lsymb{p}{p'}=\Lsymb{2}{p'}=1$, we merely provide a few examples to
  show its complexity: 
  \begin{center}
\renewcommand{\arraystretch}{1.3}
  \begin{tabular}{|>{$}c<{$}|>{$}c<{$}|>{$}c<{$}|>{$}c<{$}|>{$}c<{$}||>{$}c<{$}|>{$}c<{$}|>{$}c<{$}|>{$}c<{$}|>{$}c<{$}|}
\hline
    d&p&p'&\varepsilon&h(F)&d&p&p'&\varepsilon&h(F)\\ \hline
 323 & 19& 17 & 18+\sqrt{323}&4&  799 & 47 & 17& 424+15\sqrt{799}&8\\ \hline
    2419 & 59 & 41 & 2951+60\sqrt{2419}&12&  943 & 23 & 41 & 737+24\sqrt{943}  &4
    \\ 
\hline
  \end{tabular}
  \end{center}
\end{rem}
\begin{lem}\label{lem:criterion-ram-d=2mod4}
  Suppose that $d\equiv 2\pmod{4}$. Then $\varepsilon=a+b\sqrt{d}$ with
  $a$ odd and $b$ even, and $L/F$ is ramified at $\grp$ if and only
  if the following is true:
\[(a \bmod{4})\equiv 
\begin{cases}
  1 \qquad &\text{if } b\equiv 0\pmod{4},\\
  3\qquad &\text{if } b\equiv 2\pmod{4}.\\
\end{cases}
\]
If
  $L/F$ is ramified at $\grp$, then
  $O_F[\sqrt{-\varepsilon}]=O_F+\grp O_L$.
\end{lem}



\begin{proof}
  The parity of $a$ and $b$ follows directly from
  $\Nm_{F/\Q}(\varepsilon)=a^2-b^2d=1$.  Let $\{r, s\}$ be as in
  (\ref{eq:167}).  By \cite[Lemma~3(a)]{MR0441914}, $rs=d$.  Since
  $d\equiv 2 \pmod{4}$, we assume that $r$ is odd and $s$ is
  even. Then $L/F$ is
  ramified at $\grp$ if and only if $r\equiv 1\pmod{4}$.

Write
  $r\varepsilon=(x+y\sqrt{d})^2$ with $x, y\in \bbN$. Then
\[ra=x^2+dy^2\quad \text{and}\quad rb=2xy.\]
Necessarily, $x$ is odd.  If
$b\equiv 2 \pmod{4}$, then $y$ is odd, and hence $ra\equiv 3\pmod{4}$.
Thus $r\equiv 1\pmod{4}$ if and only if  $a\equiv 3\pmod{4}$. If
 $b\equiv 0\pmod{4}$, then $y$ is even, and hence
 $ra\equiv 1\pmod{4}$. Thus $r\equiv 1\pmod{4}$ if and only if  $a\equiv 1\pmod{4}$. 

If $L/F$ is unramified at $\grp$, then $O_F[\sqrt{-\varepsilon}]=O_F+2O_L$ by (\ref{eq:64}).
  Suppose that $r\equiv 1\pmod{4}$ so that  $L/F$ is ramified at
  $\grp$. Then  
  \[\sqrt{-\varepsilon}=\frac{1}{r}\sqrt{-r}\sqrt{r\varepsilon}\in
  (\Q\sqrt{-r}+\Q\sqrt{-s})\cap O_L=\Z\sqrt{-r}+\Z\sqrt{-s}.\]
  Hence $O_F[\sqrt{-\varepsilon}]\subseteq \Z[\sqrt{-r}, \sqrt{-s}]$.
  One calculates that
  $\grd_{\Z[\sqrt{-r},
    \sqrt{-s}]/\Z}=2^8d^2=\grd_{O_F[\sqrt{-\varepsilon}]/\Z}$
  by \eqref{eq:31}.  Therefore, 
  $O_F[\sqrt{-\varepsilon}]=\Z[\sqrt{-r}, \sqrt{-s}]$, which has index
  $2$ in $O_L$ by \cite[Exercise~II.42(b),
  p.~51]{MR0457396}. We conclude that  $O_F[\sqrt{-\varepsilon}]=O_F+\grp
  O_L$. 
\end{proof}
\begin{rem}
  Suppose that $d=2p$ with $p$ prime. If $p\equiv 3\pmod{4}$, then
  $\Nm_{F/\Q}(\varepsilon)=1$ by \cite[(V.1.7)]{ANT-Frohlich-Taylor}.
  Necessarily, $r=p$ in this case, so $L/F$ is unramified above $\grp$
  by Lemma~\ref{lem:criterion-ram-d=2mod4}.  If $p\equiv 5 \pmod{8}$,
  then $\Nm_{F/\Q}(\varepsilon)=-1$ by
  \cite[Proposition~19.9]{Conner-Hurrelbrink}.  The sign of
  $\Nm_{F/\Q}(\varepsilon)$ for $p\equiv 1\pmod{8}$ is more
  complicated and is discussed in \cite[Section~24]{Conner-Hurrelbrink}. The
  following table lists the first few examples of square-free $d\in 2\bbN$
  with more than 2 distinct odd prime factors 
\begin{center}
\renewcommand{\arraystretch}{1.3}
\begin{tabular}{|>{$}c<{$}|>{$}c<{$}|>{$}c<{$}||>{$}c<{$}|>{$}c<{$}|>{$}c<{$}|}
\hline
  d & \varepsilon & r\equiv 1\bmod{4}&   d & \varepsilon & r\equiv
                                                           3\bmod{4}
  \\ \hline
  30&11+2\sqrt{30}&5&42 & 13+2\sqrt{42} & 7\\ \hline
66 & 65+8\sqrt{66} & 33 & 78 & 53+6\sqrt{78} & 3\\ \hline
70 & 251+30\sqrt{70} & 5 & 102 & 101+10\sqrt{102} & 51\\
\hline
\end{tabular}
\end{center}
\end{rem}

\section{Calculations for $F=\Q(\sqrt{p})$ and $H=H_{\infty_1, \infty_2}$}
\label{sec:d=p}

\numberwithin{thmcounter}{section}
Let $p$ be a prime number, $F=\Q(\sqrt{p})$, and
$H=H_{\infty_1, \infty_2}$ be the totally definite quaternion
$F$-algebra that splits at all finite places of $F$. We calculate
$h(G)=h(H,G)$ for every finite group $G$. By
\cite[Corollary~18.4]{Conner-Hurrelbrink}, $h(\Q(\sqrt{p}))$ is 
odd for every prime $p$. According to Proposition~\ref{1.1} (see also
Remark~\ref{rem:free-act-odd-hF} and \cite[Section~3]{xue-yu:type_no}), 
\begin{equation}
  \label{eq:91}
  h(G)=h(F)t(G) \quad  \text{for all } G. 
\end{equation}
Hence it is simpler to list all $t(G)$ instead. 
 The case that $p\equiv 1\pmod{4}$ with $p>5$ has
already been treated by Hashimoto
\cite{Hashimoto-twisted-tr-formula} using another method.  We focus on the cases that
$p\in \{2, 3, 5\}$ or $p\equiv 3\pmod{4}$.

\subsection{Case $p\in \{2, 3, 5\}$}
\label{case-p-2-3-5}
Note that $N_{F/\Q}(\varepsilon)=-1$ if $p=2$ or $5$. If
$p=3$, then $N_{F/\Q}(\varepsilon)=1$, $2\varepsilon\in F^{\times 2}$,
and $F(\sqrt{-1})=F(\sqrt{-3})$. The CM-extensions $K/F$ with $[O_K^\times:
O_F^\times]>1$ are classified in
\cite[Section~2.8]{xue-yang-yu:num_inv}. For a
nontrivial element $\tilde{u}\in
O_K^\times/O_F^\times$,  we have 
\begin{itemize}
\item $\ord(\tilde{u})\in \{2, 3, 4\}$ if $p=2$;
\item $\ord(\tilde{u})\in \{2, 3, 4, 6, 12\}$ if $p=3$;
\item $\ord(\tilde{u})\in \{2, 3, 5\}$ if $p=5$. 
\end{itemize}
The first four rows of Table~\ref{tab:rep-elements} hold for $p=2,3,5$ as
well. Thus our knowledge on minimal $D_2$ or $D_3$-orders
applies here. If $\ord(\tilde{u})=5$, then $p=5$ and
$F(\tilde{u})\simeq\Q(\zeta_5)$.

Let $H'$ be an arbitrary
totally definite quaternion $F$-algebra. We claim that
$H'=H_{\infty_1, \infty_2}$ if it contains an
$O_F$-order $\calO$ with noncyclic reduced unit group $\calO^\bs$. 
Indeed, we have 
\[\calO^\bs\in
\begin{cases}
  \{D_2, D_3, D_4, A_4, S_4\}\quad &\text{if } p=2;\\
  \{D_2, D_3, D_4, A_4, S_4, D_6, D_{12}\}\quad &\text{if } p=3;\\
 \{D_2, D_3, A_4, D_5, A_5\}\quad &\text{if } p=5.\\
\end{cases}
\]
First, suppose that $\calO^\bs\supseteq D_2$. Recall that a
minimal $D_2$-order has discriminant
$4O_F$.  Thus $H'$ splits at all nondyadic primes of
$F$. Since $2$ is
either inert or ramified in $F$ and the number of ramified places of $H'$ is even, $H'$ necessarily splits at the dyadic
prime of $F$ as well.  Similarly, 
$H'=H_{\infty_1, \infty_2}$ if $\calO^\bs\supseteq D_3$. This verifies
our claim in the cases $p=2,3$.  Lastly, if $\calO^\bs\simeq D_5$ or
$A_5$, then $\calO$ contains a minimal $D_5$-order, which implies (by
the proof of Proposition~\ref{prop:uniqueness-finit-subgp}) that
$H'=\{\Q(\zeta_5), -1\}$. A direct calculation shows that
$\{\Q(\zeta_5), -1\}\simeq H_{\infty_1,\infty_2}$.

Thanks to \cite{vigneras:ens} (see also
\cite[Theorem~1.3]{xue-yang-yu:ECNF}), we have 
\begin{equation}
  \label{eq:92}
  t(H)=h(H)=
  \begin{cases}
    1 \qquad &\text{if } p=2, 5;\\
    2 \qquad &\text{if } p=3. 
  \end{cases}
\end{equation}
Using the Magma Computational Algebra System \cite{magma}, one easily
checks that
\begin{center}
\begin{tabular}{|c|c|c|c|}
\hline
  $p$ &    $2$    &      $3$       &    $5$   \\ \hline
  $t(G)$ & $t(S_4)=1$   & $t(S_4)=t(D_{12})=1$ &   $t(A_5)=1$ \\
\hline
\end{tabular}
\end{center}
This can also be obtained by hand using the \emph{mass formula} (\ref{eq:113}), which we
leave to the interested readers.





\subsection{Case $p\equiv 3\pmod{4}$ and $p>3$.}

First, we write down $t(G)$ for $G$ noncyclic. 
Since $3$ is unramified in $F$,
we have $3\varepsilon \not\in F^{\times 2}$. In particular,
$t(D_6)=0$. On the other hand,
$2\varepsilon\in F^{\times 2}$ by \cite[Lemma~3, p.~91]{MR1344833} or
\cite[Lemma~3.2(1)]{MR3157781}.  Note that $H_{\infty_1,
  \infty_2}\simeq \qalg{-1}{-1}{F}$ since $2$ is ramified in $F$. 
It follows from the results of
Section~\ref{sec:maximal-orders} that 
\begin{equation}
  \label{eq:93}
t(S_4)=t(D_4)=1, \qquad t(A_4)=t(D_2^\dagger)=0. 
\end{equation}
The unique conjugacy class of maximal orders with reduced unit group $S_4$  (resp.~$D_4$)
is represented by $\O_{24}$ in (\ref{eq:74}) (resp. $\O_8=\O_4^\dagger$ in
(\ref{eq:152})). 
Write
$\varepsilon=a+b\sqrt{p}$ with $a,b \in\bbN$. Then $a$ is even and $b$ is odd by
\cite[Theorem~1.1(1)]{MR3157781}, and hence $\beth(\scrO_2^\ddagger)=2$ by
Proposition~\ref{prop:scroII-unified-proof}. Therefore,
$t(D_2^\ddagger)=0$ by Corollary~\ref{cor:scO2typeII}. 
For $p>3$,  $H_{\infty_1, \infty_2} \simeq \qalg{-1}{-3}{F}$ if and
only if $p\equiv 2\pmod{3}$. Thus 
\begin{equation}
  \label{eq:94}
t(D_3^\dagger)=\frac{1}{2}\left(1-\Lsymb{p}{3}\right). 
\end{equation}
If $p\equiv 1\pmod{3}$, then the fact that $2\varepsilon\in F^{\times 2}$
implies that $\varepsilon\equiv -1\pmod{3O_F}$. Therefore, $\qalg{-\varepsilon}{-3}{F}\simeq H_{\infty_1, \infty_2}$ for all
$p>3$ by Lemma~\ref{lem:alg-ramified-places-D3-II}.   It follows from Corollary~\ref{cor:D3typeII} that 
\begin{equation}
  \label{eq:95}
t(D_3^\ddagger)=\beth(\scrO_3^\ddagger)-t(S_4)-t(D_6)=\frac{1}{2}\left(1+\Lsymb{p}{3}\right). 
\end{equation}
Combining (\ref{eq:94}) and (\ref{eq:95}), we obtain
\begin{equation}
  \label{eq:96}
t(D_3)=t(D_3^\dagger)+t(D_3^\ddagger)=1. 
\end{equation}
We pick $\O_6$ to be one of the maximal orders in the following table
according to the conditions on $p$ so that $\dbr{\O_6}$ is
the unique member of $\Tp(H)$ with $\O_6^\bs\simeq D_3$:
\begin{center}
\renewcommand{\arraystretch}{1.3}
\begin{tabular}{|l@{}c|lc|c|}
\hline
\multicolumn{2}{|c|}{$p$} & \multicolumn{2}{c|}{$\O_6$} & $\O_6^\bs$ \\
\hline
$ p\equiv 11$& $\pmod{12}$&     $\O_6^\dagger$ &in (\ref{eq:101})
                                                & $D_3^\dagger$\\

\hline 
$p\equiv 7$&$ \pmod{24}$& $\O$  &in (\ref{eq:100}) & $D_3^\ddagger$\\
\hline
  $p\equiv 19$&$\pmod{24}$&     $\O'$    &in (\ref{eq:110}) &
                                                            $D_3^\ddagger$\\
\hline
\end{tabular}
\end{center}
In summary, we
 have 
\begin{equation}
  \label{eq:129}
\Tp^\natural(H)=\{ \dbr{\O_{24}}, \dbr{\O_8}, \dbr{\O_6}\}.
\end{equation}




Next, we calculate $t(C_n)$ for $n\in \{2,3,4\}$. Recall that $\scrB_n$ denotes  the finite set of
CM $O_F$-orders $B$ such that
$B^\times/O_F^\times\simeq C_n$,  and $\scrB:=\cup_{n>1}\scrB_n$. By the classification in
Section~\ref{sec:quadratic-orders}, the orders in $\scrB_n$ for $n\in
\{2,3,4\}$ is listed in the following table: 
\begin{center}
\renewcommand{\arraystretch}{1.3}
  \begin{tabular}{*{4}{|>{$}c<{$}}|}
\hline 
    n & 2 & 3 & 4\\
\hline
\scrB_n&O_F[\sqrt{-1}],
\O_F[\sqrt{-\varepsilon}] &  O_{F(\sqrt{-3})} & O_{F(\sqrt{-1})},
  B_{1,2}\\
\hline
  \end{tabular}
\end{center}
Here $B_{1,2}\subset F(\sqrt{-1})$ is the order defined in
(\ref{eq:97}), and $O_F[\sqrt{-\varepsilon}]$
coincides with the ring of integers of
$F(\sqrt{-\varepsilon})=F(\sqrt{-2})$ by
Lemma~\ref{lem:criterion-ram-d=3mod4}. For each $B\in \scrB_n$, we
set  \[t(C_n, B):=\#\{\dbr{\O'}\in \Tp(H)\mid \O'^\bs \simeq C_n,
  \text{ and } \Emb(B, \O')\neq \emptyset\}. \]
It is shown in \cite[Corollary~3.5]{xue-yu:type_no} (see also Remark~\ref{rem:free-act-odd-hF}) that 
\begin{equation}\label{eq:99}
 h(C_n, B)=h(F)t(C_n, B) \qquad \forall B\in \scrB_n. 
\end{equation}

The main tool to compute $h(C_n, B)$ (and in turn $t(C_n, B)$) is
equation (\ref{eq:42}). In the current setting, we have $\grd(H)=O_F$
and hence $\omega(H)=0$. Moreover, $\calN(\O')=F^\times\O'^\times$ for
every member $\dbr{\O'}\in \Tp^\natural(H)$. This can be checked
individually for each $\dbr{\O'}\in \Tp^\natural(H)$ (see
(\ref{eq:98}), (\ref{eq:55}), Proposition~\ref{prop:type-num-D3I} and
Section~\ref{sect:maximal-D3II-d=1mod3}), or for all of
$\Tp^\natural(H)$ at once by 
 applying
 \cite[Proposition~2.8]{xue-yu:type_no}. The equation
(\ref{eq:42}) simplifies as 
\begin{equation}
  \label{eq:117}
\sum_{\dbr{\O'}\in \Tp^\natural(H)}m(B, \O',
  \O'^\times)+ 2t(C_n,
B)=\frac{h(B)}{h(F)}. 
\end{equation}

It remains to compute $m(B, \O', \O'^\times)$ for every $B\in
\scrB$
and $\dbr{\O'}\in \Tp^\natural(H)$.  We apply Proposition~\ref{prop:explicit-num-opt-embed}
to obtain the following table

\begin{table}[htbp]
\renewcommand{\arraystretch}{1.3}
\caption{Values of $m(B, \O', \O'^\times)$ for $B\in
\scrB$
and $\dbr{\O'}\in \Tp^\natural(H)$}\label{tab:num-opt-embed}
  \begin{tabular}{*{6}{|>{$}c<{$}}|}
\hline 
 & O_F[-1] & O_F[\sqrt{-\varepsilon}] & O_{F(\sqrt{-3})} & B_{1,2} & O_{F(\sqrt{-1})}\\
\hline
\O_{24}& 0 & 1 & 1 & 1 & 0 \\
\hline 
\O_8 & 1 & 1 & 0 & 0 & 1\\
\hline 
\O_6& 1-\Lsymb{p}{3} & 1+\Lsymb{p}{3} & 1 & 0 & 0 \\
\hline
\end{tabular}
\end{table}
Thanks to Proposition~\ref{prop:explicit-num-opt-embed}, it is enough
to work out $B_C$ for each conjugacy class of maximal cyclic subgroups $C$ of
$\O'^\bs$. If $\abs{C}=4$, then $B_C\simeq B_{1,2}$ for
$\O'=\O_{24}$, and $B_C\simeq O_{F(\sqrt{-1})}$ for
$\O'=\O_8$. If $\abs{C}=2$ and $\O'=\O_{24}$, then $B_C\simeq
O_F[\sqrt{-\varepsilon}]$. We present $\O_8^\bs\simeq D_4$ as in
(\ref{eq:137}) 
 with
generators $\tilde u, \tilde\eta\in \O_8^\bs$ satisfying the
conditions of (\ref{eq:138}). Then $B_{\dangle{\tilde u}}\simeq
O_F[\sqrt{-1}]$ and $B_{\dangle{\tilde u\tilde\eta}}\simeq
O_F[\sqrt{-\varepsilon}]$. Lastly,  if $\abs{C}=2$ and
$\O'=\O_6$, then $B_C\simeq O_F[\sqrt{-1}]$ if $p\equiv 2\pmod{3}$,
and $B_C\simeq O_F[\sqrt{-\varepsilon}]$ if $p\equiv 1\pmod{3}$.

Now we are ready to write down $t(C_n)$ for each $n\in \{2,3,4\}$.   For simplicity, let $h(m):=h(\Q(\sqrt{m}))$ for any non-square
  $m\in \Z$.  
Combining (\ref{eq:117}) with Table~\ref{tab:num-opt-embed}, we obtain
\begin{align}
  t(C_4,O_{F(\sqrt{-1})})&=\frac{h(F(\sqrt{-1}))}{2h(F)}-\frac{1}{2}=\frac{1}{2}(h(-p)-1);\\
  t(C_4,
  B_{1,2})&=\frac{h(B_{1,2})}{2h(F)}-\frac{1}{2}=\frac{1}{2}\left(\left(2-\Lsymb{2}{p}\right)h(-p)-1\right). 
\end{align}
Here, we use  (\ref{eq:102}) to rewrite
$h(F(\sqrt{-1}))/h(F)$ (see also
\cite[(2.16)]{xue-yang-yu:num_inv}).  Then
\begin{equation}
  t(C_4)=t(C_4,O_{F(\sqrt{-1})})+t(C_4,B_{1,2})=\left(3-\Lsymb{2}{p}\right)\frac{h(-p)}{2}-1. 
\end{equation}

Similarly, we have 
\begin{align}
  t(C_3)=&t(C_3, O_{F(\sqrt{-3})})=\frac{h(-3p)}{4}-1;\\
  &t(C_2, O_F[\sqrt{-1}])=\left(2-\Lsymb{2}{p}\right)\frac{h(-p)}{2}+\frac{1}{2}\Lsymb{p}{3}-1;\\
  &t(C_2, O_F[\sqrt{-\varepsilon}])=\frac{h(-2p)}{2}-\frac{1}{2}\Lsymb{p}{3}-\frac{3}{2};
                   \\
 t(C_2)=&t(C_2, O_F[\sqrt{-1}])+t(C_2, O_F[\sqrt{-\varepsilon}])\\=&\left(2-\Lsymb{2}{p}\right)\frac{h(-p)}{2}+\frac{h(-2p)}{2}-\frac{5}{2}.   \notag
\end{align}

By the mass formula
(\ref{eq:113}), we have for every maximal order $\O$ in $H$, 
\begin{equation}
\Mass(\O)=\frac{1}{2}\zeta_F(-1)h(F). 
\end{equation}
Note that $\zeta_F(-1)>0$ by the Siegel's formula \cite[Table~2, p.~70]{Zagier-1976-zeta}. 
Combining (\ref{eq:114}) and  (\ref{eq:91}), we obtain 
\begin{equation}
 \begin{split}
t(C_1)&=\frac{\Mass(\O)}{h(F)}-\frac{t(S_4)}{24}-\frac{t(D_4)}{8}-\frac{t(D_3)}{6}-\frac{t(C_4)}{4}-\frac{t(C_3)}{3}-\frac{t(C_2)}{2}\\
     &=\frac{\zeta_F(-1)}{2}+\left(-7+3\Lsymb{2}{p}\right)\frac{h(-p)}{8}-\frac{h(-2p)}{4}-\frac{h(-3p)}{12} +\frac{3}{2}.      
  \end{split}
\end{equation}
As $t(G)=0$ for all $G\not\in \{S_4, D_4, D_3, C_4, C_3, C_2, C_1\}$, this concludes the
computation of $t(G)$ for all $G$.

\section{Superspecial abelian surfaces}
\label{sec:Sp}

We keep the notation of the previous section. In particular, $F=
\Q(\sqrt{p})$, where $p$ is a prime number, and $H=H_{\infty_1,\infty_2}$. In this section, we
give two applications of Theorems~\ref{thm:p=1mod4typenum} and 
\ref{thm:p=3mod4typenum} to superspecial abelian surfaces. 
The first one gives for each finite group $G$ an explicit formula for
the number of certain superspecial abelian surfaces with reduced
automorphism group $G$; this one is straightforward. 
For the second application we  
construct superspecial abelian surfaces $X$ over some field $K$
of characteristic $p$ with endomorphism algebra\footnote{For an abelian
  variety $X$ over a field $k$, we write $\End_k(X)$, or simply
  $\End(X)$ if the ground field $k$ is clear, 
  for the endomorphism ring of $X$ over $k$. For any field
  extension $K/k$, we write $\End_K(X\otimes_k K)$ or simply
  $\End(X\otimes_k K)$ for the endomorphism ring of $X\otimes_k K$
  (over $K$). Some authors denote the latter by
  $\End_K(X)$. We caution the reader that the notation $\End(X)$ is
  also used for the
  endomorphism ring of $X\otimes_k \bar k$ over $\bar k$ in the
  literature, where
   $\bar k$ is an algebraic closure of $k$.}
$\End^0(X):=\End_K(X)\otimes \Q \simeq F$, 
provided that $p\not \equiv 1 \pmod {24}$. 
Recall that an
abelian variety over a field $k$ of characteristic $p$ is said to be
\emph{superspecial} if it is isomorphic to a product of supersingular
elliptic curves over an algebraic closure $\bar k$ of $k$. 

\def\Isog{{\rm Isog}}
\def\Fp{{\F_p}}

\subsection{The first application}
\label{sec:Sp.1}
Fix a Weil $p$-number $\pi=\sqrt{p}$ and a maximal order $\O$ in $H$. Let $\Isog^{O_F}(\pi)$ denote
the set of $\Fp$-isomorphism classes of simple abelian varieties $X$ over
$\Fp$ with Frobenius endomorphism $\pi_X$ satisfying $\pi_X^2=p$ and
with endomorphism ring $\End(X)\supset O_F$. Any member $X$ in 
$\Isog^{O_F}(\pi)$ is necessarily a superspecial abelian surface.
Let 
\begin{equation}
  \label{eq:Sp.1}
  \Tp(\pi):=\{ \End(X) | X\in \Isog^{O_F}(\pi) \}/\simeq.
\end{equation}
By~\cite[Theorem 6.1.2]{xue-yang-yu:ECNF}, 
we have a natural bijection $\Cl(\bbO)\simeq \Isog^{O_F}(\pi)$. 
If the superspecial class $[X]$ corresponds to the ideal class $[I]$,
then $\End(X)\simeq \calO_l(I)$. Thus, one obtains a natural bijection
$\Tp(\pi)\simeq \Tp(H)$. In particular,
\begin{equation}
  \label{eq:Sp.2}
  h(\pi):=\# \Isog^{O_F}(\pi)=h(H), \quad t(\pi):=\#\Tp^{O_F}(\pi)=t(H). 
\end{equation}
For any finite group $G$, define
\begin{equation}
  \label{eq:Sp.3}
  h(\pi,G):=\# \{ X\in \Isog^{O_F}(\pi) | {\rm RAut}(X)\simeq G \}, 
\end{equation}
where ${\rm RAut}(X):=\Aut(X)/O_F^\times$ is the reduced automorphism
group of $X$. As the above correspondence preserves the automorphism
groups, one has $h(\pi,G)=h(G)$, which is also equal to $h(F)t(G)$ by
Proposition~\ref{1.1}. 
By Theorems~\ref{thm:p=1mod4typenum} and \ref{thm:p=3mod4typenum},  
we obtain explicit formulas for $h(\pi,G)$.

\begin{prop}\label{Sp.1}
  For any finite group $G$, we have $h(\pi,G)=h(F)t(G)$, where an explicit
  formula for each $t(G)$ is given by 
  Theorems~\ref{thm:p=1mod4typenum} and \ref{thm:p=3mod4typenum}. 
\end{prop}

\subsection{Pop's result on embedding problems}
\label{sec:Sp.2}
We state a main result of F.~Pop on embedding problems for large
fields in \cite{Pop-1996}. Let $k$ be any field. Let
$\Gamma_k:=\Gal(k_s/k)$ denote the absolute Galois group of $k$, where $k_s$ is
a separable closure of $k$. An \emph{embedding problem} (EP) for $k$
is a diagram of surjective morphisms of profinite groups
$(\gamma:\Gamma_k\twoheadrightarrow A,\ \alpha:B\twoheadrightarrow A)$. An EP is said to be
\emph{finite} if the profinite group $B$ is finite 
(hence $A$ is also finite); it is
called \emph{split} if the homomorphism $\alpha$ has a section. 
We write $k_{\rm EP}$
for the fixed subfield of $\ker \gamma$. A solution of an EP is a
homomorphism of profinite groups $\beta:\Gamma_k\to B$ such that
$\alpha \beta=\gamma$; it is called a \emph{proper} solution if
$\beta$ is surjective. 

Let $K=k(t)$, where $t$ is a variable. We fix a separable closure
$K_s$ of $K$ which contains $k_s$, and let $\pi:\Gamma_K\to \Gamma_k$
be the canonical projection. To each EP=$(\gamma,\alpha)$ for $k$ one
associates an ${\rm EP}_K:=(\gamma \pi,\alpha)$ for $K$. If $\beta$ is
a solution of ${\rm EP}_K$, define
\[ K_{\beta}:=K_s^{\ker \beta}, \quad k_\beta:=K_\beta\cap k_s.\]
A (proper) \emph{regular} solution of an EP is a (proper) solution
$\beta$ of ${\rm EP}_K$ such that $k_\beta=k_{\rm EP}$. 

The regular inverse Galois problem for $k$ asks whether for a given
finite group $G$, there exists a regular finite Galois extension
$L/k(t)$ (regularity means that $L\cap k_s=k$ in a separable closure 
$k(t)_s$ of $k(t)$) with Galois
group $\Gal(L/k(t))\simeq G$. 
If $A=\{1\}$ and $B=G$ is finite, 
then a proper solution for an EP is precisely a solution for an
inverse Galois problem, and 
a proper regular solution for an EP is precisely a
solution for a regular inverse Galois problem. 

\begin{defn}\label{Sp.2}
  A field $k$ is said to be \emph{large} if for any smooth curve $C$
  over $k$, we have implication
\[ C(k)\neq \emptyset \implies \#C(k)=\infty. \]
\end{defn}

\begin{thm}[{\cite[Main Theorem A]{Pop-1996}}]\label{Sp.3}
  Assume that $k$ is large. Then every finite split EP for $k$ has
  proper regular solutions. In particular, every finite group $G$ is
  regularly realizable as a Galois group over $k(t)$.
\end{thm}

The proof in \cite{Pop-1996} also shows that 
that there are infinitely many solutions in Theorem~\ref{Sp.3}. 
 
\subsection{The second application}
\label{sec:Sp.3}

Let $X_0$ be an abelian variety over any field $K$. It is well-known
that the Galois cohomology $H^1(\Gamma_K, \Aut(X_0\otimes_K {K_s}))$
classifies all $K$-forms of $X_0/K$ up to $K$-isomorphism. Any class
in $H^1(\Gamma_K, \Aut(X_0\otimes_K {K_s}))$ is represented by a 1-cocycle
$\xi=(\xi_\sigma)\in Z^1(\Gal(L/K), \Aut(X_0\otimes_K L))$ for some
finite Galois extension $L/K$.

\begin{lem}\label{Sp.4}
  Suppose $X_\xi/K$ is the abelian variety corresponding to a 1-cocycle
  $\xi=(\xi_\sigma)\in Z^1(\Gal(L/K), \Aut(X_0\otimes_K L))$. Then 
  \begin{equation}
    \label{eq:Sp.4}
    \End(X_\xi)\simeq \{y\in \End(X_0\otimes_K L)\,  |\, \xi_\sigma
  \sigma(y) \xi_\sigma^{-1}=y, \ \forall\, \sigma\in \Gal(L/K) \}. 
  \end{equation}
\end{lem}

\begin{lem}\label{Sp.5}
  For any fixed power $q$ of $p$ and any positive integer $\ell$,
  the field 
  \begin{equation}
    \label{eq:Sp.5}
    k:=\F_{q^{\ell^\infty}}:=\bigcup_{m\ge 1} \F_{q^{\ell^m}}
  \end{equation}
  is large.
\end{lem}
\begin{proof}
  This follows immediately from the Hasse-Weil bound for the size $\#
  C(\F_{q^{\ell^m}})$ of $\F_{q^{\ell^m}}$-rational points of a curve
  $C$. 
\end{proof}

\begin{prop}\label{Sp.6}
  There exists a maximal $O_F$-order $\O$ in
  $H=H_{\infty_1,\infty_2}$ for which the unit group  $\O^\times$
  contains a finite non-abelian group if and only if $p\not\equiv 1
  \pmod {24}$. 
\end{prop}
\begin{proof}
  The reduced norm map $\Nr: \O^\times \to O_F^\times$ induces a map
  $\Nr: \O^\bs\to O_{F,+}^\times/(O_F^\times)^2$. The kernel of
  this map is $\O^1/\{\pm 1\}$, where $\O^1\subset \O^\times$ is
  the reduced norm one subgroup. Thus the index $[\O^\bs:\O^1/\{\pm
  1\}]\in \{1,2\}$, and $[\O^\bs:\O^1/\{\pm 1\}]=1$ if $N(\varepsilon)=-1$. 
  If $p\le 5$ or $p\equiv 3\pmod 4$. 
  then there is a maximal order
  $\O$ such that $\O^\bs\simeq S_4$ or $A_5$ by 
  Theorem~\ref{thm:p=3mod4typenum}. 
  Then the finite group $\O^1/\{\pm 1\}$
  must be non-abelian and hence  $\O^1$ is a non-abelian finite
  subgroup of $\O^\times$. 
  Note that  $p\le 5$ or $p\equiv 3 \pmod 4$ implies that 
  $p\not\equiv 1\pmod {24}$.
  Now assume that $p\equiv 1\pmod 4$ and $p\ge 7$. In this case 
  $N(\varepsilon)=-1$ by \cite[Corollary~18.4bis]{Conner-Hurrelbrink} and $\O^\bs=\O^1/\{\pm 1\}$. One has 
  $\O^1/\{\pm 1\}\simeq C_n$ for $n=1,2,3$ if and only if
  $\O^1=\mu_{2n}\simeq C_{2n}$. It then follows from
  Theorem~\ref{thm:p=1mod4typenum} that there exists a maximal order $\O$ with 
  $\O^1$ non-abelian if and only if
  $\Lsymb{2}{p}=-1$ or $\Lsymb{p}{3}=-1$. The latter is equivalent to $p\not\equiv 1\pmod {24}$ under
  the condition $p\equiv 1\pmod 4$. This proves the proposition.
\end{proof}
 
\begin{prop}\label{Sp.7}
  Suppose $X_0/\Fp\in \Isog^{O_F}(\pi)$ is a member such that
  $\Aut(X_0)$ contains a finite non-abelian group $G$. There exist
  a positivie integer $\ell$ and infinitely many abelian varieties $X$ over $K:=\F_{p^{\ell^\infty}}(t)$, 
  which are $K$-forms of $X_0\otimes_{\F_p} K$, such that the
  endomorphism algebra $\End^0(X)$ of $X$ is isomorphic to
  $F=\Q(\sqrt{p})$.  
\end{prop}

\def\Fpbar{\ol \Fp}
\newcommand{\isoto}{\stackrel{\sim}{\longrightarrow}}

\begin{proof}
  Let $\F_{p^n}/\Fp$ be a finite extension such that
  $\End(X_0\otimes \overline{\F}_p)=\End(X_0\otimes \F_{p^n})$, where
  we omit the subscript $\F_p$ from the scalar
  extensions of $X_0$.  Let $\ell$ be any positive integer with
  $(\ell,n)=1$ and put $k=\F_{p^{\ell^\infty}}$ and $K=k(t)$. Then
  $k\cap \F_{p^n}=\Fp$ and one has $\End(X_0\otimes k)=\End(X_0)$,
  which is a maximal order in $H$. Since $K/k$ is primary,
  $\End(X_0\otimes K)=\End(X_0\otimes k)=\End(X_0)$ by Chow's Theorem
  ~\cite{Chow-1955}; see also \cite[Theorem 3.19]{Conrad-2006}. Since
  $k$ is large (Lemma~\ref{Sp.5}), there exist infinitely many regular
  finite Galois extensions $L/K$ with Galois group $\Gal(L/K)\simeq G$
  by Theorem~\ref{Sp.3}. Consider the homomorphism
  $\xi: \Gal(L/K)\to \Aut(X_0\otimes L)$ defined by the composition
\[ \Gal(L/K)\isoto G\subset \Aut(X_0)\subset \Aut(X_0\otimes L). \]      
Since $L/K$ is regular, we have $\End(X_0\otimes L)=\End(X_0\otimes
K)=\End(X_0)$ by Chow's Theorem again. 
Thus, $\Gal(L/K)$ acts trivially on
$\Aut(X_0\otimes L)$ and hence $\xi$ is a 1-cocyle. 
Let $X/K$ be the abelian variety
corresponding to $\xi$. Then by Lemma~\ref{Sp.4}, 
$\End^0(X)$ is isomorphic to the centralizer of $G$ in $\End^0(X_0)=H$
which is the center $F=\Q(\sqrt{p})$.  
\end{proof}

\begin{cor}\label{Sp.8}
  Assume that $p\not\equiv 1\pmod {24}$. Then there is a superspecial
  abelian surface $X$ over some field $K$ of characteristic $p$ 
  such that $\End^0_K(X)\simeq \Q(\sqrt{p})$. 
\end{cor}
\begin{proof}
  This follows from Propositions~\ref{Sp.6} and \ref{Sp.7}. 
\end{proof}

\section*{Acknowledgements}

J. Xue is partially supported by the Natural Science Foundation of
China grant \#11601395. He thanks the Morningside Center of
Mathematics (CAS), Academia Sinica and NCTS for their hospitality and
great working conditions during his visits in 2017, hosted by Xu Shen
and Chia-Fu Yu, respectively. C.-F. Yu is partially supported by the
grants MoST 104-2115-M-001-001-MY3 and 107-2115-M-001-001-MY2. The
authors thank the referees for careful readings and helpful
suggestions that have improved the exposition of the paper
significantly.

\bibliographystyle{hplain}
\bibliography{TeXBiB}
\end{document}